\documentclass[reqno]{amsart}
\usepackage{mathrsfs}
\usepackage[all]{xy}

\usepackage{mathtools}
\usepackage{mathbbol}
\usepackage{wasysym}
\usepackage{amssymb}

\usepackage{MnSymbol}
\usepackage{comment}
\usepackage{enumerate}
\usepackage{xcolor}

\usepackage{ushort}

\usepackage{changepage}
\usepackage{algorithm}
\usepackage[noend]{algpseudocode}
\makeatletter
\def\BState{\State\hskip-\ALG@thistlm}
\makeatother

\usepackage{ifpdf}
\ifpdf 
  \usepackage[pdftex]{graphicx}
  \DeclareGraphicsExtensions{.pdf,.png,.jpg,.jpeg,.mps}
  \usepackage{pgf}
\else 
  \usepackage{graphicx}
  \DeclareGraphicsExtensions{.eps,.bmp}
  \usepackage{pgf}
  \usepackage{pstricks}
\fi
\usepackage{epic,bez123}
\usepackage{wrapfig}

\usepackage{soul}
\usepackage{url}
\usepackage{tikz}

\newtheorem{thm}{Theorem}[section]
\newtheorem{prop}[thm]{Proposition}
\newtheorem{cor}[thm]{Corollary}
\newtheorem{lem}[thm]{Lemma}

\newtheorem{claim}[thm]{Claim}

\newtheorem{mainthm}{}

\newtheorem{conv}[thm]{Convention}

\theoremstyle{definition}
\newtheorem{defn}[thm]{Definition}

\theoremstyle{remark}

\newtheorem{rem}[thm]{Remark}

\newtheorem*{ack}{Acknowledgments}

\newcommand{\act}{\curvearrowright}

\newcommand{\p}{\mathcal P}
\newcommand{\f}{\mathbb H}

\newcommand{\e}[1]{\omega({#1})}

\newcommand{\ax}{\mathrm{Ax}}

\newcommand{\proj}{\mathrm{Pr}}
\newcommand{\dist}{\mathbf{d}} 

\newcommand{\kmorse }{\kappa_{\mathrm{Morse}}}
\newcommand{\ktrack }{\kappa_{\mathrm{track}}}

\newcommand{\lab }{\textbf{Lab}}

\newcommand{\len }{\mathfrak l}

\newcommand{\inj }{\mbox{Inj}}

\usepackage[bookmarks=true, pdfauthor={YANG Wenyuan}]{hyperref}

\usepackage[margin=1.2in]{geometry}
\usepackage[normalem]{ulem}

\numberwithin{equation}{section}

\newcommand{\ywy}[1]{{\color{red}{#1}}} 

\begin{document}

\title[Sublinear projection tracking in acylindrically hyperbolic groups]{Sublinear projection tracking in acylindrically hyperbolic groups}

\author{Lihuang Ding}
\address{Beijing International Center for Mathematical Research,
Peking University, Beijing 100871, P.R. China}
\email{shanquan2@stu.pku.edu.cn}

\author{Wenyuan Yang}
\address{Beijing International Center for Mathematical Research,
Peking University, Beijing 100871, P.R. China}
\email{wyang@math.pku.edu.cn}

\thanks{W.Y. was supported by National Key R \& D Program of China (2025YFA1017500) and  NSFC (no. 12131009, no.12326601).}
 

\date{July 24, 2026}

\dedicatory{}


\begin{abstract} 
We study projection phenomena in word metrics of finitely generated
acylindrically hyperbolic groups.  For a loxodromic WPD element acting on a
hyperbolic space, we prove that shortest projection in the word metric to
the corresponding cyclic subgroup sublinearly tracks the pullback of
shortest projection to its axis in the hyperbolic space.  As applications, we obtain effective upper bounds for growth
functions and construct proper quotients whose growth rates converge to that
of the original group.  We further prove a growth--cogrowth inequality for
confined subgroups in both acylindrically hyperbolic groups and Morse
local-to-global groups with Morse elements.

\end{abstract}

\maketitle

\section{Introduction}
\subsection{Background and motivation}
A fundamental theme in geometric group theory is to understand the word metric
of a finitely generated group through its isometric actions on auxiliary metric
spaces.  For hyperbolic groups this relation is especially strong: if a group
acts geometrically on a hyperbolic space, then its Cayley graph is
quasi-isometric to that space, and many word-metric phenomena can be read
directly from hyperbolic geometry.  In particular, nearest-point projections to
quasi-convex subsets are coarsely preserved under quasi-isometries.

For relatively hyperbolic groups, the relation between the word metric and the
auxiliary hyperbolic space is weaker than in the geometric action setting, but
it remains well understood.  The coned-off space is typically locally
infinite and the orbit map is no longer a quasi-isometry; nevertheless,
distance formulas recover the word metric, up to controlled error, from the
relative metric together with peripheral projection data, as developed by
Dru\c{t}u--Sapir \cite{DruSapir} and Gerasimov--Potyagailo
\cite{GePo4} and made explicitly by Sisto \cite{Sisto12}.  In particular, projection tracking is essentially
additive in this setting: word-metric projections to peripheral cosets and projections
detected in the coned-off geometry differ by a uniformly bounded error.  From
another perspective, this bounded tracking property may be derived from the strong contraction of loxodromic axes in the word
metric.  This viewpoint underlies the counting applications in
\cite{YANG10,YANG11}.
 
Mapping class groups provide the motivating example of a distance formula
framework.  The Masur--Minsky formula expresses the word metric as a
thresholded sum of subsurface projection distances \cite{MinM2}.  For a
pseudo-Anosov axis, Duchin--Rafi \cite[Theorem 4.2]{DR09} proved that the
projection induced by shortest projection in the curve graph is weakly
contracting in the word metric; consequently, it tracks word metric
shortest projection with logarithmic error.  More generally,
hierarchically hyperbolic groups, introduced by
Behrstock--Hagen--Sisto \cite{BHS19}, admit analogous distance formulas
involving projections to a family of hyperbolic spaces.
Abbott--Behrstock--Durham \cite[Corollary 5.5]{ABDBR} established the
corresponding weak contraction result in this setting, characterizing
generalized loxodromic elements as contracting elements.

The main focus of this paper is the class of acylindrically hyperbolic groups
introduced by Osin \cite{Osin6}.  This broad  class exhibits
negatively curved behaviour and includes non-elementary hyperbolic and
relatively hyperbolic groups, mapping class groups of most finite-type
surfaces, $\mathrm{Out}(F_n)$ for $n\ge 2$, and many groups acting on CAT(0)
spaces. For general acylindrically hyperbolic groups, no comparable distance formula is
available.  The auxiliary hyperbolic space associated with an acylindrical action
may be very far from the Cayley graph, and the orbit map usually loses too much
information to recover the word metric additively.  The main point of this
paper is that along WPD directions, one can still recover a weaker but robust
substitute: the shortest projection in the word metric sublinearly tracks the
pullback of shortest projection to the corresponding axis in the auxiliary
hyperbolic space.

More precisely, let $G$ be a finitely generated group with finite symmetric
generating set $S$, and let $d_S$ be the associated word metric.  Suppose that
$G$ acts by isometries on a metric space $X$ with basepoint $o\in X$.  Recall
that an element $h\in G$ is a \emph{WPD isometry} if for every $r>0$ and every
$x\in X$, there exists $N>0$ such that
\begin{align}\label{WPDDefn}
\left|\{g\in G: d_X(x,gx)\le r,\;
d_X(h^N x,gh^N x)\le r\}\right|<\infty .    
\end{align}
This condition, introduced by Bestvina--Fujiwara \cite{BF2}, is a weak
properness condition along the orbit of $h$.  In an acylindrical action on a
hyperbolic space, every loxodromic element is WPD.

Let $H=\langle h\rangle$, where $h$ is a WPD element whose orbit $Ho$ is a Morse quasi-geodesic
in $X$.  There are then two natural projections from
$G$ to $H$.  The first is the shortest projection in the Cayley graph,
\[
        \proj_H^S(g)=\{h'\in H: d_S(g,h')=d_S(g,H)\},
\]
and the second is obtained by projecting in the space $X$ and pulling back to
$H$:
\[
        \proj_H^X(g)=\{h'\in H: d_X(go,h'o)=d_X(go,Ho)\}.
\]
We prove that these two projections differ by an error that is sublinear in
the word metric.  This shall provide a bridge between the geometry of the
action and the word metric that is well suited to counting applications given below.

\subsection{Sublinear projection tracking}

We first recall the sublinear contraction viewpoint on Morse elements.  A
quasi-geodesic $\gamma$ in $\mathrm{Cay}(G,S)$ is called
\textit{sublinearly contracting} if there exists a sublinear function
$\kappa$ such that
\[
        \mathrm{diam}\big(\proj_\gamma^S(g)\big)
        \le \kappa(d_S(g,\gamma))
\]
for every $g\in G$.  
Arzhantseva--Cashen--Gruber--Hume proved in \cite{ACGH} that the sublinear contraction is equivalent to the Morse
property of $\gamma$.  When $\kappa$ is bounded, this recovers the usual notion
of strong contraction.

In this paper, we need a slightly stronger form of contraction for pulled-back
projections.  Let $H<G$ be a subgroup.  We say that $\proj_H^X$ is
\textit{strongly $\kappa$-contracting in the word metric} if for all
$x,y\in G$ with $d_S(x,y)\le d_S(x,H)$,
\[
        \mathrm{diam}_S\big(\proj_H^X(x)\cup \proj_H^X(y)\big)
        \le \kappa(d_S(x,y)).
\]
This estimate is stronger than ordinary sublinear contraction, since the
right-hand side depends on the length of the perturbation $d_S(x,y)$ rather
than on the distance from $x$ to $H$; see Figure
\ref{fig:sublinearcontracting} for their comparison.

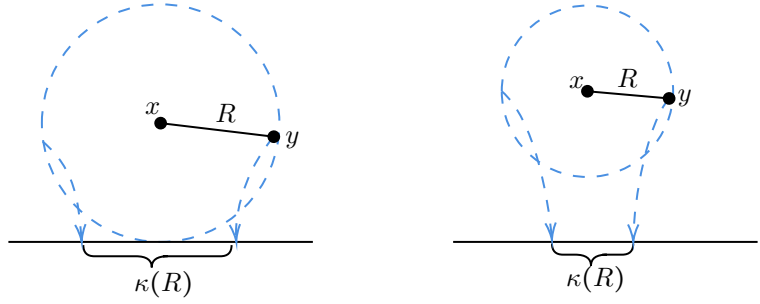
\begin{figure}
    \centering

\tikzset{every picture/.style={line width=0.75pt}} 

\begin{tikzpicture}[x=0.75pt,y=0.75pt,yscale=-0.8,xscale=0.8]

\draw  [color={rgb, 255:red, 74; green, 144; blue, 226 }  ,draw opacity=1 ][dash pattern={on 4.5pt off 4.5pt}] (105,167.5) .. controls (105,126.35) and (138.35,93) .. (179.5,93) .. controls (220.65,93) and (254,126.35) .. (254,167.5) .. controls (254,208.65) and (220.65,242) .. (179.5,242) .. controls (138.35,242) and (105,208.65) .. (105,167.5) -- cycle ;
\draw  [color={rgb, 255:red, 74; green, 144; blue, 226 }  ,draw opacity=1 ][dash pattern={on 4.5pt off 4.5pt}] (393.5,147.5) .. controls (393.5,117.81) and (417.56,93.75) .. (447.25,93.75) .. controls (476.94,93.75) and (501,117.81) .. (501,147.5) .. controls (501,177.19) and (476.94,201.25) .. (447.25,201.25) .. controls (417.56,201.25) and (393.5,177.19) .. (393.5,147.5) -- cycle ;
\draw    (84,242) -- (275,242) ;
\draw  [line width=0.75]  (131,244) .. controls (131,248.67) and (133.33,251) .. (138,251) -- (163.92,251) .. controls (170.59,251) and (173.92,253.33) .. (173.92,258) .. controls (173.92,253.33) and (177.25,251) .. (183.92,251)(180.92,251) -- (217,251) .. controls (221.67,251) and (224,248.67) .. (224,244) ;
\draw  [line width=0.75]  (425,243) .. controls (425,247.67) and (427.33,250) .. (432,250) -- (439.47,250) .. controls (446.14,250) and (449.47,252.33) .. (449.47,257) .. controls (449.47,252.33) and (452.8,250) .. (459.47,250)(456.47,250) -- (469,250) .. controls (473.67,250) and (476,247.67) .. (476,243) ;
\draw [color={rgb, 255:red, 74; green, 144; blue, 226 }  ,draw opacity=1 ] [dash pattern={on 4.5pt off 4.5pt}]  (105.5,179) .. controls (125.98,196.55) and (127.91,216) .. (129.85,240.13) ;
\draw [shift={(130,242)}, rotate = 265.43] [color={rgb, 255:red, 74; green, 144; blue, 226 }  ,draw opacity=1 ][line width=0.75]    (10.93,-3.29) .. controls (6.95,-1.4) and (3.31,-0.3) .. (0,0) .. controls (3.31,0.3) and (6.95,1.4) .. (10.93,3.29)   ;
\draw [color={rgb, 255:red, 74; green, 144; blue, 226 }  ,draw opacity=1 ] [dash pattern={on 4.5pt off 4.5pt}]  (250.5,176) .. controls (236.85,191.6) and (225.58,217.66) .. (226.88,240.27) ;
\draw [shift={(227,242)}, rotate = 265.03] [color={rgb, 255:red, 74; green, 144; blue, 226 }  ,draw opacity=1 ][line width=0.75]    (10.93,-3.29) .. controls (6.95,-1.4) and (3.31,-0.3) .. (0,0) .. controls (3.31,0.3) and (6.95,1.4) .. (10.93,3.29)   ;
\draw    (363,242) -- (554,242) ;
\draw [color={rgb, 255:red, 74; green, 144; blue, 226 }  ,draw opacity=1 ] [dash pattern={on 4.5pt off 4.5pt}]  (393.5,147.5) .. controls (413.98,165.05) and (422.57,213.5) .. (424.84,239.08) ;
\draw [shift={(425,241)}, rotate = 265.43] [color={rgb, 255:red, 74; green, 144; blue, 226 }  ,draw opacity=1 ][line width=0.75]    (10.93,-3.29) .. controls (6.95,-1.4) and (3.31,-0.3) .. (0,0) .. controls (3.31,0.3) and (6.95,1.4) .. (10.93,3.29)   ;
\draw [color={rgb, 255:red, 74; green, 144; blue, 226 }  ,draw opacity=1 ] [dash pattern={on 4.5pt off 4.5pt}]  (498.5,152) .. controls (484.85,167.6) and (476.43,214.57) .. (476.02,240.08) ;
\draw [shift={(476,242)}, rotate = 270] [color={rgb, 255:red, 74; green, 144; blue, 226 }  ,draw opacity=1 ][line width=0.75]    (10.93,-3.29) .. controls (6.95,-1.4) and (3.31,-0.3) .. (0,0) .. controls (3.31,0.3) and (6.95,1.4) .. (10.93,3.29)   ;
\draw    (179.5,167.5) -- (250.5,176) ;
\draw [shift={(250.5,176)}, rotate = 6.83] [color={rgb, 255:red, 0; green, 0; blue, 0 }  ][fill={rgb, 255:red, 0; green, 0; blue, 0 }  ][line width=0.75]      (0, 0) circle [x radius= 3.35, y radius= 3.35]   ;
\draw [shift={(179.5,167.5)}, rotate = 6.83] [color={rgb, 255:red, 0; green, 0; blue, 0 }  ][fill={rgb, 255:red, 0; green, 0; blue, 0 }  ][line width=0.75]      (0, 0) circle [x radius= 3.35, y radius= 3.35]   ;
\draw    (447.25,147.5) -- (475.82,150.01) -- (498.5,152) ;
\draw [shift={(498.5,152)}, rotate = 5.02] [color={rgb, 255:red, 0; green, 0; blue, 0 }  ][fill={rgb, 255:red, 0; green, 0; blue, 0 }  ][line width=0.75]      (0, 0) circle [x radius= 3.35, y radius= 3.35]   ;
\draw [shift={(447.25,147.5)}, rotate = 5.02] [color={rgb, 255:red, 0; green, 0; blue, 0 }  ][fill={rgb, 255:red, 0; green, 0; blue, 0 }  ][line width=0.75]      (0, 0) circle [x radius= 3.35, y radius= 3.35]   ;

\draw (161,258.4) node [anchor=north west][inner sep=0.75pt]    {$\kappa ( R)$};
\draw (432,255.4) node [anchor=north west][inner sep=0.75pt]    {$\kappa ( R)$};
\draw (168,152.4) node [anchor=north west][inner sep=0.75pt]    {$x$};
\draw (212,153.4) node [anchor=north west][inner sep=0.75pt]    {$R$};
\draw (434,136.4) node [anchor=north west][inner sep=0.75pt]    {$x$};
\draw (464,131.4) node [anchor=north west][inner sep=0.75pt]    {$R$};
\draw (502,144.4) node [anchor=north west][inner sep=0.75pt]    {$y$};
\draw (256,170.9) node [anchor=north west][inner sep=0.75pt]    {$y$};

\end{tikzpicture}
    \caption{Illustration of $\kappa$-contraction (left) and strong $\kappa$-contraction (right).}
    \label{fig:sublinearcontracting}
\end{figure}    

We say that the action $G\curvearrowright X$ has the
\textit{sublinear projection tracking property along $H$} if there exists a
sublinear function $\kappa_{\mathrm{track}}$ such that
\[
        \mathrm{diam}_S\big(\proj_H^S(g)\cup \proj_H^X(g)\big)
        \le \kappa_{\mathrm{track}}(d_S(g,H))
\]
for every $g\in G$.  In other words, the word-metric projection to $H$
sublinearly tracks the projection obtained from the action on $X$; see
Definition \ref{SublinearTrackingDefn}.

The following theorem summarizes the two main technical results of the paper. Its first conclusion is established in Lemma \ref{SublTrackInAcyl} under a weaker hypothesis, while its second conclusion is established in Lemma \ref{SSublContrAcyl}.  The notion of a
strongly $\kappa$-contracting isometry includes strongly
contracting isometries of general metric spaces and, in particular,
loxodromic isometries of hyperbolic spaces. 
\begin{mainthm}[Sublinear projection tracking]\label{SubTrackThm}
Let $G$ be a finitely generated group acting by isometries on a metric space
$X$.  Let $h\in G$ be a WPD isometry such that the orbit
$\langle h\rangle o$ is strongly $\kappa$-contracting in $X$ for some
sublinear function $\kappa$ (Definition \ref{defn:sublinear_contraction_elements}).  Then:
\begin{enumerate}
    \item the action $G\curvearrowright X$ has the sublinear projection
    tracking property along $\langle h\rangle$;

    \item the pulled-back projection $\proj_{\langle h\rangle}^X$ is strongly
    sublinearly contracting in the word metric.
\end{enumerate}
\end{mainthm}

 

The proof has two main ingredients.  The first is Sisto's geometric separation
result for WPD actions \cite{Sisto16}.  The second is a generalized Morse
lemma proved in this paper.  Roughly speaking, if a Lipschitz path joins two
points on a Morse geodesic, then it must pass near linearly many well-separated
points on any prescribed positive-density subinterval of that geodesic.  

\begin{lem}[Generalized Morse lemma]\label{GenMorseLem}
For every $\theta\in(0,1]$ and $\lambda\ge 1$, there exists a constant
$r=r(\theta,\lambda)>0$ with the following property.  Let $\gamma$ be a Morse
geodesic and let $p$ be a $\theta$-interval of $\gamma$.  Then for every
$R>4r$, every $\lambda$-Lipschitz path with the same endpoints as $\gamma$
intersects the $r$-neighborhoods of a sequence of pairwise $R$-separated
points
\[
        x_1,x_2,\ldots,x_N
\]
on $p$, where $N=\lceil \ell(p)/6R\rceil$.
\end{lem}

This lemma can be viewed as a strengthened quantitative recurrence statement
for Morse geodesics, building on earlier work of Dru\c{t}u--Mozes--Sapir
\cite{DMS} and Aougab--Durham--Taylor \cite{ADT17}.  It is also reminiscent of
Manning's bottleneck characterization of quasi-trees \cite{Man05}, but it
applies to arbitrary Morse geodesics in general metric spaces.

\subsection{Effective growth estimates}

The first application concerns growth functions in the word metric.  For a
subset $A\subseteq G$, define its \textit{relative growth rate} by
\[
        \omega(A,S)=\limsup_{n\to\infty}
        \frac{\log |B_n\cap A|}{n},
\]
where $B_n=\{g\in G:\|g\|_S\le n\}$.  For $A=G$, submultiplicativity of $|B_n|$
implies that the corresponding limit exists.

In many negatively curved groups, one expects growth to be \textit{purely exponential}: 
$$|B_n|\asymp \mathrm{e}^{n\omega(G,S)}$$
For groups with strongly contracting elements, this was proved in
\cite{YANG10}.  The sublinear projection tracking developed here gives an
effective version in the presence of Morse directions that need not be
strongly contracting.

\begin{mainthm}[Effective growth estimate]\label{GrowthBallThm}
Let $G$ be an acylindrically hyperbolic group, and let $K<G$ be a
non-virtually cyclic subgroup containing a generalized loxodromic element.
Then for every finite generating set $S$ of $G$, there exists a sublinear
function $\kappa=\kappa(K,S)$ such that
\[
|B_n\cap K|\le \mathrm{e}^{n \omega(K,S)+\kappa(n)}
\]
for every $n\ge 1$.  In particular,
\[
       \mathrm{e}^{n \omega(G,S)}\le |B_n|\le \mathrm{e}^{n \omega(G,S)+\kappa(n)}
\]
\end{mainthm}
As a consequence, the growth rate $\omega(K,S)$ is an actual limit for
such subgroups $K$, recovering a result of Schesler \cite{Sch22} in this
setting.
\begin{rem}
We note that $\kappa=O(\kmorse)$, where $\kmorse$ is the sublinear function
associated with a set of three independent generalized loxodromic elements in
$K$.  When $G$ contains strongly contracting elements, so that
$\kmorse\equiv C$, this recovers the result \cite{YANG10}.
\end{rem}

For mapping class groups, the estimate becomes polynomially sharp.  Indeed, a
theorem of Duchin--Rafi \cite[Theorem 4.2]{DR09} shows that pseudo-Anosov
elements are logarithmically contracting (by Corollary \ref{cor:log_tracking}).  Hence the preceding theorem gives
the following consequence.

\begin{cor}\label{ModGrowthThm}
Let $G=\mathrm{Mod}(\Sigma_g)$ with $g\ge 2$, and let $S$ be a finite
generating set of $G$.  Then there exists a constant $d>0$ such that
\[
        \mathrm{e}^{n \omega(G,S)}\le |B_n| \le n^d \mathrm{e}^{n \omega(G,S)}
\]
for every $n\ge 1$.
\end{cor}

In general, a weak contracting element has logarithmic contraction 
(see \cite[Corollary 5.7]{RV21}, also Corollary \ref{cor:log_tracking}). See Remark \ref{rem:contraction_terms} for the
relations among the various notions of contraction used here.
Abbott--Behrstock--Durham \cite[Corollary 5.5]{ABDBR} characterized the weakly
contracting elements of hierarchically hyperbolic groups as precisely their
generalized loxodromic elements.  Consequently, the same type of ball-growth
estimate applies to these groups.

It remains open whether a mapping class group admits a finite generating set
$S$ for which some element is strongly contracting in
$\mathrm{Cay}(G,S)$ and, more strongly, whether this property holds for every
finite generating set. For braid groups modulo their centers, Calvez--Wiest \cite{CW24} proved that
pseudo-Anosov axes are strongly contracting with respect to the classical and
dual Garside generating sets.  By contrast, Rafi--Verberne \cite{RV21}
constructed a generating set and a pseudo-Anosov element whose axis is not
strongly contracting, illustrating the dependence on the generating set. If strong contraction fails for a particular
generating set $S$, it is natural to ask whether the mapping class group
nevertheless has purely exponential growth with respect to $S$.  The estimate
above differs from pure exponential growth only by a polynomial factor, so
resolving this question would determine whether that factor reflects a genuine
geometric phenomenon or merely a limitation of the present method.

Relatedly, it remains an interesting question whether mapping class groups
admit geodesic automatic combings with respect to suitable generating sets.
If such a combing exists, the polynomial correction above should be encoded in
the combinatorial or spectral structure of the corresponding finite-state
automaton; compare Calegari \cite[Lemma 3.4.2]{Calegari}. 

\subsection{Quotients with asymptotically full growth}

The second application concerns the growth rates of quotient groups.  Growth
tightness asks whether a proper quotient of a group must have a strictly
smaller growth rate.  In a subsequent paper \cite{DY26B}, we shall establish the growth tightness for every acylindrically hyperbolic groups.  Here we prove the complementary result (without assuming growth tightness) that no
uniform gap exists: one can find proper quotients whose growth rates approach
that of the original group.

\begin{mainthm}[No uniform growth gap for quotients]\label{NoGapThm}
Let $G$ be a non-elementary acylindrically hyperbolic group.  Then for every
finite generating set $S$ of $G$, there exists a sequence of quotient groups
$G_n$ of $G$ with infinite kernels such that
\[
        \omega(G_n,\bar S_n)\to \omega(G,S),
\]
where $\bar S_n$ is the image of $S$ in $G_n$.
\end{mainthm}
 
The proof follows the strategy developed in \cite{YANG7,HLY}, which uses
rotating family theory for WPD elements to construct small-cancellation
quotients.  First, one constructs free semigroups whose growth rates converge
to $\omega(G,S)$.  This   is also the crucial ingredient in the
effective growth estimate in \ref{GrowthBallThm}.  One then takes a quotient
by sufficiently large powers of an independent WPD element.  The
large-intersection estimates obtained from sublinear projection tracking
ensure that these free semigroups embed in the resulting quotients, forcing
their growth rates to remain arbitrarily close to that of $G$.

\subsection{Growth-cogrowth inequality}
Suppose that a non-elementary group $G$ is either acylindrically hyperbolic
  or contains a Morse element and satisfies the Morse local-global
assumption introduced by Russell–Spriano–Tran \cite{RST21}. 
The final application is a growth--cogrowth inequality for confined subgroups. This part uses the same sublinear extension philosophy in both settings.  

Recall that a subgroup $H<G$ is \textit{confined} if there exists a
finite subset $P\subset G\setminus\{1\}$ such that every conjugate of $H$
intersects $P$. Such a set $P$ is called a confining subset.  In the presence of Morse
elements, Lemma \ref{lem:E(G)} provides a canonical maximal finite normal
subgroup $E(G)$.  We call a confining subset \textit{non-degenerate} if it is disjoint
from $E(G)$. We refer to \cite{CGYZ} for more relevant discussion and references therein concerning confined subgroups.

\begin{mainthm}[Growth-cogrowth inequality]\label{GrowthCoGrowthFormula}
Suppose that a non-elementary group $G$ is either acylindrically hyperbolic or
contains a Morse element and satisfies the Morse local-global assumption.
Then for every confined subgroup
$H<G$ with a non-degenerate confining subset,
\[
        \omega(H,S)+\frac{\omega(G/H,\bar S)}{2}\ge \omega(G,S),
\]
where $\omega(G/H,\bar S)$ is the growth rate of the  Schreier graph $\mathrm{Cay}(G,S)/H$ in the combinatorial metric.
\end{mainthm}
In the literature, particularly for normal subgroups \(H\) of free groups \(G\), the relative growth rate \(\omega(H,S)\) is sometimes called the \textit{cogrowth} associated with the quotient \(G/H\), whereas \(\omega(G/H,\bar S)\) is its ordinary growth rate. This terminology explains the title of this section.

Abbott--Zbinden \cite{AZ26} recently constructed a finitely generated,
non-virtually cyclic Morse local-to-global group containing 
Morse elements that is not acylindrically hyperbolic.  See
Remark \ref{rem:AHMorse} for a comparison of these two classes, which may
overlap, although neither is contained in the other.

Since infinite normal subgroups must be confined, the following corollary is immediate.   

\begin{cor}\label{cor:GrowthCoGrowthFormula}
Under the assumption of \ref{GrowthCoGrowthFormula}, for every infinite normal subgroup
$H<G$,
\[
        \omega(H,S)+\frac{\omega(G/H,\bar S)}{2}\ge \omega(G,S),
\]
where $\omega(G/H,\bar S)$ is the growth rate of the  Schreier graph $\mathrm{Cay}(G,S)/H$ in the combinatorial metric.
\end{cor}

\begin{rem}
The inequality was first proved by Jaerisch--Matsuzaki \cite{JM17} for normal
subgroups of free groups.  It was subsequently established by Coulon
\cite{Coulon22} for normal subgroups of groups with strongly contracting
elements, and more recently by Choi--Gekhtman--Yang \cite{CGYZ} for confined
subgroups in the same class of groups.  All the ambient groups considered in
these works are acylindrically hyperbolic, but no strongly contracting elements exist in general.
\end{rem}

\begin{rem}
The inequality has two immediate consequences.  First, if
$\omega(G/H,\bar S)<\omega(G,S),$
then
$\omega(H,S)>\frac{\omega(G,S)}{2}.$
Thus, growth tightness of the quotient implies cogrowth tightness of the
subgroup; in this sense, cogrowth tightness may be reduced to growth
tightness; see \cite{DY26B}.

Second, suppose that there exists a sequence of normal subgroups $(H_n)$ such
that
\[
\omega(H_n,S)\longrightarrow\frac{\omega(G,S)}{2}.
\]
Since quotient growth cannot exceed ambient growth, the growth--cogrowth
inequality implies that
\[
\omega(G/H_n,\bar S)\longrightarrow\omega(G,S).
\]
Thus, asymptotic sharpness of the cogrowth lower bound forces the absence of a
uniform positive growth gap between the quotients and the original group.  The
latter   is established independently in  
\ref{NoGapThm}.  Although the existence of a sequence satisfying the above
cogrowth convergence remains conjectural in this generality, any such sequence
would also demonstrate the asymptotic sharpness of the growth--cogrowth
inequality.
\end{rem}

\begin{rem}
The growth--cogrowth inequality is a negative-curvature phenomenon: the presence of Morse
elements is essential.  Indeed, for
\(G=F_2\times F_3\) and \(H=F_2\times\{1\}\), the split generating set
\[
S=(S_2\times\{1\})\cup(\{1\}\times S_3)
\]
satisfies the inequality, since
\[
\omega(G,S)=\log 5,\qquad \omega(H,S)=\log 3,\qquad
\omega(G/H,\bar S)=\log 5.
\]
On the other hand, for the product generating set
\[
S'=(S_2\cup\{1\})\times(S_3\cup\{1\})\setminus\{(1,1)\},
\]
one has
\[
\omega(G,S')=\log 15,\qquad \omega(H,S')=\log 3,\qquad
\omega(G/H,\bar S')=\log 5,
\]
so the inequality fails.   
\end{rem}

\begin{cor}
Under the assumption of \ref{GrowthCoGrowthFormula}, if the Schreier graph $\mathrm{Cay}(G,S)/H$ has subexponential growth, then
\[
        \omega(H,S)=\omega(G,S).
\]
Thus confined subgroups with small cogrowth must have full exponential growth
inside $G$.    
\end{cor}

It remains open whether the equality
$\omega(H,S)=\omega(G,S)$ characterizes the amenability of $G/H$.

\subsection{Strategy of proof}

We briefly explain the common mechanism underlying the applications.  In the
classical strongly contracting setting, Yang's extension lemma \cite{YANG10}
allows one to insert a connector from a fixed finite set to obtain injective
or finite-to-one maps between large subsets of the group.  The inserted
connector forces the resulting path to fellow travel a contracting axis up to
a uniformly bounded error.

In the present setting, the available contraction is only sublinear, so a
bounded connector is no longer sufficient.  Instead, we use a sublinear
extension procedure in which the length of each connector is controlled by a
sublinear function of the adjacent word lengths.  More precisely, given a
tuple
\[
(a_1,a_2,\ldots,a_{m+1}),
\]
we choose connectors $f_i$, for $1\le i\le m$, satisfying
\[
\|f_i\|_S
\asymp
\kappa\bigl(\max\{\|a_i\|_S,\|a_{i+1}\|_S\}\bigr),
\]
where $\kappa$ is sublinear.  These connectors remain negligible at the
exponential growth scale while being sufficiently long to force large
intersections with WPD axes.  See Sections
\ref{SSec:sublinear_extension} and \ref{SSublinearPersists} for details.

The sublinear projection tracking theorem then transfers large-intersection
information from the auxiliary hyperbolic space to word geodesics in the
Cayley graph.  This provides the injectivity and multiplicity bounds required
for the growth estimates, the quotient construction, and the
growth--cogrowth inequality.

\subsection{Organization of the paper}

Section~\ref{SecPrelim} collects background on Morse quasi-geodesics,
sublinear contraction, independent Morse elements, and the sublinear extension
lemma.  Section~\ref{sec:SPTProperAction} proves sublinear projection tracking
for proper actions, which serves as a model case.  Section~\ref{SecSublTrackAcyl}
proves the WPD version, including the generalized Morse lemma
\ref{GenMorseLem}.  Its final subsection establishes the large-intersection
transfer principle between geodesics in the auxiliary space and word geodesics
in the Cayley graph.  Section~\ref{SecEffEst} proves the effective growth
estimate in~\ref{GrowthBallThm}.  Section~\ref{SecNoGrowthGapQuotients}
proves~\ref{NoGapThm} on quotients with asymptotically full growth.
The final section proves the growth-cogrowth inequality in~\ref{GrowthCoGrowthFormula}.

\ack
A preliminary version of this paper was circulated in November 2024 and
contained early versions of \ref{SubTrackThm} and \ref{GrowthBallThm}.  We are
grateful to Denis Osin, Alex Sisto and Abdul Zalloum for helpful
conversations, and especially to Inhyeok Choi for identifying a serious gap in
certain statements, which were consequently removed.  W.Y. owes special thanks to Huabin Ge for his support and sustained interest in this work.    

\section{Preliminaries}\label{SecPrelim}

\subsection{Convention and basic notions}\label{SSubConvention}
A \textit{sublinear}  function $\kappa:\mathbb R_+\to \mathbb R_+$ means a monotone non-decreasing and positive function so that 
$$
\frac{\kappa(r)}{r}\to 0\quad\text{as}\quad r\to\infty.
$$
Following \cite{QRT22}, up to raising to a bounded multiple of $\kappa$  we may assume that $\kappa$ is concave and so has the following  property:  
\begin{equation} \label{Eq:Concave}
\begin{aligned}
\forall a, b>0, t\in [0,1]:&\quad t\kappa(a)+(1-t)\kappa(b)\le \kappa(ta+(1-t)b)\\
\forall a, t>1:& \quad \kappa(ta)\le t\kappa(a)    
\end{aligned}
\end{equation}
 

A map $\pi: (X,d_X)\to (Y,d_Y)$ between two metric spaces is called \textit{$\lambda$-quasi-isometric embedding} for some $\lambda>1$ if 
$$
\forall x,y\in X:\quad \lambda^{-1} d_X(x,y)-\lambda\le d_Y(\pi(x),\pi(y))\le \lambda d_X(x,y)+\lambda
$$
If, in addition, $Y$ is contained in the $D$-neighborhood  $\mathcal N_D(\pi(X))$ of  $\pi(X)$  for some $D>0$, then $\pi$ is called \textit{$\lambda$-quasi-isometry} and $X$ is  called \textit{quasi-isometric} to $Y$. 

A path $p$ in a metric space $(X,d_X)$ is  a continuous map $p:I\subseteq \mathbb R\to X$ from an interval  $I$ to $X$. If $I=[a,b]$ for $a,b\in \mathbb R$, we denote by $p^-=p(a)$ and $p^+=p(b)$ the initial and terminal endpoints. The length of $p$ is denoted by $\ell(p)$. A subpath means the restriction $p\lvert_J$ of $p$ to a sub-interval $J\subseteq I$. In practice, we often refer to the image of $p$ as a path and pick up two points $x,y$ (with  parameters $x=p(s), y=p(t)$ in context) on $p$ to specify the subpath, denoted by $[x,y]_p$. 

A path $p$ is called \textit{$\lambda$-quasi-geodesic} for some $\lambda\ge 1$ if for any subpath $q$ with finite $J\subseteq I$, $$\ell(q)\le \lambda d_X(q^-,q^+)+\lambda$$ If $\ell(p)\le \lambda d_X(p^-,p^+)$, then $p$ is called   \textit{$\lambda$-Lipschitz path}.

Let  $p$ be a geodesic {segment}. Given  $0\le \theta_1\le \theta_2\le 1$, the \textit{$(\theta_1,\theta_2)$-interval} of $p$ denoted by $p(\theta_1,\theta_2)$ means the closed subsegment of $p$ so that the  initial endpoint  has a   distance $\theta_1  n$ to $p^-$ and its terminal endpoint has a distance $\theta_2 n$ to $p^-$ where $n:=\ell(p)$.  

By abuse of language, if the values $\theta_1,\theta_2$ are not important, we often call $(\theta_1,\theta_2)$-interval by  \textit{$\theta$-interval} where $\theta:=\theta_2-\theta_1$.

\subsubsection*{\textbf{Cayley graphs}}
Let $G$ be a
group generated by a subset $S$. We assume $S=S^{-1}$ and $1\notin S$. For $g \in G$ let $\|g\|_S$ be the length of a shortest
word in the alphabet set $S$ representing $g$. This defines the word metric $d_S(a,b)=\|a^{-1}b\|_S$. Alternatively, it is the combinatorial metric on the Cayley graph of $G$ with respect to $S$, denoted by $\mathrm{Cay}(G, S)$. A combinatorial (oriented) path $p$ in the Cayley graph  is naturally labeled by  a formal word denoted by $\lab(p)$ over the alphabet $S$.

Let $S_n=\{g\in G: \|g\|_S=n\}$ and $B_{n}=\{g\in G: \|g\|_S\le n\}$. Fix $\Delta\ge 1$ and $x\in G$. Define the annulus-like set of radius $n$ with width $\Delta$
\begin{equation}\label{AnnulusEQ}
S_n(x, \Delta)=\{g\in G: |d_S(x, g)-n|\le\Delta\}
\end{equation}
For simplicity, write $S_n(\Delta)=S_n(1,\Delta)$, and $S_n(x)=S_n(x,0)$ for the $n$-th sphere at $x$.

Assume that $G$ acts by isometry on a metric space $X$. Throughout the paper, we choose a basepoint  $o\in X$ chosen once for all and the orbital map
$$
\pi: G \to Go
$$
extends naturally to the 1-skeleton of  $\mathrm{Cay}(G,S)$ by sending edges equivariantly to geodesic segments in $X$. In this way, every path $p$ is sent to a piecewise geodesic path $\pi(p)$ with the same label $\lab(p)$ as $p$. Namely, if $\lab(p)=s_1s_2\cdots s_n$ and $p$ starts with the element $g\in G$, we have 
$$
\pi(p)=(g[o,s_1o])\cdot (s_1[o,s_2o])\cdots (s_1\cdots s_{n-1}[o,s_no])$$
By abusing language, we say that the path $\pi(p)$ is labeled by $p$. 
Setting $\lambda=\max\{d_X(o,so):s\in S\}$, $\pi$ is a $\lambda$-Lipschitz map.

\subsubsection*{\textbf{Shortest projections and pullbacks}} We first introduce a general notion of projection maps.
\begin{defn}
Given a subset $A\subseteq X$, a set-valued map $\proj_A: X\to 2^A$ is called \textit{projection map} if there exists an increasing map $\hat r:\mathbf R_+\to \mathbf R_+$  so that  for any $x\in X$ with $d(x, A)\le r$, $\proj_A(x)\subseteq B(x,\hat r)$ where $\hat r=\hat r(r)$.      
\end{defn}

Most of the paper considers the shortest projections and their pullbacks via  Lipschitz maps. 
For given $x\in X$, define   $$\proj_A^X(x):=\{a\in A: d_X(x,a)\le d_X(x,A)+1\}$$  If a point $x$ is $r$-close to $A$, then $\mathrm{diam}_X(\proj_A(x))\le 2r+2$, so $\proj_A^X$ is a projection map.   
In general, if a path $\alpha$ is contained in $\mathcal N_r(\gamma)$, then  
\begin{align}\label{eq:proj_in_nbld}
\mathrm{diam}_X(\proj_A^X(\alpha))\le \ell(\alpha)+2r+2    
\end{align}
Pullbacks of shortest projections via Lipschitz maps provide another important class of projection maps. Let $\pi:X\to Y$ be a Lipschitz map. 
For $A\subset X$, let  $\proj_{\pi A}^Y: Y\to 2^{\pi(A)}$ be the shortest $d_Y$-projection map. Define  
$$
\begin{aligned}
\proj^Y_A(x):&=\pi^{-1}\proj_{\pi A}^Y(\pi(x))\\
&=\{y\in A: d_Y(\pi(x),\pi(y))\le d_Y(\pi(x),\pi(A))\}.    
\end{aligned}
$$
Since $\pi:X\to Y$ is a Lipschitz map, we see that $\proj^Y_A(x)$ is a projection map. 

Define the projection distances for which the triangle inequality holds $$\dist_A^X(x,y)=\mathrm{diam}_X\big(\proj^X_A(x)\cup \proj^X_A(y)\big),\quad \dist_A^Y(x,y)=\mathrm{diam}_X\big(\proj^Y_A(x)\cup \proj^Y_A(y)\big)$$


\subsubsection*{\textbf{Sublinearly contracting isometry}}
Let  $\proj_A: X\to 2^A$ be a (not-necessarily shortest) projection map. We introduce two notions of sublinear contraction. 

\begin{defn}\label{SublinearContrDefn}
Let $\kappa$ be a sublinear function.
We say that  $\proj_A$ has \textit{$\kappa$-contraction} (resp. \textit{strongly $\kappa$-contraction}) if for any $x, y\in X$ and $d_X(x,y)\le d_X(x,A)$, we have 
$$\mathrm{diam}_X\big(\proj_A(x)\cup \proj_A(y)\big)\le \kappa(d_X(x,A))\quad \left(\text{resp. } \quad \mathrm{diam}_X\big(\proj_A(x)\cup \proj_A(y)\big)\le \kappa(d_X(x,y))\right)$$
See Fig. \ref{fig:sublinearcontracting} for an illustration of their difference.
\end{defn}
\begin{rem}\label{rem:contraction_terms}
Let us illustrate the relation of the sublinear contractions with the existing ones. Assume $\kappa\equiv C$. If  $\proj_A$ is the shortest projection, then the (strong) $\kappa$-contraction is the usual strong contraction studied in literature \cite{YANG10,YANG11}. If  $\proj_A$ is a general projection map, the $\kappa$-contraction is usually called weak contraction  in  \cite{Sisto,ABDBR}. If $\kappa$ is a general sublinear function, according to Lemma \ref{CharMorseLem} below, the $\kappa$-contraction of the shortest projection characterizes the Morse quasi-geodesic.
\end{rem}

In practice,  the following intermediate property is used to verify the strongly sublinear contraction (eg. Lemma \ref{SSublContrInProper} and Lemma \ref{SSublContrAcyl}). 
\begin{lem}\label{lem:reduction}
Assume that there is a sublinear function \(\kappa\) such
that, whenever \(x,y\in X\) satisfy $d_X([x,y]_X,A)\ge d_X(x,y),$
one has
\[
        \mathrm{diam}_X\big(\proj_A(x)\cup \proj_A(y)\big)\le \kappa(d_X(x,y)).
\]
Then \(\proj_A\) has strongly sublinear contraction.
\end{lem}

\begin{proof}
Let \(x,y\in X\) satisfy \(d_X(x,y)\le d_X(x,A)\), and put
\(n=d_X(x,y)\). Choose points $x=x_0,x_1,\ldots,x_m$
on \([x,y]_X\) so that $d_X(x_i,y)=n/2^i,$
where \(m=m(n)\) is the smallest integer with \(n/2^m<C\), for a fixed
constant \(C>0\). The hypothesis gives
\[
        \mathrm{diam}_X\big(\proj_A(x_{i-1})\cup \proj_A(x_i)\big)\le \kappa(n/2^i).
\]
For the remaining segment $[x_m,y]_X$, either the hypothesis applies or this
segment meets the $C$-neighborhood of $A$ to which may apply the definition of the projection map. In either case, there exists a constant $D=D(C)$ so that $\mathrm{diam}_X\big(\proj_A(x_{m})\cup \proj_A(y)\big)\le D$. By the triangle
inequality for projection distances,
\begin{align}\label{eq:kappa_prime}
        \mathrm{diam}_X\big(\proj_A(x)\cup \proj_A(y)\big)
        \le
        \sum_{i=1}^{m}\kappa(n/2^i)+D=:\kappa'(n).   
\end{align}
 
Since \(\kappa\) is sublinear, \(\kappa'\) is also sublinear and hence \(\proj_A\) is strongly sublinearly contracting.
\end{proof}
\begin{cor}\label{cor:log_tracking}
Assume that \(\proj_A\) has $C$-contraction for some $C>0$. Then \(\proj_A\) logarithmically tracks the shortest projection \(\proj_A^\pi\); namely, there exists a large constant $c>1$ so that
$$
\forall x\in X,\quad \mathrm{diam}\big(\proj_A(x)\cup \proj_A^\pi(x)\big) \le c\,\log\big(d_X(x,A)\big)
$$
\end{cor}
\begin{proof}
It suffices to repeat the proof of Lemma \ref{lem:reduction} by assuming that $y$ is a shortest projection point of $x$ to $A$. The inequality (\ref{eq:kappa_prime}) is exactly what we wanted, where $\kappa'(n)$ is bounded above by $c\log(n)$ for some constant $c>1$.
\end{proof}
In this context, it is natural to introduce the corresponding notions for
isometries.

\begin{defn}\label{defn:sublinear_contraction_elements}
An isometry $h$ of $X$ is called \textit{sublinearly contracting}
(respectively, \textit{strongly sublinearly contracting}) if the orbit
$A=\langle h\rangle o$ is a quasi-geodesic in $X$ and the shortest
projection $\proj_A^X$ with respect to $d_X$ is $\kappa$-contracting
(respectively, strongly $\kappa$-contracting) for some sublinear function
$\kappa$.
\end{defn}

\begin{rem}
We emphasize that these contraction properties are defined using
shortest projection with respect to the metric $d_X$. Since sublinear
contraction of shortest projection characterizes the Morse property,
a sublinearly contracting isometry is precisely a Morse isometry, namely,
an isometry whose orbit is a Morse quasi-geodesic.
\end{rem}

\subsection{Sublinear projection geometry  of Morse quasi-geodesics}\label{SSubMorseIsometry}
We study Morse quasi-geodesics via the sublinearly contracting property of shortest  projections. Most of materials are adapted from \cite{ACGH,DMS, ADT17} to our needs of the counting purpose. 

\begin{defn}
Let $\mu:\mathbb R_{+}\to \mathbb R_{+}$.
A subset $\gamma$ in $X$ is called \textit{$\mu$-Morse} if for every $c\ge 1$, any $c$-quasi-geodesic with two endpoints on $\gamma$ lies in the $\mu(c)$-neighborhood of  $\gamma$.     
\end{defn}
The following  elementary fact will be  implicitly used later on.
\begin{lem}\label{MorseSubpath}\cite[Lemma 3.1]{Liu21}\cite[Lemma 2.4]{RST21}
Let $\gamma$ be a $\mu$-Morse $c$-quasi-geodesic for some $\mu$ and $c>0$. Then there exists $\mu'$ depending only on $\mu$ and $c$ with the following property. Any subpath of $\gamma$ is also $\mu'$-Morse, and  $\gamma$ contains a $\mu'$-Morse geodesic  in its finite neighborhood.
\end{lem}

We shall actually use  the other two equivalent formulations of Morse quasi-geodesics, the first of which is in terms of middle recurrence. This notion  was  introduced in \cite{DMS} and further clarified in  \cite{ADT17}. We shall   improve the middle recurrence in Section \ref{SecFracMorseLem}.

Let $\gamma$ be a finite path with initial and terminal points $\gamma^-,\gamma^+$. Given some $\theta\in (0,1/2)$, the \textit{$\theta$-middle} denoted by $\gamma^\theta$  is defined to be the following   subset on $\gamma$:
\begin{align}\label{defn:theta-middle}
\gamma^\theta \,:=\,\Bigl\{x\in \gamma: \min\{d(x,\gamma^-),  d_X(x,\gamma^+)\} \ge \theta d_X(\gamma^-,\gamma^+)\Bigr\}    
\end{align}
Note that if $\gamma$ is a geodesic segment, then the $\theta$-middle is exactly the segment $\gamma(\theta, 1-\theta)$.

\begin{defn}\label{MidRecurrentDefn}
We say that a quasi-geodesic $\gamma$ is \textit{$\theta$-middle recurrent} for some $\theta\in (0,1/2)$ if for any $c\ge 1$, there exists $R=R(c)$ so that any $c$-Lipschitz path $\alpha$ with endpoints $a$ and $b$ on $\gamma$ intersects the $R$-neighborhood of the $\theta$-middle of the subpath $[a,b]_\gamma$, i.e.
$$
\mathcal N_R(\alpha)\cap ([a,b]_\gamma)^\theta\ne \emptyset.
$$
\end{defn}

In the sequel, unless mentioned otherwise, $\proj_\gamma^X$ denotes the set of shortest projection in $d_X$-metric. 
The other characterization of Morse quasi-geodesics introduced  by \cite{ACGH} uses the sublinear contraction of shortest projection maps. We summarize those into the following.
\begin{lem}[\cite{ACGH,DMS, ADT17}]\label{CharMorseLem}
Let $\gamma$ be a  quasi-geodesic in $X$.   Then $\gamma$ is Morse if and only if one of the following holds:
\begin{enumerate}
    \item 
    the  gauge function   $\kappa_0:\mathbb R_{+}\to\mathbb R_+$ defined by $$ 
r\quad \longmapsto\quad \sup \left\{\mathrm{diam}_X\big(\proj_\gamma^X(B_r(x))\big): x\in X, r=d_X(x,\gamma)\right\}$$
is  sublinear.
\item 
$\gamma$ admits the $\theta$-middle recurrence for some $\theta\in (0,1/2)$.
\end{enumerate}
\end{lem}

Until  the end of this subsection, let  $\gamma$ be a Morse quasi-geodesic  in $X$, which is  $\kappa_0$-contracting as in Lemma \ref{CharMorseLem}. 

The following lemma is essentially contained in the proof of \cite[Theorem 7.1]{ACGH}. We revisit their arguments to give a slightly more general statement, adapted to our use in Section \textsection\ref{SecSublTrackAcyl}.
 
\begin{lem}\label{LipProj2Morse}
For any $\theta>0$ there exists $r=r(\theta,\kappa_0)>0$ with the following property.
Let $\alpha$ be a path with  $d_X(\alpha,\gamma)\ge r$. Then 
$$
\begin{aligned}
\mathrm{diam}_X\big(\proj_\gamma^X(\alpha)\big) \,\le\,   \theta \ell(\alpha)+\theta d_X(\alpha,\gamma).      
\end{aligned}$$  
Moreover, the same inequality holds for any projection map $\proj_\gamma^X$ that is sublinearly contracting in the sense of Definition \ref{SublinearContrDefn}.
\end{lem}
\begin{proof} 
Let   $x_0=\alpha^-$ and $L_0=d_X(x_0,\gamma)$. We inductively choose $x_{i+1}\in \alpha$  to be the   point after $x_i$ with $$\ell([x_i,x_{i+1}]_\alpha)=d_X(x_i,\gamma)=:L_i$$ for $i\ge 0$. In the end, we have  a maximal integer $N\ge 0$ so that $d_X(x_N,\alpha^+)\le L_N=d_X(x_N,\gamma)$.  Note that, $\sum_{0\le i< N}  L_i\le \ell(\alpha)$ and $L_N\le  \ell(\alpha)+d_X(\alpha,\gamma)$. 

Thanks to the sublinear $\kappa_0$, choose $r=r(\kappa_0,\theta)$ large enough so that $\kappa_0(a)\le \theta a/2$ for any $a>r$. For each $0\le i\le N$, the path $[x_i,x_{i+1}]_\alpha$ with  $x_{N+1}=\alpha^+$ is contained in a ball of radius $L_i\ge r$. Those balls are  disjoint with $\gamma$, so their projections to $\gamma$ give  $$\begin{aligned}
\mathrm{diam}_X\big(\proj_\gamma^X(\alpha)\big) & \le \sum_{0\le i< N} \kappa_0(L_i)+\kappa_0(L_N)\\
&\le \sum_{0\le i< N} \theta L_i/2+\theta \ell(\alpha)/2 + \theta d_X(\alpha,\gamma)/2  \\
&\le \theta \ell(\alpha)+\theta d_X(\alpha,\gamma)/2.
\end{aligned}$$ 
The proof does not use the particular nature of shortest projection map rather than the sublinear projection, so the same proof proves the ``moreover" statement.
\end{proof}

The following sublinearly bounded image property is proved in \cite[Theorem 7.1]{ACGH}, which characterizes the sublinear contraction. We use this property extensively in the sequel. 
\begin{lem}\label{SubLinearProj}
There  exist a sublinear function $\kappa_1=\kappa_1(\kappa_0)$ and a constant $C=C(\kappa_0)>0$  with following property. For every geodesic $\alpha=[x,y]_X$ disjoint with  $\mathcal N_C(\gamma)$,  we have   $$\mathrm{diam}_X\big(\proj_\gamma^X(\alpha)\big)\le \max\{\kappa_1(L_1),\kappa_1(L_2)\}$$
where $L_1=d_X(x,\gamma),L_2=d_X(y,\gamma)$. 
\end{lem}

We also need the sublinear bounded projection between Morse quasi-geodesics.
\begin{lem}\cite[Proposition 8.3]{ACGH}\label{SublBddProj}
There  exist a sublinear function $\kappa_2=\kappa_2(\kappa_0)$ and a constant $C=C(\kappa_0)>0$  with following property. 
Let $\alpha, \beta$ be a pair of $\kappa_0$-Morse geodesics.   If $d_X(\alpha,\beta)>C$, then $$\mathrm{diam}_X\bigl(\proj_\alpha^X(\beta)\bigr)\le \kappa_2(d_X(\alpha,\beta))$$   
\end{lem}
We derive a few   corollaries used later. In the next three results, let $\kappa_1,C$ satisfy Lemma \ref{SubLinearProj}, and  assume  $\kappa_1(C)\le C$ up to increasing $C$. 

\begin{cor}\label{SubLinearEntry}
Let $[x,y]_X$ be a geodesic entering the $C$-neighborhood of $\gamma$ at the point $y$.  Set $L=d_X(x,\gamma)$. Then $\proj_\gamma^X(x)$ is contained in the $(\kappa_1(L)+2C)$-neighborhood of $y$.  
\end{cor}
\begin{proof}
By Lemma \ref{SubLinearProj}, $\proj_\gamma^X(x)$ is contained in the $(\kappa_1(\max\{L,C\})+C)$-neighborhood of $y$. The proof is finished by $\kappa_1(\max\{L,C\})\le \max\{\kappa_1(L),\kappa_1(C)\}$ and $\kappa_1(C)\le C$. 
\end{proof}

A few more corollaries follow, whose proofs are  straightforward and left to the interested reader. %

\begin{cor}\label{SubLinearCommonEntry}
Let $\alpha, \beta$ be two geodesics emanating  from a point $o$ and ending at   $\gamma$. Let $x,y$ be the corresponding entry points of $\alpha$ and $\beta$ in $\mathcal N_C(\gamma)$. Set $L=d_X(o,\gamma)$. Then $d_X(x,y)\le 2\kappa_1(L)+4C$.    
\end{cor}


\begin{cor}\label{SubLinearContr}
Let $x$ and $y$ be two points in $X$ so that $$\dist_\gamma^X(x,y)\,>\, \max\{\kappa_1(L_1),\kappa_1(L_2)\}$$ where $L_1=d_X(x,\gamma)$ and $L_2=d_X(y,\gamma)$.  Then any geodesic $[x,y]_X$ intersects the $C$-neighborhood of $\gamma$. Moreover,   $$\left|\mathrm{diam}_X\big([x,y]_X\cap \mathcal N_C(\gamma)\big) - \dist_\gamma^X(x,y)\right|\,\le\, \kappa_1(L_1)+\kappa_1(L_2)+4C.$$ 
\end{cor}

\subsection{Morse  isometries in group actions}
Let $h\in G$ be an element  of infinite order. According to \cite{DMS}, $h$ is called \textit{Morse element} if the cyclic group $\langle h\rangle$ is a Morse  quasi-geodesic in some Cayley graph. Since Morse  quasi-geodesics are preserved under quasi-isometries,   Morse elements are well-defined independent of generating sets.   


Given an element $h\in G$, define the coarse stabilizer of $\langle h\rangle$:
$$
E(h):=\{g\in G: d_{Haus}(g\langle h\rangle, \langle h\rangle)<\infty\}
$$
where $d_{Haus}$ denotes the Hausdorff distance between subsets in word metric.

\begin{lem}
If $h$ is a Morse element in $G$, then $E(h)$ is the maximal elementary group containing $h$. Further, $E(h)$ is \emph{weakly malnormal} (i.e. for any $g\notin E(h)$, $gE(h)g^{-1}\cap E(h)$ is finite), and   
$$
E(h)=\{g\in G: \exists n>0,\, (gh^n=h^ng)\lor (gh^n=h^{-n}g) \}
$$  
\end{lem}
\begin{proof}
Via the Morse property of $\langle h\rangle$, it is a standard argument that $[E(h):\langle h\rangle]$ is finite (see \cite{YANG6}), so   $E(h)$ is an elementary group. If $E$ is an elementary subgroup with $h\in E$, then since $\langle h\rangle$ is of finite index in $E$, any element $f\in E$ is contained in $E(h)$, so $E\subset E(h)$.  

The weakly malnormality follows by the very definition of the coarse stabilizer $E(h)$.
And the description of $E(h)$ is standard, where  the $\pm$ depends on whether $g$ preserves or reserves the orientation of $\langle h\rangle$.
\end{proof}

\begin{lem}\label{lem:E(G)}
Suppose that $G$ is non-elementary and contains a Morse element. Then $G$ contains a unique maximal finite normal subgroup denoted as $E(G)$.  In addition, for any Morse element $h\in G$, $E(G)=\cap_{g\in G} gE(h)g^{-1}$.   
\end{lem}
\begin{proof}
Assume that $h$ is a Morse element. We claim that $E(G)=\cap_{g\in G} gE(h)g^{-1}$ is the unique maximal finite normal subgroup. Indeed, since  $E(h)$ is {weakly malnormal},   $gE(h)g^{-1}\cap E(h)$ is finite for any $g\notin E(h)$. Such an element $g$ exists because $G$ is non-elementary. Then $E(G)$ is a finite normal group. To see the maximality, let $F<G$ be a finite normal subgroup. For every $f\in F$ and $n\in\mathbb Z$, there exists $f_n\in F$ such that $fh^n=h^nf_n$. Thus,
$$d_{Haus}(f\langle h\rangle,\langle h\rangle)\le \max\{\|u\|_S:u\in F\}.$$
By definition, $E(h)$ is the stabilizer of $\langle h\rangle$ up to finite Hausdorff distance, so $F\subset E(h)$. Applying the same argument to every conjugate of $E(h)$ and using the normality of $F$, we obtain $F\subset E(G)$. The proof is complete.
\end{proof}

\subsubsection*{\textbf{Independent Morse elements}} The independence of two Morse elements is a non-elementary condition imposed on the groups and actions. 

\begin{defn}\label{defn:indepMorse}
Two Morse elements $h_1, h_2\in G$ are called \textit{independent}  if they have no conjugate powers; equivalently $E(h_1)\ne gE(h_2)g^{-1}$ for any $g\in G$.    
\end{defn}

As shown in the definition of $E(h)$, the independence of Morse elements manifests a kind of bounded intersection condition. 
If   $F$ is a finite set of pairwise independent Morse elements in $G$, the definition of independence implies that  
\begin{align}\label{eq:Morsesystem}
\f=\{gE(f): f\in F, g\in G \}    
\end{align} has   bounded intersection in the Cayley graph of $G$ relative to some (or any) finite generating set $S$. Conversely, the bounded intersection of $\f$ characterizes the independence of $F$.

We introduce the bounded intersection in a general isometric action of $G$ on a metric space $X$.  

\begin{defn}\label{BddProj}
A family of subsets $\f$ in $X$ has    \textit{bounded intersection property} if there exists a bounded intersection function $\sigma:\mathbb R_+\to \mathbb R_+$ so that the following holds
$$
\mathrm{diam}_X\bigl(\mathcal N_r(U)\cap \mathcal N_r(V)\bigl) \,\le\, \sigma(r)
$$
for any $r>0$ and $U\ne V\in \f$. If there exists a finite number $B>0$ so that $\mathrm{diam}_X\bigl(\proj_U^X(V)\bigr)\le B$ for any $U\ne V\in \f$, then $\f$ has    \textit{bounded projection property}. 
\end{defn}
\begin{rem}
If  $\f$ consists of  $\kappa$-contracting subsets for a constant $\kappa\equiv C$, then  the bounded intersection property of $\f$ is equivalent to the bounded projection property (\cite{YANG6}).   
\end{rem}
 
\begin{lem}\label{SublBddProj4H}
There exists a sublinear function $\kappa_2$ depending on $\f$ so that $$\mathrm{diam}_S(\proj_U^S(V))\,\le\, \kappa_2(d_S(U,V))$$  
for any $U\ne V\in \f$.
\end{lem}
\begin{proof}
Let $C,\kappa_2$ be given by Lemma \ref{SublBddProj}.
It suffices to consider  $U\ne V$ with  $d_S(U,V)\le C$.  Since  up to isometry there  are only finitely many such pairs $(U,V)$ and   $\mathrm{diam}_S(\proj_U^S(V))$ is finite, the conclusion follows by adding a fixed constant to $\kappa_2$.
\end{proof}

Fix a basepoint $o\in X$.  We call the quasi-geodesic in $X$ $$\ax(f)=E(f)o$$ \textit{quasi-axis} (depending on $o$) of $f$. We say that two Morse elements $f_1,f_2$ act as \textit{independent} isometries on $X$ if their axes $\{g\ax(f_i): i=1,2, g\in G\}$ form a bounded intersection system. 

Consider the  family of associated axis $$\pi(\f)=\{g\ax(f) : f\in F, g\in G \}$$
In general, $\pi(\f)$ may not have bounded intersection in $X$. Let us record a few settings in which such a bounded intersection system arises. 

\begin{lem}\label{lem:bddintersectionamples}
Let $F$ be a finite set of independent Morse elements in $G$. If one of the following happens
\begin{enumerate}
    \item 
    the action $G\act X$ is proper on a metric space; or 
    \item 
    the action $G\act X$ is acylindrical on a hyperbolic space; or
    \item 
    the set of elements in $F$ act as independent WPD loxodromic isometries on a hyperbolic space $X$,
\end{enumerate}
then $\pi(\f)$ has bounded intersection in $X$.    
\end{lem}
\begin{proof}
For the item (1), since the action is proper, the bounded intersection of $\pi(\f)$ in $X$ follows from that of $\f$ in $\mathrm{Cay}(G,S)$. 

The bounded intersection is exactly geometrically separated condition in \cite[Definition 4.40]{DGO}. The  items (2) and (3) follow by \cite[Theorem 6.8]{DGO}, where $\pi(\f)$ is proven to be geometrically separated.  
\end{proof}

\subsubsection*{\textbf{Morse local-global groups}}
The class of groups satisfying Morse local-global assumption is introduced in \cite[Definition 2.12]{RST21} to enable the construction of the concatenated path to be Morse. 

\begin{defn}\label{defn:MorseLoc2Glob}
We say that $G$ satisfies \textit{Morse local-global assumption} if for any Morse gauge $\mu$ and $c>1$, there exist $L>0,c'>0$, and Morse gauge $\mu'$ depending only on $\mu,c$ with the following property. If $\gamma$ is a $L$-local $\mu$-Morse $c$-quasi-geodesic path, then it is a $\mu'$-Morse $c'$-quasi-geodesic path.     
\end{defn}

The following fact is a  consequence of  \cite[Lemma 2.15]{RST21}, which is reminiscent to the classical fact in hyperbolic geometry that sufficiently long local quasi-geodesics are global quasi-geodesics. 
\begin{lem}\label{lem:Morseadmissiblepath}
Suppose that $G$ satisfies the Morse local-global assumption. Let $h,k\in G$ be two Morse elements with $E(h)\ne E(k)$. Then there exist $N,\mu,c>1$ depending on $E(h)$ and $E(k)$ so that for any $h'\in E(h)$ and for any $m_i, n_i>N$ with $1\le i\le l$, the word  
$$
(h^{n_1},k^{m_1},h^{n_2},k^{m_2},\cdots, h^{n_l},k^{m_l}, h')
$$
labels a $\mu$-Morse $c$-quasi-geodesic in $\mathrm{Cay}(G,S)$. 
\end{lem}
In particular, the following is proved.  

\begin{lem}\cite[Corollaries 3.6 \& 4.9]{RST21}\label{lem:Morse-non-elementary}
Suppose that $G$ contains a Morse element and satisfies the Morse local-global assumption. If  $G$  is non-elementary, then  
\begin{enumerate}
    \item 
    $G$ contains infinitely many independent Morse elements;
    \item 
    Every infinite normal subgroup contains a Morse element.   
\end{enumerate}
\end{lem} 
\begin{rem}
Similar conclusions were established in the proper action of  groups on general metric spaces with strongly contracting element (\cite{YANG6}).    
\end{rem}

\subsection{Sublinear extension lemma via Morse elements}\label{SSec:sublinear_extension}
Let $F$ be a set of three independent Morse elements in $G$ and let $\f$ be the system of Morse axes defined in (\ref{eq:Morsesystem}). 

\begin{conv}[$\kmorse,C$]\label{kmorseConv}
Let $\kmorse$ denote the common sublinear function for $H\in \f$. Assume in addition that  $(\kmorse, C)$ satisfy Lemmas \ref{SubLinearProj} and \ref{SublBddProj4H}. Moreover, any geodesic with endpoints in $H\in \f$ is contained in $\mathcal N_C(H)$.  
\end{conv}

Once $D$ is fixed, we freely increase $\kmorse$ by an additive constant
depending only on $\f$ and $D$. This preserves sublinearity and all the
properties in Convention \ref{kmorseConv}; fixed constants below are absorbed
in this way. We continue to denote the enlarged function by $\kmorse$.

\begin{lem}\label{BddprojViaMorse}
There exists $D=D(\f)$ with the following property.  For any geodesic $\alpha$ from $1$ to $g\in G$, there exists $f_1\ne f_2\in F$ so that 
\begin{align}
\label{eq:goodprojectpts1}\max\{\mathrm{diam}_S(\alpha\cap \mathcal N_C(H_1)),\;\mathrm{diam}_S(\alpha\cap \mathcal N_C(H_2))\}\; \le D
\end{align}
where $H_1=E(f_1)$ and $H_2=E(f_2)$.
\end{lem}
\begin{proof}
By Morse property, for any $U\in \f$, there exists $C'>0$ so that any geodesic segment with two endpoints in $\mathcal N_C(U)$ is contained in $\mathcal N_{C'}(U)$.
By the bounded intersection of $\f$, there exists  $D=D(C',\f)$ so that for any $U\ne V\in \f$, $$\mathrm{diam}_S\bigl(\mathcal N_{C'}(U)\cap \mathcal N_{C'}(V)\bigr)\le D.$$  Thus, if a word geodesic $[1,g]_S$  intersects $\mathcal N_C(E(f))$ in a diameter greater than $D$,  the intersection with $\mathcal N_C(E(f'))$ has diameter at most $D$ for any $f'\ne f\in F$. Note that $F$ contains three independent elements. Thus there are at least two elements $f_1,f_2$ from $F$ satisfying (\ref{eq:goodprojectpts1}).  
\end{proof}
\begin{lem}\label{BddprojViaMorse2}
Fix any $D>0$. Let $\alpha$ be a geodesic  from $1$ to $g\in G$ so that  $\mathrm{diam}_S(\alpha\cap \mathcal N_C(H))\le D$ where $H=E(f)$ for some $f\in F$.
Then \begin{align}
\label{eq:goodprojectpts2}\dist_{H}^S(1,g)\le D+C+\kmorse(\|g\|_S) 
\end{align} 
\end{lem}
\begin{proof}
Let $x$ be the exit point of $[1,g]_S$ in $\mathcal N_C(H)$, so  $d_S(1,x)\le D$ by assumption. Note that $x$ is $C$-close to $\proj^S_{H}(x)$, so $\proj^S_{H}(x)$ is $(D+C)$-close to $1$.  Denote $n=d_S(1,g)$. Using  Lemma \ref{SubLinearProj}, $$\mathrm{diam}_S(\proj^S_H([x,g]_S))\le \kmorse(n).$$ 
Thus, $\dist_{H}^S(1,g)\le \kmorse(n)+D+C$, so (\ref{eq:goodprojectpts2}) is proved.     
\end{proof}

\begin{defn}\label{defn:admissibletuples}
A finite sequence of elements  $(g_1, g_2,  \cdots, g_{n+1})$ in $G$ is said to be \textit{$D$-admissible} relative to $\f$ for some $D>0$ if  for any choice of geodesics $\alpha_i=[1,g_i]_S$  with $1\le i\le n+1$ there exists some $f_i\in F$  so that for $1\le i\le n$,
\begin{enumerate}
    \item 
    $\mathrm{diam}_S\bigl(\alpha_i\cap \mathcal N_C(g_iE(f_{i}))\bigr)\le D$ and $\mathrm{diam}_S\bigl(\alpha_{i+1}\cap \mathcal N_C(E(f_{i}))\bigr)\le D$.
     
    \item 
    $E(f_i)\ne g_{i+1}E(f_{i+1})$ provided that $i<n$.
\end{enumerate}
For concreteness, we also say that $(g_1, g_2,  \cdots, g_{n+1})$ in $G$ is \textit{$D$-admissible} relative to $(E(f_1),\cdots,E(f_n))$.    
\end{defn}

\begin{lem}\label{lem:admissible_tuple}
Let  $D=D(\f)$ be as in Lemma \ref{BddprojViaMorse}.
Then every finite sequence $(g_1, g_2,  \cdots, g_{n+1})$ in $G$ is \textit{$D$-admissible} relative to $\f$. 
\end{lem}
\begin{proof}
It is almost a repeated application of Lemma \ref{BddprojViaMorse}, but we need to take care of the condition (2): $E(f_i)\ne g_{i+1}E(f_{i+1})$. We explain this choice for $(g_1, g_2,g_3)$ and the general case follows inductively.

Fix $\alpha_1=[1,g_1]_S$, $\alpha_2=[1,g_2]_S$ and $\alpha_3=[1,g_3]_S$.
Applying first  Lemma \ref{BddprojViaMorse}  for $g_1^{-1}$ and $g_2$, we may choose a common $f_1\in F$ so that $$\mathrm{diam}_S\bigl(\alpha_1\cap \mathcal N_C(g_1E(f_1))\bigr),\;\mathrm{diam}_S\bigl(\alpha_2\cap \mathcal N_C(E(f_1))\bigr)\le D,$$ and then for $g_2^{-1}$ and $g_3$, a common $f_2$ so that $$\mathrm{diam}_S\bigl(\alpha_2\cap \mathcal N_C(g_2E(f_2))\bigr),\;\mathrm{diam}_S\bigl(\alpha_3\cap \mathcal N_C(E(f_2))\bigr)\le D.$$   We need to verify additionally that $E(f_1)\ne g_{2}E(f_{2})$. If $f_1\ne f_2$, then  $E(f_1)\ne g_{2}E(f_{2})$ holds by the independent assumption. Thus, if  $E(f_1)= g_{2}E(f_{2})$, then $f_1=f_2$ and $g_2\in E(f_1)$. It follows by Convention \ref{kmorseConv} that $\alpha_2\subset \mathcal N_C(E(f_1))$, so $d_S(1,g_2)\le D$.

In this case we may choose $f_2'\in F\setminus \{f_2\}$ by  Lemma \ref{BddprojViaMorse} so that $\mathrm{diam}_S\bigl(\alpha_3\cap \mathcal N_C(E(f_2'))\bigr)\le D$. Since $F$ is a finite set of independent elements, we may increase $D$ if necessary so that $$\mathrm{diam}_S(\proj_{E(f_2')}(\mathcal N_C(E(f_1))))\le D$$ so $\mathrm{diam}_S\bigl(\alpha_2\cap \mathcal N_C(g_2E(f_2'))\bigr)\le D$. Necessarily, $E(f_1)\ne g_{2}E(f_{2}')$ by independence, so the conditions (1) and (2) are fulfilled for $(\alpha_1,\alpha_2,\alpha_3)$.  Thus, the proof is finished by induction.    
\end{proof}

The following two results may be viewed as a sublinear analogue of   Extension Lemma in \cite{YANG10}. There are two principal differences. First, the required lengths of the inserted Morse elements are sublinear in the lengths of the elements being concatenated and also depend on the length of the admissible sequence. Second, the fellow-travelling conclusion has sublinear error, whereas the original lemma has a uniform additive error.

For ease of exposition, we start with the basic case of concatenating two elements. 

\begin{lem}\label{ExtensionViaMorse}
For any $D>0,$  let $(g_1,g_2)$ be a $D$-admissible tuple in $G$ relative to $E(f)$ for some $f\in F$. Denote  $n=\max\{\|g_1\|_S,\|g_2\|_S\}$. For every $h\in E(f)$ satisfying 
\begin{align}\label{eq:longMorseInsertion1}
\|h\|_S  \ge  4\kmorse(n),
\end{align} 
the path labeled by  the word $(g_1,h,g_2)$ \emph{sublinearly fellow travels} any geodesic $[1,g_1hg_2]_S$; namely,  
\begin{enumerate}
    \item 
    $[1,g_1hg_2]_S$ intersects $\mathcal N_C(g_1E(f))$ in the entry and exit points denoted as $z,w$;
    \item 
    $d_S(g_1,z)\le 3\kmorse(\|g_1\|_S)$ and $d_S(g_1h,w)\le 3\kmorse(\|g_2\|_S)$. 
\end{enumerate}
In particular,  $$
\Big | \|g_1\|_S+\|h\|_S+\|g_2\|_S - \|g_1hg_2\|_S \Big|\le 12\kmorse(n).
$$
\end{lem}
\begin{proof}
Assume by the convention above that
$\kmorse(r)>2D+2C$ for every $r$. Write $H=g_1E(f)$. Let us assume $\mathcal N_C(H)\cap [1,g_1hg_2]_S= \emptyset$ to derive a contradiction. By Lemma \ref{SubLinearProj}, $\dist^S_{H}(1,g_1hg_2)\le \kmorse(n).$ By Lemma \ref{BddprojViaMorse2}, $\dist_{H}^S(1,g_1)\le \kmorse(\|g_1\|_S)+D+C$ and $\dist_{H}^S(g_1h,g_1hg_2)\le \kmorse(\|g_2\|_S)+D+C$. Hence, 
$$\begin{aligned}
d_S(g_1,g_1h) &\le \dist_H^S(g_1,g_1h) \\
& \le \dist_H^S(g_1,1)+\dist_H^S(1,g_1hg_2)+\dist_H^S(g_1hg_2,g_1h)\\
&\le  3\kmorse(n)+2D+2C    
\end{aligned}$$ Thus $d_S(g_1,g_1h)< 4\kmorse(n)$, contradicting the choice of $h$ in Eq. (\ref{eq:longMorseInsertion1}). Hence, the assertion (1) follows.

To see the assertion (2), let $z,w$ be the entry and exit points of $[1,g_1hg_2]_S$ in $\mathcal N_C(H)$ respectively. By Lemma \ref{SubLinearProj} again, $\dist_{H}^S(1,z)\le \kmorse(\|g_1\|_S)$, so  $$\begin{aligned}
d_S(g_1,z)&\le C+ \dist_{H}^S(g_1,1)+\dist_{H}^S(1,z)\\
&\le 2\kmorse(\|g_1\|_S)+D+2C \le 3\kmorse(\|g_1\|_S)    
\end{aligned}$$ Similarly, one obtains $d_S(g_1h,w)\le 3\kmorse(\|g_2\|_S)$. Thus, the assertion (2) is proved. The ``in particular" follows by applying twice triangle inequality.
\end{proof}

The same argument applies to arbitrarily long admissible sequences, although the required lengths of the inserted elements depend on the length of the sequence. In later applications, sequences of length three suffice.

\begin{lem}\label{ExtensionViaMorse2}
For any $D>0$, assume that a triple  of elements $(g_1,g_2, g_3)$ in $G$ is $D$-admissible relative to $(E(f_1),E(f_2))$ for some $f_1,f_2\in F$. 

If $n=\max\{\|g_i\|_S: i=1,2,3\}$ is sufficiently large, then for every $h_1\in E(f_1), h_2\in E(f_2)$ satisfying 
\begin{align}\label{eq:longMorseInsertion2}
n\,\ge \,\max\{\|h_1\|_S,\|h_2\|_S\}\,\ge \,\min\{\|h_1\|_S,\|h_2\|_S\}\,>\,  10\kmorse(n),    
\end{align} the  path labeled by  $(g_1, h_1,g_2,h_2, g_3)$ \emph{sublinearly fellow travels} any geodesic $[1,g_1h_1g_2h_2g_3]_S$:  
\begin{enumerate}
    \item 
    $[1,g_1h_1g_2h_2g_3]_S$ intersects $\mathcal N_C(g_1E(f_1))$ and $\mathcal N_C(g_1h_1g_2E(f_2))$ in the entry and exit points denoted as $z_1,w_1$ and $z_2,w_2$ respectively;
    \item 
    $d_S(g_1,z_1)\le 3\kmorse(\|g_1\|_S)$, $d_S(g_1h_1g_2h_2,w_2)\le 3\kmorse(\|g_3\|_S)$, and  $$d_S(g_1h_1,w_1),\,d_S(g_1h_1g_2,z_2)\le 10\kmorse(n).$$ 
\end{enumerate}  
\end{lem}
\begin{proof} 
Let $C_1$ be such that every geodesic with endpoints in
$\mathcal N_C(U)$ is contained in $\mathcal N_{C_1}(U)$ for $U\in\f$, and
let $L$ bound
\[
\mathrm{diam}_S\bigl(\mathcal N_C(U)\cap\mathcal N_{C_1}(V)\bigr)
\]
over distinct $U,V\in\f$. Using the convention above, assume that
$\kmorse(r)>2D+8C+L$ for every $r$.
Choose $n_1>C$ so that $10\kmorse(r)<r$ for every $r\ge n_1$.
Let $\gamma=[1,g_1h_1g_2h_2g_3]_S$ be any geodesic. The proof strategy is the same as Lemma \ref{ExtensionViaMorse} but with more involved estimates. The conclusion holds up to isometry, so we may apply $g_1^{-1}$ to the labeled path and consider its  subpath  $$\gamma_1:=h_1[1,g_2]_S\cdot (h_1g_2[1, h_2]_S)\cdot (h_1g_2h_2[1, g_3]_S)$$ Denote $H_1=E(f_1)$ and $H_2=h_1g_2E(f_2)$; these are distinct by admissibility. We use the same notation for points on $g_1^{-1}\gamma$ and their translates back to $\gamma$. We shall prove that $\gamma_1$ has a sublinearly bounded projection to $H_1$. To this end, let $x$ be the exit point of $\gamma_1$ in $\mathcal N_C(H_1)$. We first prove that $x$ must be contained in the subpath of $\gamma_1$ $$\gamma_2:=h_1[1,g_2]_S\cdot (h_1g_2[1, h_2]_S)$$

Indeed, suppose that $x$ lies on
$\gamma_1\setminus\gamma_2=h_1g_2h_2[1,g_3]_S$, and write
$x=h_1g_2h_2t$, where $t$ is a vertex of $[1,g_3]_S$. By
$D$-admissibility and Lemma \ref{BddprojViaMorse2},
\[
\dist_{H_2}^S(h_1,h_1g_2),\quad
\dist_{H_2}^S(h_1g_2h_2,x)\le D+C+\kmorse(n).
\]
Let $\delta=[h_1,x]_S$ be any geodesic. If $\delta$ were disjoint from
$\mathcal N_C(H_2)$, Lemma \ref{SubLinearProj} would give
$\dist_{H_2}^S(h_1,x)\le\kmorse(n)$, contradicting
\[
\dist_{H_2}^S(h_1,x)
\ge d_S(1,h_2)-2(D+C+\kmorse(n)).
\]
Let $a,b$ be the entry and exit points of $\delta$ in
$\mathcal N_C(H_2)$. Applying Lemma \ref{SubLinearProj} to the two outside
subsegments and using the preceding inequalities gives
\[
\dist_{H_2}^S(a,b)
\ge d_S(1,h_2)-4\kmorse(n)-2D-2C.
\]
Since $a,b\in\mathcal N_C(H_2)$, it follows that
\[
\mathrm{diam}_S\bigl(\delta\cap\mathcal N_C(H_2)\bigr)
\ge d_S(1,h_2)-6\kmorse(n),
\]
where fixed constants are absorbed according to the convention above.
On the other hand, since $H_1$ is Morse, $\delta$ is contained in
$\mathcal N_{C_1}(H_1)$. By the choice of $L$,
\[
\mathrm{diam}\bigl(\mathcal N_C(H_2)\cap
\mathcal N_{C_1}(H_1)\bigr)\le L.
\]
Then $d_S(1,h_2)-6\kmorse(n)>L$ by
Eq. (\ref{eq:longMorseInsertion2}), a contradiction. Hence, $x$ lies on
$\gamma_2$, so $h_1g_2h_2[1, g_3]_S\cap \mathcal N_C(H_1)=\emptyset$.  By Lemma \ref{SubLinearProj}, $$\dist_{H_1}^S(h_1g_2h_2,h_1g_2h_2g_3)\le \kmorse(3n)\le 3\kmorse(n)$$ Further, the $D$-bounded intersection $[h_1,g_2]_S\cap \mathcal N_C(H_1)$ with Lemma \ref{BddprojViaMorse2} implies $\dist_{H_1}^S(h_1,h_1g_2)\le C+D+\kmorse(\|g_2\|_S)$. We now bound the projection of $\gamma_1$ to $H_1$ as follows 
\begin{equation}\label{eq:gamma1toH1}\begin{aligned}
&\quad \dist_{H_1}^S(h_1,h_1g_2h_2g_3)\\
\le & \quad \dist_{H_1}^S(h_1,h_1g_2)+\dist_{H_1}^S(h_1g_2,h_1g_2h_2)+\dist_{H_1}^S(h_1g_2h_2,h_1g_2h_2g_3)\\
\le & \quad   (C+D+\kmorse(n))+\kmorse(n)+3\kmorse(n)\\
\le & \quad C+D+5\kmorse(n)   
\end{aligned}\end{equation}
where $\dist_{H_1}^S(h_1g_2,h_1g_2h_2)\le \mathrm{diam}_S\bigl(\proj_{H_1}^S(H_2)\bigr)\le \kmorse(n)$ by Lemma \ref{SublBddProj4H}.
\medskip

\noindent The remaining proof proceeds to show that $g_1^{-1}\gamma$ intersects $\mathcal N_C(H_1)$ and $\mathcal N_C(H_2)$. 
\medskip

\noindent \textbf{(1).}  We first see $\mathcal N_C(H_1)\cap g_1^{-1}\gamma\ne \emptyset$. Indeed, if not,  Lemma \ref{SubLinearProj} gives $$\dist_{H_1}^S(g_1^{-1},h_1g_2h_2g_3)\le \kmorse(3n)\le 3\kmorse(n)$$ The $D$-bounded intersection of $[1,g_1]_S\cap \mathcal N_C(H_1)$  with Lemma \ref{BddprojViaMorse2} implies $$\dist_{H_1}^S(g_1^{-1},1)\le C+D+\kmorse(\|g_1\|_S)$$ Combining these inequalities with Eq. (\ref{eq:gamma1toH1}),
$$
\begin{aligned}
d_S(1,h_1)&\le \dist_{H_1}^S(1,g_1^{-1})+\dist_{H_1}^S(g_1^{-1},h_1g_2h_2g_3)+\dist_{H_1}^S(h_1,h_1g_2h_2g_3)\\
&\le 2(C+D)+9\kmorse(n).
\end{aligned}$$ This contradicts $d_S(1,h_1)>10\kmorse(n)$ in Eq. (\ref{eq:longMorseInsertion2}). Hence, $\mathcal N_C(H_1)\cap g_1^{-1}\gamma\ne \emptyset$. 

\medskip
\noindent
Let $z_1,w_1$ be the entry and exit points of $g_1^{-1}\gamma$ in $\mathcal N_C(H_1)$ respectively.  So  Lemma \ref{SubLinearProj} gives $\dist_{H_1}^S(g_1^{-1},z_1)\le \kmorse(\|g_1\|_S)$ and $\dist_{H_1}^S(w_1,h_1g_2h_2g_3)\le \kmorse(3n)\le 3\kmorse(n)$. We first bound  $$\begin{aligned}
d_S(1,z_1)&\le C+ \dist_{H_1}^S(g_1^{-1},1)+\dist_{H_1}^S(g_1^{-1},z_1)\\
&\le 2\kmorse(\|g_1\|_S)+D+2C \le 3\kmorse(\|g_1\|_S)    
\end{aligned}$$
and using (\ref{eq:gamma1toH1}),  $$\begin{aligned}
d_S(h_1,w_1)&\le C+ \dist_{H_1}^S(h_1,h_1g_2h_2g_3)+\dist_{H_1}^S(h_1g_2h_2g_3,w_1)\\
&\le 2C+D+8\kmorse(n) \le 9\kmorse(n).    
\end{aligned}$$

\noindent \textbf{(2).} Let us first assume to the contrary that $g_1^{-1}\gamma$ is disjoint with $\mathcal N_C(H_2)$.  Choose some $w_1'\in H_1$ so that $d_S(w_1,w_1')\le C$. If  $[w_1,w_1']_S$ meets $N_C(H_2)$, then $\dist_{H_2}^S(w_1,w_1')\le (3C+2)$ by Eq.~\ref{eq:proj_in_nbld}. Otherwise, since
\[
d_S(w_1,H_2)\le d_S(w_1,h_1)+d_S(1,g_2)
\le9\kmorse(n)+n\le2n,
\] we obtain $\dist_{H_2}^S(w_1,w_1')\le\kmorse(2n)$ by the sublinear contraction. Thus,  
$\dist_{H_2}^S(w_1,w_1')\le \max\{\kmorse(2n),(3C+2)\}<3\kmorse(n)$.
We see that $$
\begin{aligned}
\dist_{H_2}^S(w_1,h_1g_2)&\le \dist_{H_2}^S(w_1,w_1')+\mathrm{diam}_S(\proj_{H_2}(H_1))+\dist_{H_2}^S(h_1,h_1g_2)\\
&\le D+C+5\kmorse(n).    
\end{aligned}$$ 
Here, $\dist_{H_2}^S(h_1,h_1g_2), \dist_{H_2}^S(h_1g_2h_2,h_1g_2h_2g_3)\le D+C+ \kmorse(n)$ by the admissible condition at $H_2$ with Lemma \ref{BddprojViaMorse2}. Moreover, Lemma \ref{SubLinearProj} gives $$\dist_{H_2}^S(w_1,h_1g_2h_2g_3) \le \kmorse(2n)\le 2\kmorse(n).$$ Combining these inequalities gives $$
\begin{aligned}
d_S(1,h_2)&\le \dist_{H_2}^S(h_1g_2,h_1g_2h_2)\\
&\le \dist_{H_2}^S(h_1g_2,w_1)+\dist_{H_2}^S(w_1,h_1g_2h_2g_3)+\dist_{H_2}^S(h_1g_2h_2g_3,h_1g_2h_2)\\
&\le  2D+2C+8\kmorse(n)\le 9\kmorse(n).
\end{aligned}$$ 
This contradicts $d_S(1,h_2)>10\kmorse(n)$ in Eq. (\ref{eq:longMorseInsertion2}).  
Consider the entry and exit points $z_2,w_2$ of $g_1^{-1}\gamma$ in $\mathcal N_C(H_2)$.
Reversing the orientation of $\gamma$, $H_1$ and $H_2$ are symmetric positions on $\gamma$, so we may obtain similarly as above: $d_S(h_1g_2,z_2)\le 9\kmorse(n)$ and $d_S(h_1g_2h_2,w_2)\le 3\kmorse(\|g_3\|_S)$. Hence, the assertion (2) follows and the proof is complete.
\end{proof}

\section{Sublinear projection tracking in a proper action}\label{sec:SPTProperAction}

Let $G$ act by isometry on a metric space $(X,d_X)$ with a finite generating set $S$. We recall the   notion of sublinear projection tracking given in Introduction. 

\begin{defn}\label{SublTrackDefn}
Let $H\subseteq G$ be a subset and $g\in G$ be any element. 
\begin{enumerate}
    \item 
    The \textit{$d_S$-projection} $\proj^S_{H}(g)$   is  the set of elements $h\in H$ so that $h$ is a shortest $d_S$-projection point of $g$ to $H$: $d_S(g,h)=d_S(g,H)$.
    
    \item The \textit{$d_X$-projection} $\proj^X_{H}(g)$  is  the set of elements $h\in H$ so that $ho$ is a shortest $d_X$-projection point of $go$ to $Ho$: $d_X(go, ho)=d_X(go,Ho)$.  
\end{enumerate}
By definition,  $\proj^X_{fH}(fg)=f\cdot \proj^X_{H}(g)$ and $\proj^S_{fH}(fg)=f\cdot \proj^S_{H}(g)$  for any $f\in G$.
\end{defn}

The sublinear tracking property relates these projections up to a sublinear error.
\begin{defn}\label{SublinearTrackingDefn}
Let $\ktrack$ be a sublinear function.  
We say that the $d_S$-projection   \textit{$\ktrack$-tracks} the $d_X$-projection along $H$ if for any $g\in G$, 
\begin{align}
\label{SublTrackEq}
\mathrm{diam}_S\Big(\proj^X_{H}(g)\cup \proj^S_{H}(g)\Big)\le \ktrack\left(d_S(g,H)\right)   
\end{align}
Further, the $d_S$-projection   \textit{$\ktrack$-tracks} the $d_X$-projection along a family of subsets $\f=\{H_i\subseteq G:i\in I\}$ if it does so along each member $H\in \f$.

We also say that the action $G\act X$ has \textit{sublinear projection tracking} along $H$ (respectively $\f$).
\end{defn}
\begin{rem}
By triangle inequality,  $$\Big|d_S(g,H)-d_S(g,\proj^X_{H}(g))\Big|\le \mathrm{diam}_S\Big(\proj^X_{H}(g)\cup \proj^S_{H}(g)\Big)$$ 
Equivalently, the inequality (\ref{SublTrackEq})  may be  replaced  with the following
\begin{align}
\label{SublTrackEq2}
\mathrm{diam}_S\Big(\proj^X_{H}(g)\cup \proj^S_{H}(g)\Big)\le \ktrack'\left(d_S(g,\proj^X_{H}(g))\right)
\end{align}
for some $\ktrack'$ depending only on $\ktrack$. 
\end{rem}

The point of sublinear projection tracking is as follows. 
If we denote $$\dist_{H}^S(g,h)=\mathrm{diam}_S\Big(\proj^S_{H}(g)\cup \proj^S_{H}(h)\Big)\quad \text{and}\quad\dist_{H}^X(g,h)=\mathrm{diam}_S\Big(\proj^X_{H}(g)\cup \proj^X_{H}(h)\Big)$$ 
then
\begin{align}
\label{SublTrackInquality}
\Big|\dist_{H}^S(g,h)-\dist_{H}^X(g,h)\Big|\le \ktrack(d_S(h,H))+\ktrack(d_S(g,H))
\end{align}

We prove sublinear projection tracking along Morse isometries for proper
actions.

\begin{lem}\label{SublTrackInProper}
{Assume that $G$ acts properly on a metric space $X$.} Consider an element  $f\in G$ which acts as a Morse  isometry on $X$. Set $H=\langle f\rangle$.
There exists a sublinear function $\ktrack$ with   the following property.  For any $g\in G$, denoting $n={d_S(g,\proj^S_{H}(g))}$, we have
$$
{\mathrm{diam}_S\Big(\proj^S_{H}(g)\cup\proj^X_{H}(g)\Big)}\le \ktrack(n). 
$$
\end{lem}
\begin{rem}
In Sect.~\ref{SecSublTrackAcyl}, the corresponding result is proved for Morse WPD isometries in general, possibly non-proper, actions.  Since   every Morse element in a
proper action is WPD, the
proper case is formally subsumed by the WPD case.  However, its proof is considerably simpler and provides the main intuition for the general argument, so we include it separately.    
\end{rem}

\begin{proof}[Proof of Lemma \ref{SublTrackInProper}]
Let
\[
\rho(r)=\sup\Big\{\mathrm{diam}_S\big(\proj^S_H(g)\cup\proj^X_H(g)\big):
g\in G,\ d_S(g,H)\le r\Big\}.
\]
First note that $\rho(r)<\infty$ for every $r$. Indeed, if $k\in\proj^S_H(g)$,
then, after translating by $k^{-1}$, the element $k^{-1}g$ belongs to the finite
ball $B_S(1,r)$. The $d_S$-projection sets of elements in this ball are finite,
and properness of the action implies that their $d_X$-projection sets to $Ho$
are finite as well. Thus their union has finite $d_S$-diameter.

Fix $g\in G$, and choose arbitrary points $h\in\proj^X_H(g)$ and
$k\in\proj^S_H(g)$. Set
\[
n=d_S(g,k)=d_S(g,H),\qquad l=d_X(ho,ko).
\]
Choose $\lambda\ge1$ so that the orbit map is $\lambda$-Lipschitz and its
restriction to $H$ is a $\lambda$-quasi-isometric embedding. It suffices to
prove, uniformly over all such triples, that $l/n\to0$ as $n\to\infty$.
Suppose otherwise. Then there are $\epsilon>0$ and triples $(g,h,k)$ with
$n\to\infty$ such that
\[
l\ge\epsilon n.
\]
   
\medskip

Consider the path $\alpha$ obtained by concatenating $\pi([k,g]_S)$ and
$[go,ho]_X$. Since $h$ is a shortest $d_X$-projection point and $k\in H$,
\[
d_X(go,ho)\le d_X(go,ko)\le\lambda n.
\]
Also $\ell(\pi([k,g]_S))\le\lambda n$, and hence
\[
\ell(\alpha)\le2\lambda n\le \frac{2\lambda}{\epsilon}d_X(ko,ho).
\]
Thus $\alpha$ is a $c$-Lipschitz path for $c=2\lambda/\epsilon$.
\medskip

Denote by $\gamma=Ho$ the Morse quasi-geodesic. We regard $\gamma$ and
$\pi([k,g]_S)$ as paths whose consecutive orbital points are at distance at
most $\lambda$.
\medskip

\noindent By Lemma \ref{CharMorseLem}, $\gamma$ is $\theta$-middle recurrent
for some $\theta>0$. Therefore there is $R=R(c)$ such that $\alpha$ meets the
$R$-neighborhood of the $\theta$-middle $([ho,ko]_\gamma)^\theta$. Every point
of this middle is at distance at least $\theta l$ from $ho$. If a point of the
middle is $R$-close to $[go,ho]_X$, the shortest-projection property of $ho$
implies that it is $(2R+\lambda)$-close to $ho$, which is impossible for all
sufficiently large $l$. Consequently, for all sufficiently large triples, the
$\theta$-middle meets the $R$-neighborhood of $\pi([k,g]_S)$.
\medskip

Choose a point in this middle and then $x\in H$ whose orbit point is within
$\lambda$ of it. There is also $y\in[k,g]_S$ such that
$d_X(xo,yo)\le R+2\lambda$. Properness gives a constant
\[
\widetilde R=\max\{\|q\|_S:q\in G,\ d_X(o,qo)\le R+2\lambda\}<\infty
\]
such that $d_S(x,y)\le\widetilde R$. On the other hand,
\[
d_S(x,k)\ge \frac{d_X(xo,ko)}{\lambda}
\ge \frac{\theta l-\lambda}{\lambda}.
\]
For sufficiently large $l$, this is greater than $2\widetilde R$, and hence
\[
d_S(y,k)>\widetilde R\ge d_S(y,x).
\]
Since $y\in[k,g]_S$, it follows that
\[
d_S(g,x)<d_S(g,y)+d_S(y,k)=d_S(g,k),
\]
contradicting $k\in\proj^S_H(g)$. Thus $l/n\to0$ uniformly. Since $h$ and $k$
were arbitrary, the quasi-isometric embedding of $H$ and the finiteness of
$\rho(r)$ imply, by splitting its defining supremum into bounded and unbounded
values of $d_S(g,H)$, that $\rho(r)/r\to0$. This proves the lemma.
\end{proof}

\begin{rem}\label{SPTQInv}
The proof actually does not use the group action in an essential way. One may use the same proof, with minor modifications, to prove the following.
Let $\pi:X\to Y$ be a quasi-isometry and  $A\subseteq X$ be a Morse quasi-geodesic. Then the shortest $d_X$-projection to $A$ sublinearly tracks the pullback of the shortest $d_Y$-projection to $\pi(A)$.      
\end{rem}

The other fact is that the pullback of the shortest $d_X$-projection has strongly sublinear contraction (see Definition \ref{SublinearContrDefn}). 
\begin{lem}\label{SSublContrInProper}
Assume that the action of $G$ on a metric space $X$ is proper. Let $f\in G$ be an element which acts as a strongly contracting isometry on $X$. Set $H=\langle f\rangle$. Then there exists a sublinear function $\kmorse$ with the following property.
Let $\alpha=[x,y]$ be a word geodesic between $x,y\in G$ of length $n$ so that $d_S(\alpha,H)\ge n$. Then $$\mathrm{diam}_S\Big(\proj^X_{H}(x)\cup \proj^X_{H}(y)\Big) \le \kmorse(n).$$
\end{lem}
\begin{proof}
By assumption, there is a sublinear function $\kappa_0$ such that
$\gamma:=Ho$ is a strongly $\kappa_0$-contracting quasi-geodesic. Choose
$\lambda\ge1$ so that the orbit map is $\lambda$-Lipschitz and
$\pi|_H$ is a $\lambda$-quasi-isometric embedding.

We first note that the supremum of
\[
d_X(uo,vo),
\]
over all word geodesics $[x,y]_S$ of length at most $N$ satisfying
$d_S([x,y]_S,H)\ge d_S(x,y)$ and all choices
$uo\in\proj^X_{Ho}(xo)$, $vo\in\proj^X_{Ho}(yo)$, is finite. Indeed, if
$d_X(xo,Ho)\ge d_X(xo,yo)$, strong contraction bounds this quantity by
$\kappa_0(\lambda N)$. Otherwise $d_X(xo,Ho)<\lambda N$; after translating a
projection point of $xo$ to $o$, properness leaves only finitely many possible
elements $x$ and then finitely many possible $y$ and projection points.

It therefore suffices to prove uniformly that, whenever
$d_S([x,y]_S,H)\ge d_S(x,y)$,
\[
\frac{d_X(uo,vo)}{d_S(x,y)}\longrightarrow0
\quad\text{as }d_S(x,y)\to\infty.
\]
Suppose otherwise. Then there are $\epsilon>0$, word geodesics
$\alpha_i=[x_i,y_i]_S$, and projection points $u_io,v_io\in Ho$ such that,
with $N_i=d_S(x_i,y_i)\to\infty$,
\[
d_S(\alpha_i,H)\ge N_i,
\qquad d_X(u_io,v_io)\ge\epsilon N_i.
\]
\noindent We may assume, after passing to a subsequence, that
\[
\max\{d_X(x_io,u_io),d_X(y_io,v_io)\}\le\lambda N_i.
\]
Indeed, otherwise orienting $[x_i,y_i]_S$ from an endpoint whose distance to
$Ho$ is greater than $\lambda N_i\ge d_X(x_io,y_io)$ and applying strong
contraction gives
\[
d_X(u_io,v_io)\le\kappa_0(\lambda N_i),
\]
contrary to $d_X(u_io,v_io)\ge\epsilon N_i$.

\medskip

\noindent Let
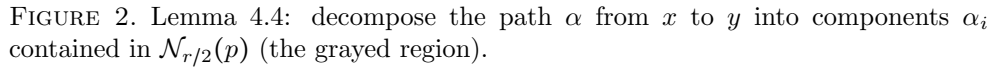
\[
\beta_i=[u_io,x_io]_X\cdot\pi(\alpha_i)\cdot[y_io,v_io]_X.
\]
This path joins $u_io$ to $v_io$ and has length at most $3\lambda N_i$.
Thus, for $c=3\lambda/\epsilon$,
\[
\ell(\beta_i)\le c\,d_X(u_io,v_io),
\]
so $\beta_i$ is a $c$-Lipschitz path.
\medskip

Recall that $\gamma=Ho$ is $\theta$-middle recurrent for some $\theta>0$.
Thus the $\theta$-middle of $[u_io,v_io]_\gamma$ meets the $R$-neighborhood of
$\beta_i$ for a constant $R=R(c)$. If such a middle point is $R$-close to
$[u_io,x_io]_X$ or $[y_io,v_io]_X$, the shortest-projection property implies
that it is $(2R+\lambda)$-close to $u_io$ or $v_io$, respectively. This
contradicts its being in the $\theta$-middle because
$d_X(u_io,v_io)\to\infty$. Hence, for all
large $i$, there are $z_i\in H$ and $a_i\in\alpha_i$ such that
\[
d_X(z_io,a_io)\le R+2\lambda.
\]
Properness gives a constant $R'$ independent of $i$ such that
$d_S(z_i,a_i)\le R'$. This contradicts
\[
d_S(\alpha_i,H)\ge N_i\longrightarrow\infty.
\]
The locally finite supremum above therefore defines a sublinear bound, which
the quasi-isometric embedding $\pi|_H$ transfers to the word metric. Finally,  Lemma
\ref{lem:reduction} implies the
strongly sublinear contraction property in the word metric.
\end{proof}

\section{Sublinear projection tracking along WPD elements}\label{SecSublTrackAcyl}
We establish \ref{SubTrackThm} on sublinear projection tracking in acylindrical actions. The arguments use the acylindricity of the action (or merely the existence of WPD elements) and the Morse property of the axis, but not hyperbolicity of the ambient space.

\subsection{Generalized Morse Lemma}\label{SecFracMorseLem}
We first prove a generalized version of the Morse lemma in general metric spaces. The statement is akin to the bottleneck property of quasi-trees (Remark \ref{QTRem}), so this part may be omitted by readers interested only in acylindrical actions on quasi-trees (see \cite{Balasub} for such a reduction). 


\begin{figure}
    \centering

\tikzset{every picture/.style={line width=0.75pt}} 

\begin{tikzpicture}[x=0.75pt,y=0.75pt,yscale=-1,xscale=1]

\draw  [fill={rgb, 255:red, 155; green, 155; blue, 155 }  ,fill opacity=0.54 ][dash pattern={on 4.5pt off 4.5pt}] (213.5,146) .. controls (213.5,136.61) and (221.11,129) .. (230.5,129) -- (413.5,129) .. controls (422.89,129) and (430.5,136.61) .. (430.5,146) -- (430.5,214) .. controls (430.5,214) and (430.5,214) .. (430.5,214) -- (213.5,214) .. controls (213.5,214) and (213.5,214) .. (213.5,214) -- cycle ;
\draw    (105.5,214) -- (539.5,213) ;
\draw  [dash pattern={on 4.5pt off 4.5pt}]  (173.5,152) -- (174.5,212) ;
\draw [shift={(174.5,212)}, rotate = 89.05] [color={rgb, 255:red, 0; green, 0; blue, 0 }  ][fill={rgb, 255:red, 0; green, 0; blue, 0 }  ][line width=0.75]      (0, 0) circle [x radius= 3.35, y radius= 3.35]   ;
\draw [shift={(173.5,152)}, rotate = 89.05] [color={rgb, 255:red, 0; green, 0; blue, 0 }  ][fill={rgb, 255:red, 0; green, 0; blue, 0 }  ][line width=0.75]      (0, 0) circle [x radius= 3.35, y radius= 3.35]   ;
\draw  [dash pattern={on 4.5pt off 4.5pt}]  (471.13,119.38) -- (469.63,212) ;
\draw [shift={(469.63,212)}, rotate = 90.93] [color={rgb, 255:red, 0; green, 0; blue, 0 }  ][fill={rgb, 255:red, 0; green, 0; blue, 0 }  ][line width=0.75]      (0, 0) circle [x radius= 3.35, y radius= 3.35]   ;
\draw [shift={(471.13,119.38)}, rotate = 90.93] [color={rgb, 255:red, 0; green, 0; blue, 0 }  ][fill={rgb, 255:red, 0; green, 0; blue, 0 }  ][line width=0.75]      (0, 0) circle [x radius= 3.35, y radius= 3.35]   ;
\draw  [line width=0.75]  (235.5,229) .. controls (235.47,233.67) and (237.79,236.01) .. (242.46,236.04) -- (301.32,236.39) .. controls (307.99,236.43) and (311.31,238.78) .. (311.28,243.45) .. controls (311.31,238.78) and (314.65,236.47) .. (321.32,236.51)(318.32,236.49) -- (396.46,236.96) .. controls (401.13,236.99) and (403.47,234.67) .. (403.5,230) ;
\draw [color={rgb, 255:red, 74; green, 144; blue, 226 }  ,draw opacity=1 ][line width=1.5]    (230,213) -- (411.5,214) ;
\draw [shift={(411.5,214)}, rotate = 180.32] [color={rgb, 255:red, 74; green, 144; blue, 226 }  ,draw opacity=1 ][line width=1.5]      (6.71,-6.71) .. controls (3.01,-6.71) and (0,-3.7) .. (0,0) .. controls (0,3.7) and (3.01,6.71) .. (6.71,6.71) ;
\draw [shift={(230,213)}, rotate = 0.32] [color={rgb, 255:red, 74; green, 144; blue, 226 }  ,draw opacity=1 ][line width=1.5]      (6.71,-6.71) .. controls (3.01,-6.71) and (0,-3.7) .. (0,0) .. controls (0,3.7) and (3.01,6.71) .. (6.71,6.71) ;
\draw [color={rgb, 255:red, 74; green, 144; blue, 226 }  ,draw opacity=1 ][line width=1.5]    (173.5,152) .. controls (167,115.5) and (171.5,55) .. (215.5,57) .. controls (259.5,59) and (234.5,190) .. (264.5,134) .. controls (294.5,78) and (292.5,177) .. (305.5,137) .. controls (318.5,97) and (350.5,122) .. (348.5,136) .. controls (346.5,150) and (367.5,164) .. (387.5,120) .. controls (407.5,76) and (447.31,48.19) .. (471.13,119.38) ;
\draw [shift={(471.13,119.38)}, rotate = 71.5] [color={rgb, 255:red, 74; green, 144; blue, 226 }  ,draw opacity=1 ][fill={rgb, 255:red, 74; green, 144; blue, 226 }  ,fill opacity=1 ][line width=1.5]      (0, 0) circle [x radius= 4.36, y radius= 4.36]   ;
\draw [shift={(173.5,152)}, rotate = 259.9] [color={rgb, 255:red, 74; green, 144; blue, 226 }  ,draw opacity=1 ][fill={rgb, 255:red, 74; green, 144; blue, 226 }  ,fill opacity=1 ][line width=1.5]      (0, 0) circle [x radius= 4.36, y radius= 4.36]   ;
\draw  [dash pattern={on 4.5pt off 4.5pt}]  (267.5,130) -- (268.5,214) ;
\draw [shift={(268.5,214)}, rotate = 89.32] [color={rgb, 255:red, 0; green, 0; blue, 0 }  ][fill={rgb, 255:red, 0; green, 0; blue, 0 }  ][line width=0.75]      (0, 0) circle [x radius= 3.35, y radius= 3.35]   ;
\draw [shift={(267.5,130)}, rotate = 89.32] [color={rgb, 255:red, 0; green, 0; blue, 0 }  ][fill={rgb, 255:red, 0; green, 0; blue, 0 }  ][line width=0.75]      (0, 0) circle [x radius= 3.35, y radius= 3.35]   ;
\draw  [dash pattern={on 4.5pt off 4.5pt}]  (291.5,130) -- (292.5,213) ;
\draw [shift={(292.5,213)}, rotate = 89.31] [color={rgb, 255:red, 0; green, 0; blue, 0 }  ][fill={rgb, 255:red, 0; green, 0; blue, 0 }  ][line width=0.75]      (0, 0) circle [x radius= 3.35, y radius= 3.35]   ;
\draw [shift={(291.5,130)}, rotate = 89.31] [color={rgb, 255:red, 0; green, 0; blue, 0 }  ][fill={rgb, 255:red, 0; green, 0; blue, 0 }  ][line width=0.75]      (0, 0) circle [x radius= 3.35, y radius= 3.35]   ;
\draw  [dash pattern={on 4.5pt off 4.5pt}]  (308.5,130) -- (309.5,213) ;
\draw [shift={(309.5,213)}, rotate = 89.31] [color={rgb, 255:red, 0; green, 0; blue, 0 }  ][fill={rgb, 255:red, 0; green, 0; blue, 0 }  ][line width=0.75]      (0, 0) circle [x radius= 3.35, y radius= 3.35]   ;
\draw [shift={(308.5,130)}, rotate = 89.31] [color={rgb, 255:red, 0; green, 0; blue, 0 }  ][fill={rgb, 255:red, 0; green, 0; blue, 0 }  ][line width=0.75]      (0, 0) circle [x radius= 3.35, y radius= 3.35]   ;
\draw  [dash pattern={on 4.5pt off 4.5pt}]  (347.5,130) -- (347.66,142.99) -- (348.5,213) ;
\draw [shift={(348.5,213)}, rotate = 89.31] [color={rgb, 255:red, 0; green, 0; blue, 0 }  ][fill={rgb, 255:red, 0; green, 0; blue, 0 }  ][line width=0.75]      (0, 0) circle [x radius= 3.35, y radius= 3.35]   ;
\draw [shift={(347.5,130)}, rotate = 89.28] [color={rgb, 255:red, 0; green, 0; blue, 0 }  ][fill={rgb, 255:red, 0; green, 0; blue, 0 }  ][line width=0.75]      (0, 0) circle [x radius= 3.35, y radius= 3.35]   ;
\draw  [dash pattern={on 4.5pt off 4.5pt}]  (246,130) -- (247.5,213) ;
\draw [shift={(247.5,213)}, rotate = 88.96] [color={rgb, 255:red, 0; green, 0; blue, 0 }  ][fill={rgb, 255:red, 0; green, 0; blue, 0 }  ][line width=0.75]      (0, 0) circle [x radius= 3.35, y radius= 3.35]   ;
\draw [shift={(246,130)}, rotate = 88.96] [color={rgb, 255:red, 0; green, 0; blue, 0 }  ][fill={rgb, 255:red, 0; green, 0; blue, 0 }  ][line width=0.75]      (0, 0) circle [x radius= 3.35, y radius= 3.35]   ;
\draw  [dash pattern={on 4.5pt off 4.5pt}]  (384.5,129) -- (384.5,213) ;
\draw [shift={(384.5,213)}, rotate = 90] [color={rgb, 255:red, 0; green, 0; blue, 0 }  ][fill={rgb, 255:red, 0; green, 0; blue, 0 }  ][line width=0.75]      (0, 0) circle [x radius= 3.35, y radius= 3.35]   ;
\draw [shift={(384.5,129)}, rotate = 90] [color={rgb, 255:red, 0; green, 0; blue, 0 }  ][fill={rgb, 255:red, 0; green, 0; blue, 0 }  ][line width=0.75]      (0, 0) circle [x radius= 3.35, y radius= 3.35]   ;

\draw (159,141.4) node [anchor=north west][inner sep=0.75pt]    {$x$};
\draw (479,106.4) node [anchor=north west][inner sep=0.75pt]    {$y$};
\draw (158.5,217.4) node [anchor=north west][inner sep=0.75pt]    {$u=u_{0}$};
\draw (454,215.4) node [anchor=north west][inner sep=0.75pt]    {$v=v_{3}$};
\draw (280,241.4) node [anchor=north west][inner sep=0.75pt]    {$\ell ( p) =\theta n$};
\draw (199,63.4) node [anchor=north west][inner sep=0.75pt]    {$\beta _{0}$};
\draw (426,76.4) node [anchor=north west][inner sep=0.75pt]    {$\beta _{3}$};
\draw (249,148.4) node [anchor=north west][inner sep=0.75pt]    {$\alpha _{1}$};
\draw (293,148.4) node [anchor=north west][inner sep=0.75pt]    {$\alpha _{2}$};
\draw (356,151.4) node [anchor=north west][inner sep=0.75pt]    {$\alpha _{3}$};
\draw (273,95.4) node [anchor=north west][inner sep=0.75pt]    {$\beta _{1}$};
\draw (326,95.4) node [anchor=north west][inner sep=0.75pt]    {$\beta _{2}$};
\draw (286,218.4) node [anchor=north west][inner sep=0.75pt]    {$v_{1}$};
\draw (264,218.4) node [anchor=north west][inner sep=0.75pt]    {$u_{1}$};
\draw (241,218.4) node [anchor=north west][inner sep=0.75pt]    {$v_{0}$};
\draw (307,218.4) node [anchor=north west][inner sep=0.75pt]    {$u_{2}$};
\draw (342,218.4) node [anchor=north west][inner sep=0.75pt]    {$v_{2}$};
\draw (379,219.4) node [anchor=north west][inner sep=0.75pt]    {$u_{3}$};
\draw (473,159.4) node [anchor=north west][inner sep=0.75pt]    {$\leq \epsilon n$};
\draw (143,173.4) node [anchor=north west][inner sep=0.75pt]    {$\epsilon n\geq $};

\end{tikzpicture}
    \caption{Lemma \ref{fractalMorseLem}: decompose the path $\alpha$ from $x$ to $y$ into components $\alpha_i$ contained in  $\mathcal N_{r/2}(p)$ (the grayed region). }
    \label{fig:fractalMorseLem}
\end{figure}
\medskip

By Lemma \ref{CharMorseLem}, we shall understand Morse quasi-geodesics from the view point of sublinear contraction. Fix a sublinear contraction function $\kappa_0$ in the sequel.
Before presenting the full version, we  note the following special case, which   gives a  generalization of middle recurrence (Definition \ref{MidRecurrentDefn}). 

\begin{lem}\label{fractalMorseCor}
For any $\theta\in (0,1]$ and $c, \lambda\ge 1$, there exists a constant $r=r(\kappa_0, \theta, c, \lambda)>0$ with the following property.
Let $\gamma$ be a $\kappa_0$-contracting $c$-quasi-geodesic  in $X$ and $\alpha$ a $\lambda$-Lipschitz path  from $x$ to $y$ on $\gamma$. Consider a $\theta$-interval $p$ of $[x,y]_\gamma$. For any $R>4r$, set $N=\lfloor\ell(p)/6R\rfloor$. Then there are points $z_1,\ldots,z_N\in p$ so that $d_X(z_i,\alpha)\le r$ for $1\le i\le N$ and $d_X(z_i,z_{i+1})\ge R$ for $1\le i<N$. 
\end{lem}

\begin{rem}\label{QTRem}
This   improves the middle recurrence in the following sense: the Lipschitz path   intersects   \emph{every} $(1-2\theta)$-interval (instead of a particular $\theta$-middle one) for \emph{any} admissible values of $\theta$ (instead of a fixed one),  in a linear number of points comparable with its length.

This could be compared with a similar property in quasi-trees.  Recall that a geodesic metric space being a quasi-tree is characterized by the   \textit{$\Delta$-bottleneck property} in \cite{Man05} : any path from $x$ to $y$ in $X$ intersects    the $\Delta$-ball around any point on  $[x,y]_X$. 
\end{rem}
A   direct but   worth-noting corollary is as follows.
\begin{cor}\label{fractalMorseCor1}
For any $\theta\in (0,1]$ and $\lambda\ge 1$, there exists a constant $r=r(\kappa_0,\theta, \lambda)>0$ with the following property.
Let $\gamma$ be a $\kappa_0$-contracting geodesic  in $X$ and $\alpha$ a $\lambda$-Lipschitz path  from $x$ to $y$ on $\gamma$. Then $\alpha$ intersects the $r$-neighborhood of any $\theta$-interval of $[x,y]_\gamma$. 
\end{cor}

We prove the following general version, which  shall  find its use in the next subsection. 
\begin{lem}\label{fractalMorseLem}
For any $\epsilon>0, \theta\in (0,1]$ and $c,\lambda\ge 1$, there exists a constant $r=r(\kappa_0,\epsilon, \theta, c, \lambda)>0$ with the following property.

Consider a $\kappa_0$-contracting $c$-quasi-geodesic  $\gamma$  in $X$ and a $\lambda$-Lipschitz path $\alpha$ from $x$ to $y$ (not necessarily on $\gamma$). Assume that $u,v\in \gamma$ are some shortest $d_X$-projection points of  $x$, $y$ respectively so that $$d_X(x,u), d_X(y,v)\le \epsilon d_X(u,v).$$ Let $p$  be a $\theta$-interval of $[u,v]_\gamma$. For any $R>4r$, set $N=\lfloor\ell(p)/6R\rfloor$. Then we may find linearly ordered points $z_1,\ldots,z_N\in p$ and $w_1,\ldots,w_N\in \alpha$ so that $d_X(z_i,w_i)\le r$ for $1\le i\le N$ and $d_X(z_i,z_{i+1})\ge R$ for $1\le i<N$. 
\end{lem}
In the proof, we shall use the following result in the proof of \cite[Lemma 3.6, Eq.(1)]{CashenMackay2019}. Roughly speaking, the shortest projection to a subsegment is sublinearly close to the endpoints of it. 

\begin{lem}\label{lem:endpoint-projection}
Let $\gamma$ be a $\kappa_0$--contracting geodesic. Then there exist a
constant $C_0>0$ and a sublinear function $\kappa_0'\asymp\kappa_0$ with the
following property.

Let $p=[a,b]_\gamma$ be a subsegment and let $x\in X$. Suppose that
$\proj_\gamma^X(x)\cap p=\varnothing$ and that $\proj_\gamma^X(x)$ lies in
the component of $\gamma\setminus (a,b)$ containing $a$. Then
\[
\mathrm{diam}_X\bigl(\{a\}\cup\proj_p^X(x)\bigr)
 \le C_0+\kappa_0'\bigl(d_X(x,\gamma)\bigr).
\]
The analogous conclusion holds at $b$ when $\proj_\gamma^X(x)$ lies on
the other side of $p$.
\end{lem}

\begin{proof}[Proof of Lemma \ref{fractalMorseLem}]
By Lemma \ref{MorseSubpath}, $[u,v]_\gamma$ is quantitatively Morse and any geodesic with the same endpoints is Morse. It is easy to see that the conclusion remains valid up to perturbing $\gamma$ in a fixed neighborhood. Therefore, we may assume without loss of generality that $\gamma$ is a geodesic in the proof.

Choose $n_0$ so that $C_0+\kappa_0'(\epsilon n)\le \theta n/16$ for $n\ge n_0$. Set \begin{equation}\label{theta0Eq}
\theta_0= \frac{\theta}{4(\lambda (1+2\epsilon) +\epsilon +1-\theta)}    
\end{equation} Let $r_0=r_0(\kappa_0,\theta_0/2)>0$ be given by Lemma \ref{LipProj2Morse} such that if the distance of a path $\beta$ to $p$ lies between $r_0/2$ and $R$, then $$\mathrm{diam}_X\bigl(\proj_p^X(\beta)\bigr)\le \theta_0 (\ell (\beta) +R)/2.$$
Increase $r_0$, if necessary, and set $r=2r_0$ so that $6r>\theta n_0$. Denote $n=d_X(u,v)$. Assume $n>n_0$; otherwise $\lfloor\ell(p)/6R\rfloor=0$ and the conclusion is empty. By assumption,  $d_X(x,y)\le d_X(u,v) +d_X(u,x)+d_X(v,y) \le (1+2\epsilon) n$, so  $$\ell(\alpha)\le \lambda d_X(x,y)\le \lambda (1+2\epsilon) n.$$

First of all, observe that the $\lambda$-Lipschitz path $\alpha$ intersects $\mathcal N_{r_0/2}(p)$. Indeed, if not, then $\alpha\cap \mathcal N_{r_0/2}(p)=\emptyset$. By assumption, the endpoints $x$ and $y$ of $\alpha$ are within a distance at most $(\epsilon n +(1-\theta) n)$ to $p^-$ and $p^+$ respectively, so we obtain by  Lemmas \ref{LipProj2Morse} and \ref{lem:endpoint-projection} that   
\begin{equation}\label{lengthofpEQ}
\begin{aligned}
\ell(p)\;\le &\;\;\mathrm{diam}_X\bigl(\proj_p^X(\alpha)\bigr)+\mathrm{diam}_X(\proj_p^X(x)\cup\{p^-\})+\mathrm{diam}_X(\proj_p^X(y)\cup\{p^+\})\\
\le &\;\theta_0(\ell(\alpha)+\epsilon n +(1-\theta) n)+2(C_0+\kappa_0'(\epsilon n))\\
\le &\;\theta_0(\ell(\alpha)+\epsilon n +(1-\theta) n)+\theta n/8\\
\le &\;n \theta_0 (\lambda (1+2\epsilon) +\epsilon+1-\theta)+\theta n/8 \,<\, n\theta/2  
\end{aligned}
\end{equation}
where the last inequality follows by the choice of $\theta_0$ in Eq. (\ref{theta0Eq}). This contradicts $\ell(p)=\theta n$, whence $\alpha\cap \mathcal N_{r_0/2}(p)\ne\emptyset$.
\medskip

\noindent We now decompose $\alpha=\beta_0\alpha_1\beta_1\cdots \alpha_N \beta_{N}$ for some $N\ge 1$, where  $\{\beta_i:0\le i\le N\}$ are all the components  of $\alpha\setminus \mathcal N_{r_0/2}(p)$. In addition,
\begin{itemize}
    \item For each $1\le i\le N-1$, $\beta_i$ intersects $\mathcal N_{r_0/2}(p)$ only at its two endpoints;
    \item 
    the initial and terminal ones $\beta_0, \beta_{N}$ intersect $\mathcal N_{r_0/2}(p)$ in one endpoint, while their other endpoints have distance at most $(\epsilon n+ (1-\theta)n)$ to $p$.
\end{itemize} 
See Fig. \ref{fig:fractalMorseLem} for the illustration with $N=3$.
\medskip

The idea of the proof is to show that at least one half of $p$ is covered, in the correct order, by subsegments which stay uniformly close to the subpaths $\alpha_i$.

\medskip
Orient $p$ from $p^-$ to $p^+$. For $1\le i<N$, choose $u_i,v_i\in p$ as projection points of the initial and terminal endpoints of $\beta_i$, respectively. Set $u_0=p^-$ and $v_N=p^+$, choose $v_0$ as a projection point of the terminal endpoint of $\beta_0$, and choose $u_N$ as a projection point of the initial endpoint of $\beta_N$. If $1\le i<N$ and $\ell(\beta_i)>r_0/2$, then
$$d_X(u_i,v_i)\le \mathrm{diam}_X\big(\proj_p^X(\beta_i)\big)\le \theta_0 \ell(\beta_i).$$
By the same argument as in Eq. (\ref{lengthofpEQ}), for $i\in\{0,N\}$ we have
$$d_X(u_i,v_i)\le \theta_0(\ell(\beta_i)+\epsilon n+(1-\theta)n)/2+C_0+\kappa_0'(\epsilon n).$$
Taking into account $\sum_{i=0}^N \ell(\beta_i)\le \ell(\alpha)\le \lambda (1+2\epsilon) n$, we deduce that  
\begin{equation}\label{projectedlengthEq}
\begin{aligned}
\sum_{i\in\{0,N\}\,\text{or}\,\ell(\beta_i)>r_0/2} d_X(u_i,v_i) &\le \left(\theta_0 \sum_{i=0}^N \ell(\beta_i)\right) + \theta_0 (\epsilon n+ (1-\theta)n)+2(C_0+\kappa_0'(\epsilon n)) \\
&\le \theta_0 \lambda (1+2\epsilon) n + \theta_0 (\epsilon n+ (1-\theta)n)+\theta n/8 \\
&<\ell(p)/2.       
\end{aligned} 
\end{equation} 
Here the last inequality follows from the choices of $\theta_0$ and $n_0$.

\medskip
\noindent If $\alpha_i$ and $\alpha_{i+1}$ are separated by an interior component $\beta_i$ with $\ell(\beta_i)\le r_0/2$, concatenate them across $\beta_i$. After re-indexing, each resulting connector $\alpha_i$ lies in $\mathcal N_{r_0}(p)$ and joins the $r_0$-neighborhoods of $v_{i-1}$ and $u_i$. Orienting $p$ as above and deleting overlaps and backtracking by the running-maximum procedure, we obtain pairwise disjoint, ordered subsegments
$$J_i\subseteq [v_{i-1},u_i]_p$$
such that Eq.~\eqref{projectedlengthEq} gives
\begin{equation}\label{eq:alpha-uv}
\sum_i\ell(J_i)>\ell(p)/2.
\end{equation}
For every $z\in J_i$, connectedness applied to the two sets
$\mathcal N_{r_0}([p^-,z]_p)$ and $\mathcal N_{r_0}([z,p^+]_p)$ gives a point $w\in\alpha_i$ with $d_X(z,w)\le 2r_0=r$. Choosing the first entry of $\alpha_i$ into the second set as $z$ moves from left to right makes the corresponding points $w$ linearly ordered on $\alpha_i$.

\medskip
Choose a maximal $R$-separated subset $(z_j)_{j=1}^{M}$ of $\bigcup_iJ_i$, ordered from left to right. Its $R$-neighborhood in $p$ covers $\bigcup_iJ_i$, and hence Eq.~\eqref{eq:alpha-uv} gives
\[
M\ge \frac{\sum_i\ell(J_i)}{2R}>\frac{\ell(p)}{4R}.
\]
For each $z_j\in J_i$, choose the corresponding first crossing point $w_j\in\alpha_i$ as above. Since the $J_i$ and the $\alpha_i$ occur in the same order, the points $w_j$ are linearly ordered on $\alpha$. Therefore, setting
\[
N_0=\left\lfloor\frac{\ell(p)}{6R}\right\rfloor,
\]
we may retain $N_0$ pairs so that the points $z_j$ are linearly ordered
and $R$-separated on $p$, the points $w_j$ are linearly ordered on
$\alpha$, and $d_X(z_j,w_j)\le r.$
This completes the proof.
\end{proof}

\subsection{Sublinear projection tracking along WPD elements} Let $G$ act  by isometry on a metric space $X$ and $\pi: G\to Go$  the orbital map for a fixed basepoint $o\in X$.  The notion of WPD elements  by Bestvina-Fujiwara is    generalized  in a broader setting by Sisto \cite{Sisto16}.

\begin{defn}\label{AcylinAlongSubsetDefn}
We say that the action $G\act X$ is \textit{acylindrical} along a subgroup $H$  if for every $r>0$ there exists $R>0$ so that for any $x, y \in \pi(H)$
with $d(x, y) \ge R$, we have $$|\{g \in G: d(x, gx), d(y, gy) \le r\}|<\infty.$$    
\end{defn}
  
\begin{rem}
The definition does not depend on the choice of basepoint $o$, and is also invariant under isometry. 
This generalizes the notion of WPD elements as follows: if $f$ is a Morse isometry, then $f$ is a WPD element if and only if the action $G\act X$ is {acylindrical} along the subgroup $H:=\langle f\rangle$. If $X$ is hyperbolic, this is proved by \cite[Lemma 6.4]{DGO} that any larger $N$ in Eq. (\ref{WPDDefn}) works. The  proof uses only the Morse property of $\pi(H)$, so it works for a general metric space $X$.  

\end{rem}


The following geometric separation result proved by Sisto underlies the proof of the main result. See \cite{AM24} for a further development of such kind of results.

\begin{lem}\cite[Lemma 3.4]{Sisto16}\label{GeomSepAlongH}
Suppose that   $G$ acts acylindrically along a subgroup $H$ on a metric space $X$. Assume in addition that $\pi\lvert_H$ is a proper map.
Then for $r \ge 0$ there exists $R>0$ so that for any $x, y\in \pi(H)$ with $d_X(x, y)\ge R$  the
following holds. For each $D>0$ there exists $M>0$ so that for any $a\in \pi^{-1}(B_r(x))$ and $b\in \pi^{-1}(B_r(y))$,
\[
d_S(a,b)\le D \quad\Longrightarrow\quad d_S(a,H)\le M.
\]
\end{lem}

We divide  \ref{SubTrackThm} into Lemma  \ref{SublTrackInAcyl} and Lemma \ref{SSublContrAcyl} in the remainder of this subsection. Let us begin by recalling their setup.

\begin{enumerate}
    \item 
    Assume that $G$ is a group generated by a finite symmetric set $S$. 
    \item 
    Let $f\in G$ act  as a Morse isometry on $X$: this means by Definition \ref{SublinearContrDefn} that $\pi(H)$ is a Morse quasi-geodesic with $H=\langle f\rangle$.   
    \item 
    Assume that    $G$ acts acylindrically along  $H$ on $X$; equivalently, $f$ is a WPD isometry as in (\ref{WPDDefn}). 

\end{enumerate}

\begin{figure}
    \centering

\tikzset{every picture/.style={line width=0.75pt}} 

\begin{tikzpicture}[x=0.75pt,y=0.75pt,yscale=-1,xscale=1]

\draw    (56.5,199) -- (368.5,198) ;
\draw    (356.5,27) -- (353.5,198) ;
\draw [shift={(353.5,198)}, rotate = 91.01] [color={rgb, 255:red, 0; green, 0; blue, 0 }  ][fill={rgb, 255:red, 0; green, 0; blue, 0 }  ][line width=0.75]      (0, 0) circle [x radius= 3.35, y radius= 3.35]   ;
\draw [color={rgb, 255:red, 74; green, 144; blue, 226 }  ,draw opacity=1 ][line width=1.5]    (74.5,199) .. controls (68,162.5) and (69.5,49) .. (113.5,51) .. controls (157.5,53) and (112.5,216) .. (151.5,171) .. controls (190.5,126) and (182.5,205) .. (212.5,168) .. controls (242.5,131) and (237.5,173) .. (268.5,178) .. controls (299.5,183) and (304.5,117) .. (320.5,172) .. controls (336.5,227) and (317.5,17) .. (356.5,28) ;
\draw [shift={(356.5,28)}, rotate = 15.75] [color={rgb, 255:red, 74; green, 144; blue, 226 }  ,draw opacity=1 ][fill={rgb, 255:red, 74; green, 144; blue, 226 }  ,fill opacity=1 ][line width=1.5]      (0, 0) circle [x radius= 4.36, y radius= 4.36]   ;
\draw [shift={(74.5,199)}, rotate = 259.9] [color={rgb, 255:red, 74; green, 144; blue, 226 }  ,draw opacity=1 ][fill={rgb, 255:red, 74; green, 144; blue, 226 }  ,fill opacity=1 ][line width=1.5]      (0, 0) circle [x radius= 4.36, y radius= 4.36]   ;
\draw  [fill={rgb, 255:red, 155; green, 155; blue, 155 }  ,fill opacity=0.52 ][dash pattern={on 4.5pt off 4.5pt}] (119,197.75) .. controls (119,186.84) and (127.84,178) .. (138.75,178) .. controls (149.66,178) and (158.5,186.84) .. (158.5,197.75) .. controls (158.5,208.66) and (149.66,217.5) .. (138.75,217.5) .. controls (127.84,217.5) and (119,208.66) .. (119,197.75) -- cycle ;
\draw  [fill={rgb, 255:red, 155; green, 155; blue, 155 }  ,fill opacity=0.52 ][dash pattern={on 4.5pt off 4.5pt}] (179,198.75) .. controls (179,187.84) and (187.84,179) .. (198.75,179) .. controls (209.66,179) and (218.5,187.84) .. (218.5,198.75) .. controls (218.5,209.66) and (209.66,218.5) .. (198.75,218.5) .. controls (187.84,218.5) and (179,209.66) .. (179,198.75) -- cycle ;
\draw  [fill={rgb, 255:red, 155; green, 155; blue, 155 }  ,fill opacity=0.52 ][dash pattern={on 4.5pt off 4.5pt}] (248.75,197.75) .. controls (248.75,186.84) and (257.59,178) .. (268.5,178) .. controls (279.41,178) and (288.25,186.84) .. (288.25,197.75) .. controls (288.25,208.66) and (279.41,217.5) .. (268.5,217.5) .. controls (257.59,217.5) and (248.75,208.66) .. (248.75,197.75) -- cycle ;
\draw  [fill={rgb, 255:red, 155; green, 155; blue, 155 }  ,fill opacity=0.52 ][dash pattern={on 4.5pt off 4.5pt}] (306,196.75) .. controls (306,185.84) and (314.84,177) .. (325.75,177) .. controls (336.66,177) and (345.5,185.84) .. (345.5,196.75) .. controls (345.5,207.66) and (336.66,216.5) .. (325.75,216.5) .. controls (314.84,216.5) and (306,207.66) .. (306,196.75) -- cycle ;
\draw    (198.75,198.75) -- (268.5,197.75) ;
\draw [shift={(268.5,197.75)}, rotate = 359.18] [color={rgb, 255:red, 0; green, 0; blue, 0 }  ][fill={rgb, 255:red, 0; green, 0; blue, 0 }  ][line width=0.75]      (0, 0) circle [x radius= 3.35, y radius= 3.35]   ;
\draw [shift={(198.75,198.75)}, rotate = 359.18] [color={rgb, 255:red, 0; green, 0; blue, 0 }  ][fill={rgb, 255:red, 0; green, 0; blue, 0 }  ][line width=0.75]      (0, 0) circle [x radius= 3.35, y radius= 3.35]   ;
\draw    (268.5,197.75) -- (268.5,178) ;
\draw [shift={(268.5,178)}, rotate = 270] [color={rgb, 255:red, 0; green, 0; blue, 0 }  ][fill={rgb, 255:red, 0; green, 0; blue, 0 }  ][line width=0.75]      (0, 0) circle [x radius= 3.35, y radius= 3.35]   ;
\draw    (198.75,198.75) -- (198.75,179) ;
\draw [shift={(198.75,179)}, rotate = 270] [color={rgb, 255:red, 0; green, 0; blue, 0 }  ][fill={rgb, 255:red, 0; green, 0; blue, 0 }  ][line width=0.75]      (0, 0) circle [x radius= 3.35, y radius= 3.35]   ;

\draw (31,210.4) node [anchor=north west][inner sep=0.75pt]    {$1\in \mathrm{Pr}_{H}^{S}( x)$};
\draw (365,9.4) node [anchor=north west][inner sep=0.75pt]    {$xo$};
\draw (350,209.4) node [anchor=north west][inner sep=0.75pt]    {$y\in \mathrm{Pr}_{H}^{X}( x)$};
\draw (361,101.4) node [anchor=north west][inner sep=0.75pt]    {$\leq \lambda n$};
\draw (362,178.4) node [anchor=north west][inner sep=0.75pt]    {$yo$};
\draw (57,177.4) node [anchor=north west][inner sep=0.75pt]    {$o$};
\draw (157,49.4) node [anchor=north west][inner sep=0.75pt]    {$\ell ( \pi [ 1,x]_{S}) \leq \lambda n$};
\draw (157,74.4) node [anchor=north west][inner sep=0.75pt]    {$d_{X}( o,yo) =l \geq \epsilon n$};
\draw (261,155.4) node [anchor=north west][inner sep=0.75pt]    {$g_{i} o$};
\draw (184,150.4) node [anchor=north west][inner sep=0.75pt]    {$g_{i-1} o$};
\draw (257,219.4) node [anchor=north west][inner sep=0.75pt]    {$m_{i} o$};
\draw (178,219.4) node [anchor=north west][inner sep=0.75pt]    {$m_{i-1} o$};

\end{tikzpicture}
    \caption{Lemma \ref{SublTrackInAcyl}}
    \label{fig:SublTrackInAcyl}
\end{figure}

We first prove that the $d_S$-projection sublinearly tracks the $d_X$-projection along $H$. This is analogous to Lemma \ref{SublTrackInProper} in the proper action.  
\begin{lem}\label{SublTrackInAcyl}
Let $f\in G$ be a Morse WPD  isometry on $X$. Then 
there exists a sublinear function $\ktrack$ so that 
the $d_S$-projection   $\ktrack$-tracks $d_X$-projection along  $H=\langle f\rangle$. 
\end{lem}
\begin{proof}
Note that $d_S$ and $d_X$ projections are invariant under isometry. Given $x\in G$, we may assume without loss of generality that the   identity $1$ is contained in the shortest $d_S$-projection $\proj^S_{H}(x)$, and $y$ is in the $d_X$-projection $\proj^X_{H}(x)$. That is to say, $1$ is a shortest $d_S$-projection point of $x$ to $H$, and $yo$ is a shortest $X$-projection point of $xo$ to $\pi(H)$.

As $\pi\lvert_H$ is a $\lambda$-quasi-isometric embedding for some $\lambda>1$, we have   $$\lambda^{-1}\|y\|_S -\lambda \le d_X(o,yo)\le  \lambda \|y\|_S.$$
It suffices  to prove that   $$\frac{d_X(o,yo)}{\|x^{-1}y\|_S}\longrightarrow  0\quad \text{ as }\quad \|x^{-1}y\|_S\to \infty.$$
Indeed, if $n:=\|x^{-1}y\|_S$ and $d_X(o,yo)=o(n)$, then $\|y\|_S=o(n)$, while $\|x\|_S=d_S(x,H)\le n\le \|x\|_S+\|y\|_S$. Thus $n\sim d_S(x,H)$ and the required tracking estimate follows. We argue by contradiction. Let $l:=d_X(o,yo)$. There exists $\epsilon>0$ so that $l\ge \epsilon n$ holds for infinitely many pairs $(x,y)$ with $n\to\infty$. 

\medskip
Consider the  path $\alpha:=\pi([1,x]_S)$ labeled by $[1,x]_S$ in  $X$. As $\|x\|_S\le n$, $\alpha$ has length at most $\lambda n$. Let $p$  be the $(1/2,1)$-interval of $[o,yo]_{\pi H}$; that is, $p$ is the second half of it, so $\ell(p) = l/2$. The endpoint $xo$ of   $\alpha$ projects to $yo$ with a distance at most $$d_X(xo,yo)\le \lambda \|x^{-1}y\|_S\le \lambda n\le  \epsilon_0 d_X(o,yo)$$
where $\epsilon_0:=\lambda/\epsilon$. See Fig. \ref{fig:SublTrackInAcyl} for the illustration. 

\medskip
\noindent We then apply Lemma \ref{fractalMorseLem} to the $\lambda$-Lipschitz path $\alpha$ and the $(1/2,1)$-interval $p$, with parameter $(\epsilon_0,\theta:=1/2,\lambda)$. Let $r=r(\epsilon_0,1/2,\lambda)$ be the given constant. 

Let $R=R(r)>4r$ be the constant  given by Lemma \ref{GeomSepAlongH}. Set $N= \lfloor \ell(p) /6R\rfloor-1$. By Lemma \ref{fractalMorseLem}, the path $\alpha=\pi([1,x]_S)$ intersects a sequence of $r$-balls around  points $m_io$ on $p$ where $m_i\in H$ and $1\le i\le N+1$ so that  $$d_X(o,m_io)\ge d_X(o,yo)/2,\qquad d_X(m_{i}o,m_{i+1}o)\ge R$$ Here we use the choice of $p$ as the second half of $[o,yo]_{\pi H}$.

Moreover, we may choose a linearly ordered sequence of points $g_i\in [1,x]_S$ so that $d_X(g_io, m_io)\le r$. Since  $1$ is a shortest $d_S$-projection point of $x$ to $H$, for every $a\in H$ the triangle inequality gives
$$d_S(g_i,a)\ge d_S(x,a)-d_S(x,g_i)\ge d_S(x,1)-d_S(x,g_i)=d_S(g_i,1).$$
Thus $d_S(g_i,H)\ge d_S(g_i,1)$ for each $1\le i\le N+1$. 

For each $D$ to be given below, we have  $M$ given by Lemma \ref{GeomSepAlongH} so that if $d_S(g_i, H)> M$ for some $1\le i\le N$, then $d_S(g_i,g_{i+1})>D$. 

\medskip
\noindent
Now for infinitely many $(x,y)$  and $g_i\in [1,x]_S$ chosen as above,  we claim that  $d_S(g_i,H)> M$ for each $1\le i\le N+1$. Indeed, if $d_S(g_i,H)\le M$ for some $i$, then $d_S(g_i,1)\le M$ and thus $d_X(g_io, o)\le \lambda M$ by $\lambda$-Lipschitz property of $\pi$. Taking into account $d_X(m_io,g_io)\le r$, we have $d_X(m_io,o)\le r+\lambda M$. This contradicts the inequality $d_X(o,m_io)>d_X(o,yo)/2$ as $l=d_X(o,yo)\to\infty$.   

\medskip
\noindent
To conclude the proof, since $[1,x]_S$ is a word geodesic, we obtain   
$$
d_S(1,x)\ge \sum_{i=1}^N d_S(g_i,g_{i+1})>ND$$
Recall that $N\ge \ell(p)/6R - 1 \ge \epsilon n/12R-1\ge \epsilon n/24R$ when $n\gg 1$. Choose $D$ satisfying Lemma \ref{GeomSepAlongH} large enough so that $\frac{\epsilon D}{24R}>1$. This gives a contradiction: $d_S(1,x)\ge \frac{\epsilon D}{24R}\cdot n> n$. The proof is complete.
\end{proof}

We now establish that, under a stronger assumption,  the pullback of  the shortest $d_X$-projection to $\pi(H)$    has  strongly sublinear contraction in word metric as in Definition \ref{SublinearContrDefn}. We do not know whether $\pi(H)$  could be weakened as a Morse quasi-geodesic as in Lemma \ref{SublTrackInAcyl}.  

\begin{figure}
    \centering

\tikzset{every picture/.style={line width=0.75pt}} 

\begin{tikzpicture}[x=0.75pt,y=0.75pt,yscale=-1,xscale=1]

\draw    (38.5,227) -- (263.5,225) ;
\draw    (339,226) -- (544.5,225) ;
\draw    (350.5,100) -- (352.5,224) ;
\draw [shift={(352.5,224)}, rotate = 89.08] [color={rgb, 255:red, 0; green, 0; blue, 0 }  ][fill={rgb, 255:red, 0; green, 0; blue, 0 }  ][line width=0.75]      (0, 0) circle [x radius= 3.35, y radius= 3.35]   ;
\draw    (545,145) -- (544.5,225) ;
\draw [shift={(544.5,225)}, rotate = 90.36] [color={rgb, 255:red, 0; green, 0; blue, 0 }  ][fill={rgb, 255:red, 0; green, 0; blue, 0 }  ][line width=0.75]      (0, 0) circle [x radius= 3.35, y radius= 3.35]   ;
\draw  [dash pattern={on 4.5pt off 4.5pt}]  (93.5,226) -- (116.5,137) ;
\draw [shift={(116.5,137)}, rotate = 284.49] [color={rgb, 255:red, 0; green, 0; blue, 0 }  ][fill={rgb, 255:red, 0; green, 0; blue, 0 }  ][line width=0.75]      (0, 0) circle [x radius= 3.35, y radius= 3.35]   ;
\draw [shift={(93.5,226)}, rotate = 284.49] [color={rgb, 255:red, 0; green, 0; blue, 0 }  ][fill={rgb, 255:red, 0; green, 0; blue, 0 }  ][line width=0.75]      (0, 0) circle [x radius= 3.35, y radius= 3.35]   ;
\draw [color={rgb, 255:red, 0; green, 0; blue, 0 }  ,draw opacity=1 ][line width=1.5]    (86.5,116) .. controls (109.5,151) and (199,155) .. (239,125) ;
\draw [shift={(239,125)}, rotate = 323.13] [color={rgb, 255:red, 0; green, 0; blue, 0 }  ,draw opacity=1 ][fill={rgb, 255:red, 0; green, 0; blue, 0 }  ,fill opacity=1 ][line width=1.5]      (0, 0) circle [x radius= 4.36, y radius= 4.36]   ;
\draw [shift={(86.5,116)}, rotate = 56.69] [color={rgb, 255:red, 0; green, 0; blue, 0 }  ,draw opacity=1 ][fill={rgb, 255:red, 0; green, 0; blue, 0 }  ,fill opacity=1 ][line width=1.5]      (0, 0) circle [x radius= 4.36, y radius= 4.36]   ;
\draw  [dash pattern={on 4.5pt off 4.5pt}]  (125,226) -- (146.5,144) ;
\draw [shift={(146.5,144)}, rotate = 284.69] [color={rgb, 255:red, 0; green, 0; blue, 0 }  ][fill={rgb, 255:red, 0; green, 0; blue, 0 }  ][line width=0.75]      (0, 0) circle [x radius= 3.35, y radius= 3.35]   ;
\draw [shift={(125,226)}, rotate = 284.69] [color={rgb, 255:red, 0; green, 0; blue, 0 }  ][fill={rgb, 255:red, 0; green, 0; blue, 0 }  ][line width=0.75]      (0, 0) circle [x radius= 3.35, y radius= 3.35]   ;
\draw  [dash pattern={on 4.5pt off 4.5pt}]  (165,226) -- (186.5,143) ;
\draw [shift={(186.5,143)}, rotate = 284.52] [color={rgb, 255:red, 0; green, 0; blue, 0 }  ][fill={rgb, 255:red, 0; green, 0; blue, 0 }  ][line width=0.75]      (0, 0) circle [x radius= 3.35, y radius= 3.35]   ;
\draw [shift={(165,226)}, rotate = 284.52] [color={rgb, 255:red, 0; green, 0; blue, 0 }  ][fill={rgb, 255:red, 0; green, 0; blue, 0 }  ][line width=0.75]      (0, 0) circle [x radius= 3.35, y radius= 3.35]   ;
\draw  [dash pattern={on 4.5pt off 4.5pt}]  (202,226) -- (225.5,133) ;
\draw [shift={(225.5,133)}, rotate = 284.18] [color={rgb, 255:red, 0; green, 0; blue, 0 }  ][fill={rgb, 255:red, 0; green, 0; blue, 0 }  ][line width=0.75]      (0, 0) circle [x radius= 3.35, y radius= 3.35]   ;
\draw [shift={(202,226)}, rotate = 284.18] [color={rgb, 255:red, 0; green, 0; blue, 0 }  ][fill={rgb, 255:red, 0; green, 0; blue, 0 }  ][line width=0.75]      (0, 0) circle [x radius= 3.35, y radius= 3.35]   ;
\draw  [fill={rgb, 255:red, 128; green, 128; blue, 128 }  ,fill opacity=0.42 ][dash pattern={on 4.5pt off 4.5pt}] (362.5,226) .. controls (362.5,217.16) and (369.66,210) .. (378.5,210) .. controls (387.34,210) and (394.5,217.16) .. (394.5,226) .. controls (394.5,234.84) and (387.34,242) .. (378.5,242) .. controls (369.66,242) and (362.5,234.84) .. (362.5,226) -- cycle ;
\draw  [fill={rgb, 255:red, 128; green, 128; blue, 128 }  ,fill opacity=0.42 ][dash pattern={on 4.5pt off 4.5pt}] (409.75,225.5) .. controls (409.75,216.66) and (416.91,209.5) .. (425.75,209.5) .. controls (434.59,209.5) and (441.75,216.66) .. (441.75,225.5) .. controls (441.75,234.34) and (434.59,241.5) .. (425.75,241.5) .. controls (416.91,241.5) and (409.75,234.34) .. (409.75,225.5) -- cycle ;
\draw  [fill={rgb, 255:red, 128; green, 128; blue, 128 }  ,fill opacity=0.42 ][dash pattern={on 4.5pt off 4.5pt}] (457.5,225) .. controls (457.5,216.16) and (464.66,209) .. (473.5,209) .. controls (482.34,209) and (489.5,216.16) .. (489.5,225) .. controls (489.5,233.84) and (482.34,241) .. (473.5,241) .. controls (464.66,241) and (457.5,233.84) .. (457.5,225) -- cycle ;
\draw  [fill={rgb, 255:red, 128; green, 128; blue, 128 }  ,fill opacity=0.42 ][dash pattern={on 4.5pt off 4.5pt}] (503.5,225) .. controls (503.5,216.16) and (510.66,209) .. (519.5,209) .. controls (528.34,209) and (535.5,216.16) .. (535.5,225) .. controls (535.5,233.84) and (528.34,241) .. (519.5,241) .. controls (510.66,241) and (503.5,233.84) .. (503.5,225) -- cycle ;
\draw  [line width=0.75]  (147.5,149) .. controls (147.36,153.53) and (149.56,155.86) .. (154.09,156) -- (154.09,156) .. controls (160.56,156.2) and (163.72,158.56) .. (163.59,163.09) .. controls (163.72,158.56) and (167.03,156.4) .. (173.5,156.59)(170.59,156.5) -- (173.5,156.59) .. controls (178.03,156.73) and (180.36,154.53) .. (180.5,150) ;
\draw    (425.75,225.5) -- (473.5,225) ;
\draw [shift={(473.5,225)}, rotate = 359.4] [color={rgb, 255:red, 0; green, 0; blue, 0 }  ][fill={rgb, 255:red, 0; green, 0; blue, 0 }  ][line width=0.75]      (0, 0) circle [x radius= 3.35, y radius= 3.35]   ;
\draw [shift={(425.75,225.5)}, rotate = 359.4] [color={rgb, 255:red, 0; green, 0; blue, 0 }  ][fill={rgb, 255:red, 0; green, 0; blue, 0 }  ][line width=0.75]      (0, 0) circle [x radius= 3.35, y radius= 3.35]   ;
\draw  [line width=0.75]  (110.5,141) .. controls (105.99,139.81) and (103.14,141.48) .. (101.95,145.99) -- (96.98,164.94) .. controls (95.29,171.39) and (92.18,174.02) .. (87.67,172.83) .. controls (92.18,174.02) and (93.59,177.83) .. (91.9,184.28)(92.66,181.38) -- (84.5,212.45) .. controls (83.32,216.96) and (84.99,219.81) .. (89.5,221) ;
\draw    (289,52) -- (287.5,259) ;
\draw [color={rgb, 255:red, 74; green, 144; blue, 226 }  ,draw opacity=1 ][line width=1.5]    (350.5,100) .. controls (380.5,102) and (362.5,245) .. (390.5,203) .. controls (418.5,161) and (414.5,219) .. (430.5,208) .. controls (446.5,197) and (449.5,192) .. (464.5,208) .. controls (479.5,224) and (481.5,179) .. (507.5,203) .. controls (533.5,227) and (529.5,174) .. (545,145) ;
\draw [shift={(545,145)}, rotate = 298.12] [color={rgb, 255:red, 74; green, 144; blue, 226 }  ,draw opacity=1 ][fill={rgb, 255:red, 74; green, 144; blue, 226 }  ,fill opacity=1 ][line width=1.5]      (0, 0) circle [x radius= 4.36, y radius= 4.36]   ;
\draw [shift={(350.5,100)}, rotate = 3.81] [color={rgb, 255:red, 74; green, 144; blue, 226 }  ,draw opacity=1 ][fill={rgb, 255:red, 74; green, 144; blue, 226 }  ,fill opacity=1 ][line width=1.5]      (0, 0) circle [x radius= 4.36, y radius= 4.36]   ;
\draw    (425.75,225.5) -- (425.75,209.5) ;
\draw [shift={(425.75,209.5)}, rotate = 270] [color={rgb, 255:red, 0; green, 0; blue, 0 }  ][fill={rgb, 255:red, 0; green, 0; blue, 0 }  ][line width=0.75]      (0, 0) circle [x radius= 3.35, y radius= 3.35]   ;
\draw    (473.5,225) -- (473.5,209) ;
\draw [shift={(473.5,209)}, rotate = 270] [color={rgb, 255:red, 0; green, 0; blue, 0 }  ][fill={rgb, 255:red, 0; green, 0; blue, 0 }  ][line width=0.75]      (0, 0) circle [x radius= 3.35, y radius= 3.35]   ;

\draw (76,91.4) node [anchor=north west][inner sep=0.75pt]    {$g$};
\draw (242,104.4) node [anchor=north west][inner sep=0.75pt]    {$h$};
\draw (125,96.4) node [anchor=north west][inner sep=0.75pt]    {$d_{S}( g,h) =n$};
\draw (324,89.4) node [anchor=north west][inner sep=0.75pt]    {$go$};
\draw (553,135.4) node [anchor=north west][inner sep=0.75pt]    {$ho$};
\draw (112,232.4) node [anchor=north west][inner sep=0.75pt]    {$m_{i-1} ,m_{i} \in H$};
\draw (401,117.4) node [anchor=north west][inner sep=0.75pt]    {$\ell ( \pi [ g,h]_S) \leq \lambda n$};
\draw (309,157.4) node [anchor=north west][inner sep=0.75pt]    {$\lambda n\geq $};
\draw (547,168.4) node [anchor=north west][inner sep=0.75pt]    {$\leq 2\lambda n$};
\draw (334,226.4) node [anchor=north west][inner sep=0.75pt]    {$xo$};
\draw (547.5,222.4) node [anchor=north west][inner sep=0.75pt]    {$yo$};
\draw (146,163.4) node [anchor=north west][inner sep=0.75pt]    {$\geq D$};
\draw (62,165.4) node [anchor=north west][inner sep=0.75pt]    {$n\geq $};
\draw (269,50.4) node [anchor=north west][inner sep=0.75pt]    {$G$};
\draw (296,50.4) node [anchor=north west][inner sep=0.75pt]    {$X$};
\draw (144,122.4) node [anchor=north west][inner sep=0.75pt]    {$g_{i-1}$};
\draw (185,122.4) node [anchor=north west][inner sep=0.75pt]    {$g_{i}$};
\draw (402,93.4) node [anchor=north west][inner sep=0.75pt]    {$d_{X}(xo,yo) \geq \epsilon n$};
\draw (463,241.4) node [anchor=north west][inner sep=0.75pt]    {$m_{i} o$};
\draw (406,241.4) node [anchor=north west][inner sep=0.75pt]    {$m_{i-1} o$};
\draw (465,186.4) node [anchor=north west][inner sep=0.75pt]    {$g_{i} o$};
\draw (416,184.4) node [anchor=north west][inner sep=0.75pt]    {$g_{i-1} o$};

\end{tikzpicture}
    \caption{Lemma \ref{SSublContrAcyl}}
    \label{fig:SSublContrAcyl}
\end{figure}
 
\begin{lem}\label{SSublContrAcyl}
Assume that a WPD element $f\in G$ acts as a strongly sublinearly contracting isometry on $X$. Then the pullback projection $\proj_H^X$, where $H=\langle f\rangle$, has strongly sublinear contraction. More precisely, there exists a sublinear function $\kmorse$ with the following property.
Let $\alpha=[g,h]_S$ be a word geodesic  of length $n$ so that $d_S(\alpha,H)\ge n$. Then $$\mathrm{diam}_S\Big(\proj^X_{H}(g)\cup \proj^X_{H}(h)\Big)\le \kmorse(n)$$    
\end{lem}
\begin{proof}
By assumption, 
there exists a sublinear function  $\kappa_0$ so that $\pi(H)$ is strongly $\kappa_0$-contracting.
Choose $X$-projection points $xo,yo$ of $go,ho$ to $\pi(H)$, with $x,y\in H$, so that $\dist_{H}^X(g,h)=d_S(x,y)$. As the map $\pi\lvert_H: H\to \pi(H)$ is a $\lambda$-quasi-isometric embedding, it suffices to prove $$d_X(xo,yo)/d_S(g,h)\to 0,\quad \text{as}\quad d_S(g,h)\to\infty.$$ Assume to the contrary that there exist $\epsilon>0$ and infinitely many pairs $(g,h)$ with $d_S(g,h)\to\infty$ and $d_S([g,h]_S,H)\ge d_S(g,h)$ so that $$d_X(xo, yo)\ge \epsilon\, d_S(g,h).$$  

We now examine the   two cases, each of  which shall lead to a contradiction. 
\medskip

\noindent \textbf{Case 1.} If we have infinitely many pairs $(g,h)$ so that  $$\max\{d_X(go,xo),d_X(ho,yo)\}> \lambda d_S(g,h)$$
then, since $\lambda d_S(g,h) \ge d_X(go,ho)$, we obtain  $d_X(xo, yo)\le \kappa_0(\lambda d_S(g,h))$ by the strongly $\kappa_0$-contracting property of $\pi(H)$.  This, however, violates the  lower bound   $d_X(xo, yo)\ge \epsilon d_S(g,h)$ as  $d_S(g,h)\to \infty$.  
\medskip

\noindent \textbf{Case 2.} Let us next assume that $$d_X(go,xo), d_X(ho,yo)\le \lambda d_S(g,h)$$ holds for infinitely many pairs $(g,h)$ with $d_S(g,h)\to\infty$ and $d_S([g,h]_S,H)\to\infty$. 

\medskip
\noindent
Fixing such a pair $(g,h)$, we apply Lemma \ref{fractalMorseLem} with $\theta=1$ to the $\lambda$-Lipschitz path $\alpha:=\pi([g,h]_S)$ from $go$ to $ho$. The endpoints of $\alpha$ project to the endpoints $xo,yo$ of the segment $\beta=[xo,yo]_{\pi(H)}$. Here we understand $\pi(H)$ as a quasi-geodesic path, so $\beta$ is the subpath between $xo$ and $yo$. If we denote $\epsilon_0:=\lambda/\epsilon$, then $$d_X(go,xo),d_X(ho,yo)\le  \lambda d_S(g,h) \le \epsilon_0 d_X(xo, yo).$$


Before moving on, let us make precise the constants. 
Let    $r=r(\epsilon_0,\lambda,1)$ be the constant given by  Lemma \ref{fractalMorseLem} and then $R=R(r)>4r$ be  given by Lemma \ref{GeomSepAlongH}. Set $$N:=  \lfloor d_X(xo,yo)/6R\rfloor-1$$ Then the  path $\pi([g,h]_S)$ in $X$ passes through a  sequence of $r$-balls around  $R$-separated  points $m_io$ $(1\le i\le N+1)$ on $\beta=[xo,yo]_{\pi(H)}$, where $m_i\in [x,y]_H$. That is,  $d_X(m_io, m_jo)>R$ for $1\le i\ne j\le N+1$, and there are $g_i\in [g,h]_S$ so that $g_io\in B_r(m_io)$. Here $B_r(m_io)$ denotes the ball in $X$ around $m_io$ with radius $r$.  See Fig. \ref{fig:SSublContrAcyl}. 

\medskip
\noindent
For any $D>0$,  Lemma \ref{GeomSepAlongH} provides a constant $M=M(D)$ with the following property: if $g_i,g_{i+1}$ are not contained in the $M$-neighborhood of $H$ in $d_S$-metric, then $d_S(g_i,g_{i+1})\ge D$.  As a consequence, since $[g,h]_S$ is a word geodesic, we obtain  
$$
d_S(g,h)\ge \sum_{i=1}^N d_S(g_i,g_{i+1})\ge ND$$
By the choice of $N$,   $N\ge d_X(xo,yo)/6R-2\ge   \epsilon d_S(g,h)/6R-2$, thus $$
d_S(g,h)\ge  D(\epsilon d_S(g,h)/6R-2)$$ 
Choose $D$ large enough so that $D\epsilon/6R\ge 2$, and let $M=M(D)$ be supplied by Lemma \ref{GeomSepAlongH}. Since $d_S([g,h]_S,H)\ge d_S(g,h)\to\infty$, the geodesics $[g,h]_S$ lie outside $\mathcal N_M(H)$ for all sufficiently large pairs. The preceding inequality then implies $d_S(g,h)\ge 2d_S(g,h)-2D$, and hence $d_S(g,h)\le 2D$, a contradiction.

Finally,  Lemma \ref{lem:reduction} upgrades this estimate, after replacing $\kmorse$ by another sublinear function, to the strongly sublinear contraction asserted in the first sentence.
\end{proof}

From Lemmas \ref{SSublContrAcyl} and \ref{SublTrackInAcyl} we obtain  the following corollary.
\begin{cor}\label{MorseInAcylin}
Suppose that  $G$ acts by isometry on a geodesic metric space $X$. Assume that an element $f\in G$ acts as a strongly sublinearly contracting isometry. If $f$ is a WPD element, then $f$ is a Morse element in $G$.   
\end{cor}
\begin{proof}
Denote $\gamma=\langle f\rangle$. Given $c>1$, let $\alpha$ be a $c$-quasi-geodesic with endpoints on $\gamma$.  By Lemma \ref{SSublContrAcyl}  the pullback projection map  $\proj_\gamma^X$ is strongly $\kmorse$-contracting.  Let $R_0=R_0(\kmorse,1/2c)$ be given  by the ``moreover" part of Lemma \ref{LipProj2Morse}.   

Hence, if $\beta$ is a component of $\alpha\setminus \mathcal N_{R_0}(\gamma)$, Lemma \ref{LipProj2Morse} shows that 
$\mathrm{diam}_S\big(\proj_\gamma^X(\beta)\big)\le  (\ell(\beta)+ R_0)/2c$. 
By Lemma \ref{SublTrackInAcyl}, we see that 
$$
\begin{aligned}
d_S(\beta^-,\beta^+)&\le 2R_0+2\ktrack(R_0)+\mathrm{diam}_S\big(\proj_\gamma^X(\beta)\big)\\
&\le  \ell(\beta)/2c+ 2\ktrack(R_0)+(4c+1)R_0/2c.    
\end{aligned}$$ 
As $\ell(\beta)\le cd_S(\beta^-,\beta^+)+c$, we derive $$\ell(\beta)\le 4c\,\ktrack(R_0)+(4c+1)R_0+2c=:R_1$$ Setting $R=R_0+R_1$ proves $\alpha\subseteq \mathcal N_R(\gamma)$. The proof is complete.
\end{proof}

\begin{rem}
When the action is proper (stronger) and f acts as a Morse isometry (weaker), this result was proved earlier in \cite[Theorem 1.1]{ADT17}. When the action is acylindrical on a hyperbolic space and $f$ is a loxodromic WPD isometry, it was proved in \cite[Theorem 1]{Sisto16}, using Lemma \ref{GeomSepAlongH}. The proofs of Lemmas \ref{SSublContrAcyl} and \ref{SublTrackInAcyl} also rely on this result, although the technical details might be different. 
    
\end{rem}

\subsection{Sublinearly large intersection persists}\label{SSublinearPersists}
Recall that the group $G$ with a finite generating set $S$   acts by isometry on a metric space $X$. Assume that $F\subseteq G$ is a finite non-empty set of  elements, which  act as independent Morse isometries   on $\mathrm{Cay}(G,S)$ and on $X$.  Denote $$\f=\{gE(f): g\in G, f\in F\},\qquad \pi(\f)=\{g\ax(f): g\in G, f\in F\}$$ 
Let  $(C,\kmorse^S)$   be given as in Convention \ref{kmorseConv} for Morse quasi-geodesics $\gamma\in \f$ in $\mathrm{Cay}(G,S)$;  similarly, $(r,\kmorse^X)$ for   Morse quasi-geodesics $\pi(\gamma)$ in $X$.

Axiomatizing the above discussion, we  propose  the following two assumptions: 
\begin{enumerate}
    \item[\textbf{(A)}] 
    The $d_S$-projection sublinearly $\ktrack$-tracks $d_X$-projection along  $\f$. 
\item[\textbf{(B)}]  The pulled-back $d_X$-projection $\proj^X_{\gamma}$ to $\gamma\in \f$ is strongly $\kmorse$-contracting in the word metric.  
\end{enumerate} 
where  $\kmorse$ and  $\ktrack$ are two fixed sublinear functions. 

Let $\kappa:\mathbb R_{>0}\to\mathbb R_{>0}$ be a (not necessarily sublinear) positive function. 
We now introduce  two technical notions of large intersections in the Cayley graph and the space.
\begin{defn}\label{LargeIntSets}
We say that an element $g\in G$ has \textit{$(C,\kappa)$-intersection} along $\gamma\in \f$ if for some word geodesic $[1,g]_S$, we have $$
\mathrm{diam}_S\bigl([1,g]_S\cap \mathcal N_C(\gamma)\bigr)> \kappa(\|g\|_S)
$$
where $\mathcal N_C(\cdot)$ denotes the neighborhood in $d_S$-metric.

Similarly,  $g$ has \textit{$(r,\kappa)$-intersection} along $\pi(\gamma)\in \pi(\f)$  if for some $d_X$-geodesic $[o,go]_X$, we have $$
\mathrm{diam}_X\bigl([o,go]_X\cap \mathcal N_r(\pi(\gamma))\bigr)> \kappa(\|g\|_S)
$$  
where $\mathcal N_r(\cdot)$ denotes the neighborhood in $d_X$-metric.
We emphasize that  the distance taken on the right-hand side is the word distance of $g$.
\end{defn}

\begin{lem}\label{LargeIntLem}
Under the  assumptions (A) and (B),   there exists a function $\tilde \kappa=\tilde \kappa(\kappa)$ with the following properties: 

\begin{enumerate}
\item 
If $g\in G$ has {$(r,\tilde \kappa)$-intersection} along $\pi(\gamma)$, then $g$ has \textit{$(C,\kappa)$-intersection} along $\gamma$.

\item   
If $g\in G$ has {$(C,\tilde \kappa)$-intersection} along $\gamma$, then $g$ has \textit{$(r,\kappa)$-intersection} along $\pi(\gamma)$.

\end{enumerate}
The function $\tilde \kappa$ differs from a constant multiple of $\kappa$ by constant multiples of $\kmorse$ and $\ktrack$.
\end{lem}
\begin{proof}
We only prove (1); the proof of (2) is analogous.

Let $\lambda$ be the Lipschitz constant of the orbital map
$\pi:G\to X$.  Enlarging $\kmorse$ and $\ktrack$ if necessary, we may
assume that
\[
\kmorse(n)\geq 
\max\{\kmorse^X(\lambda n+r),\kmorse^S(3n)\}.
\]
Define
\[
\tilde\kappa(n)
=
\lambda\bigl(
4\kmorse(n)+2\ktrack(3n)+\kappa(n)+4r+4C
\bigr).
\]

Suppose that $g$ has $(r,\tilde\kappa)$-intersection along
$\pi(\gamma)$, and let $n=\|g\|_S$.  Fix a $d_X$-geodesic
$[o,go]_X$.  Choose $x,y\in\gamma$ to be $d_X$-projection points of
$o,go$ to $\pi(\gamma)$ respectively.  Since
$[o,go]_X$ intersects $\mathcal N_r(\pi(\gamma))$ and
$d_X(o,go)\leq \lambda n$, we have
\[
d_X(o,xo),d_X(go,yo)\leq \lambda n+r .
\]
Hence, by Lemma \ref{SubLinearContr} applied to $\pi(\gamma)$,
\[
\begin{aligned}
d_X(xo,yo)
&\geq 
\operatorname{diam}_X
([o,go]_X\cap\mathcal N_r(\pi(\gamma)))
-2\kmorse^X(\lambda n+r)-4r
\\
&>
\tilde\kappa(n)-2\kmorse^X(\lambda n+r)-4r
\\
&\geq
\lambda\bigl(
2\kmorse(n)+2\ktrack(3n)+\kappa(n)+4C
\bigr).
\end{aligned}
\]
Therefore
\begin{align}\label{eq:dS(x,y)}
d_S(x,y)>
2\kmorse(n)+2\ktrack(3n)+\kappa(n)+4C.    
\end{align}

Now let $[1,g]_S$ be an arbitrary word geodesic, and set
$L_1=d_S(1,\gamma)$ and $L_2=d_S(g,\gamma)$. If $L_1\ge n$, then
Assumption (B), applied to the pair $(1,g)$, gives
$d_S(x,y)\le \kmorse(n)$, contradicting (\ref{eq:dS(x,y)}). The same
argument with the pair $(g,1)$ excludes $L_2\ge n$. Hence
$L_1,L_2<n$.

By Assumption (A) and \eqref{SublTrackInquality},
\[
\begin{aligned}
\dist^S_\gamma(1,g)
&\geq d_S(x,y)-\ktrack(L_1)-\ktrack(L_2)\\
&>2\kmorse^S(3n)+\kappa(n)+4C.
\end{aligned}
\]
In particular, the hypothesis of Lemma \ref{SubLinearContr} is
satisfied in the Cayley graph, so $[1,g]_S$ intersects
$\mathcal N_C(\gamma)$. The moreover part of that lemma gives
\[
\begin{aligned}
\operatorname{diam}_S
([1,g]_S\cap\mathcal N_C(\gamma))
&\geq
\dist^S_\gamma(1,g)
-\kmorse^S(L_1)-\kmorse^S(L_2)-4C
\\
&>
\kappa(n),
\end{aligned}
\]
Hence $g$ has $(C,\kappa)$-intersection along $\gamma$.
The same argument with the two metrics interchanged proves (2), after
increasing $\tilde\kappa$ by a constant multiple if necessary.
\end{proof}

\section{Effective estimates on the growth function}\label{SecEffEst}


Assume that $G$ is an  acylindrical hyperbolic group with a finite generating set $S$.  Let $K<G$ be a non-virtually cyclic subgroup with a loxodromic WPD element.  That is, there exists an isometric action of $G$ on a hyperbolic space $X$ so that $K$ is a non-elementary subgroup with a loxodromic WPD isometry. The goal of this section is to estimate the relative growth function of $K$ $$n\,\longmapsto \,|B_n\cap K|$$ 
where $B_n=\{g\in G: \|g\|_S\le n\}$.

We shall prove the following slightly more precise version of ~\ref{GrowthBallThm}.
Let $F\subseteq K$ be a set of 3 elements which act as independent loxodromic WPD  isometries on $X$ (since a non-virtually cyclic subgroup $K$ contains infinitely many independent ones). Assume that $\kmorse$  is the  sublinear function  associated to $f\in F$.  
Denote $\f=\{kE(f): k\in K, f\in F\}$.

\begin{thm}\label{GrowthBallThm2}
There exists a sub-exponential function $\phi$ so that $\log \phi$ is controlled by the sublinear functions above and
$$
\forall n\ge 1:\; |B_n\cap K|\le \phi(n)\, \mathrm{e}^{n \omega(K,S)}
$$
In particular,  $$\mathrm{e}^{n \omega(G,S)}\le |B_n|\le  \phi(n)\, \mathrm{e}^{n \omega(G,S)}$$
\end{thm}
Fix once and for all a number $\Delta\ge 0$. Define the annulus-like set of radius $n$ with width $\Delta$
\[
S_n(x, \Delta)=\{g\in G: |d_S(x, g)-n|\le\Delta\}
\]
For simplicity, write $S_n(\Delta)=S_n(1,\Delta)$ if $x=1$, and $S_n(x)=S_n(x,0)$ for the $n$-th sphere at $x$, and $S_n=S_n(1)$. Since balls are unions of linearly many thickened spheres, we have
$$
\limsup_{p\to\infty}\frac{\log |B_p\cap K|}{p}
=
\limsup_{p\to\infty}\frac{\log \max\{1,|S_p(\Delta)\cap K|\}}{p}.
$$
The technical bone is the following intermediate result.

\begin{prop}\label{LimitExistsProp}
There exists a sublinear function $\kappa$, controlled by the sublinear functions above and depending on $S$, so that the following holds for any $n, m\ge 1$:    
\[
|S_m(\Delta)\cap K|^n
\le
\mathrm{e}^{n\kappa(m)}\, |B_{n(m+\kappa(m))}\cap K|.
\]
\end{prop}

Assuming Proposition \ref{LimitExistsProp}, if   set $s(p)=\log \max\{1,|S_p(\Delta)\cap K|\}$ and $\tilde s(p)=\log |B_{p}\cap K|$,   we obtain 
\begin{align}\label{eq:sublinearsubadditive}
\forall n,m\ge 1,\quad \tilde s({nm+n\kappa(m)})\ge n\cdot s({m})-  n\kappa(m)    
\end{align} 
which implies Theorem \ref{GrowthBallThm2} (\ref{GrowthBallThm}) by the following sublinear variant of Fekete's lemma. 

\begin{lem}
Let $\{s(p):p\ge 1\}$ and $\{\tilde s(p):p\ge 1\}$ be two  sequences of  non-negative real numbers so that $$\limsup_{p} \frac{\tilde s(p)}{p}=\limsup_{p} \frac{s(p)}{p}.$$  Let $\kappa$ be a sublinear function so that Eq. (\ref{eq:sublinearsubadditive}) holds for any   integers $m$ and $n$. Then the limit exists $$\omega:=\lim_{p\to\infty}\frac{\tilde s(p)}{p}$$ and for any $p\ge 1$: $$s(p)\le \omega p+(\omega+1) \kappa(p).$$  
\end{lem}
\begin{proof}
Fix an integer $m\ge 1$ and denote $q_m=m+\kappa(m)$. Let us write an integer $p=nq_m+r$ for $0\le r\le q_m$.  By Eq. (\ref{eq:sublinearsubadditive}), 
$$
\begin{aligned}
 \frac{\tilde s({p})}{p}\ge \frac{\tilde  s({nq_m})}{p} &\ge \frac{nq_m}{p}\cdot \left(\frac{s(m)}{q_m}-\frac{\kappa(m)}{q_m}\right)\\ 
\end{aligned}$$
Letting $p\to \infty$,  we have : 
\begin{align}\label{defn:omegalimit}
\omega:=\liminf_{p} \frac{\tilde s(p)}{p}\ge  \left(\frac{s(m)}{m+\kappa(m)}-\frac{\kappa(m)}{m+\kappa(m)}\right)    
\end{align} for every fixed $m\ge 1$.   The $\kappa$ is sublinear, so  $\frac{\kappa(m)}{m}\to 0$ as $m\to\infty$.  Taking the limit super on the RHS over $m$  yields
$$\liminf_{p} \frac{\tilde s(p)}{p}\ge   \limsup_{m} \frac{s(m)}{m}=\limsup_{m} \frac{\tilde s(m)}{m}$$ 
where the equality follows by assumption. Thus the limit exists, $\omega=\lim_{p\to\infty}\frac{\tilde s(p)}{p}$, and by Eq. (\ref{defn:omegalimit}), $$s(m)\le   \omega (m+\kappa(m))+\kappa(m)\le \omega m+(\omega+1)\kappa(m)$$ holds for any $m\ge 1$. The proof is complete.
\end{proof}

\subsection{Construction of free semi-groups}
The remainder of this section is to prove Proposition \ref{LimitExistsProp}, which relies on the construction of free semi-groups. We first prepare  the data used in construction.
\medskip

\noindent\textbf{Sublinear function $\kappa$.}
By Lemma \ref{LargeIntLem}, let us fix   $\kappa \ge 10\kmorse$ so that
\begin{align}\label{eq:kappa0defn}
\forall \gamma\in \f:\; \mathrm{diam}_X\bigl([o,go]_X\cap \mathcal N_r(\pi(\gamma))\bigr)>\kappa(\|g\|_S)\quad \Rightarrow\quad [1,g]_S\cap \mathcal N_C(\gamma)\ne \emptyset.   
\end{align}

Let $\omega_0$ be a large constant depending on $S$ so that $|S^{\le n}|\le \mathrm{e}^{\omega_0 n}$ for any $n\ge 1$.

\medskip

\noindent\textbf{Sublinearly separated annulus.}
Fix an integer $m\gg 1$.  Let $A$ be a maximal $\kappa(m)$-separated subset of $S_m(\Delta)\cap K$, so we have 
$$\begin{aligned}\label{eqn:Amaxsize}
|A|\;\ge\; \frac{|S_m(\Delta)\cap K|}{|B_{\kappa(m)}|}\;\ge\; |S_m(\Delta)\cap K|\,\mathrm{e}^{-\omega_0\kappa(m) }    
\end{aligned}$$

The following lemma is crucial in the construction.
\begin{lem}\label{SublExtensionViaLoxo}
There exist $c,D,L_0>1$ depending only on $\f$ with the following property. 
Fix $m\ge 1$. There exist a subset of $A$ (still denoted by $A$) and $f\in F$ so that     
\begin{align}\label{eqn:Asize}
|A|\;\ge\;  \frac{1}{6}\cdot |S_m(\Delta)\cap K|\,\mathrm{e}^{-\omega_0\kappa(m) }    
\end{align}
Moreover,  let $(a_1,a_2,\cdots, a_\ell)$  be any $\ell$-tuple of elements  in $A$.
    \begin{enumerate}
    \item if $\ell\ge 2$, then $(a_1,a_2,\cdots, a_\ell)$ is $D$-admissible relative to the $(\ell-1)$-tuple $(E(f),\cdots, E(f))$ in the sense of Definition \ref{defn:admissibletuples};
    \item for any fixed $h\in E(f)$ with $d_X(o,ho)>L_0\kappa(m)$, the word $$W=a_1h a_2h\cdots a_\ell h$$   labels a $c$-quasi-geodesic in $X$. 
\end{enumerate}  
\end{lem}
\begin{proof}
Note that $F$ consists of 3 independent WPD loxodromic elements. For every $a\in A$ there exist two distinct elements $f\in F$ so that
$$
\mathrm{diam}_S\bigl([1,a]_S\cap \mathcal N_C(E(f))\bigr)\le D
$$
by Lemma \ref{BddprojViaMorse}. Hence, there are two distinct elements  $f_1,f_2\in F$ and a subset $A_1\subset A$ with $|A_1|\ge |A|/3$ so that for any $a\in A_1$ and $i=1,2$, 
$$
\mathrm{diam}_S\bigl([1,a]_S\cap \mathcal N_C(E(f_i))\bigr)\le D.
$$
Applying Lemma \ref{BddprojViaMorse} to the inverses $\{a^{-1}:a\in A_1\}$ and translating back by $a$, there exists a subset $A_2\subset A_1$ with $|A_2|\ge |A_1|/2$ and some $f\in \{f_1,f_2\}$ so that for any $a\in A_2$,
$$
\mathrm{diam}_S\bigl([1,a]_S\cap \mathcal N_C(aE(f))\bigr)\le D.
$$ 
Together with (\ref{eqn:Amaxsize}), this gives (\ref{eqn:Asize}) after renaming $A_2$ as $A$.
To conclude the $D$-admissibility, it suffices to have $aE(f)\ne E(f)$ for any $a\in A$. This holds when $m\gg 1$. Indeed, if $a\in E(f)$, then by Convention \ref{kmorseConv} the geodesic $[1,a]_S$ is contained in a fixed neighborhood of $E(f)$, contradicting the bounded intersection above when $m-\Delta>D$. 

The bounded intersection above and Lemma \ref{BddprojViaMorse2}, after applying the orbital map to $X$, imply that the projections of $[o,a_io]_X$ and $h[o,a_{i+1}o]_X$ to $\ax(f)$, after translating the first projection by $a_i^{-1}$, have diameter at most a constant multiple of $\kappa(m)$. Therefore, choosing $L_0$ sufficiently large, each concatenation
$$
[o,a_io]_X\cdot(a_i[o,ho]_X)\cdot(a_ih[o,a_{i+1}o]_X)
$$
is a uniform local quasi-geodesic. The conclusion then follows from the classical fact in hyperbolic geometry that a sufficiently long local quasi-geodesic is a global $c$-quasi-geodesic.    
\end{proof}
\begin{rem}
We stress that Lemmas \ref{ExtensionViaMorse} and \ref{ExtensionViaMorse2} are local nature in the sense that the involved estimates depend on the length of sequences, so these are not sufficient in proving that the semi-group is free in Lemma \ref{FreeSemigroup}. In the proof given below, the sublinear projection tracking for WPD actions are essential.   
\end{rem}

The next lemma explains how to choose suitable $h\in E(f)$ in Lemma \ref{SublExtensionViaLoxo}. Such elements  exist, since $\langle f\rangle$ with the ambient metric $d_S$ is quasi-isometric to $\mathbb R$ and its orbit in $X$ is quasi-isometrically embedded.
\begin{lem}\label{SuitableEDefn}
For each $f\in F$ and each $L_0>0$, there exists $B_0>0$ so that for all sufficiently large $n$ there exists an element $h\in \langle f\rangle$ satisfying
$$
\|h\|_S  \le  B_0\kappa(n)
\quad\text{and}\quad
d_X(o,ho)>L_0\kappa(n).
$$
\end{lem}
\begin{proof}
Since $\langle f\rangle$ is undistorted in the word metric and $f$ acts loxodromically on $X$, both $\|f^q\|_S$ and $d_X(o,f^qo)$ grow linearly in $|q|$. Choosing $q=q(n)$ to be a sufficiently large constant multiple of $\kappa(n)$ gives the desired element $h=f^q$.
\end{proof}
For each $m\gg 1$, choose $h=h(m)\in \langle f\rangle$ as above. Applying the lemma with a sufficiently large value of $L_0$ and enlarging $\kappa$ by a constant multiple if necessary, we shall simply write
$$
\|h\|_{S}\le \kappa(m)
\quad\text{and}\quad
d_X(o,ho)>L_0\kappa(m).
$$

\medskip

\noindent\textbf{Construction of extension map.}
Let $\mathbb W(A)=\cup_{n\ge 1}A^n$ where $A^n$ denotes the set of words of length $n$ over the alphabet set $A$.  We  define a \textit{$\kappa$-extension} map   $\Phi: \mathbb W(A)\to K$ by 
$$
\Phi: W=(a_i)_1^n\longmapsto \prod_{i=1}^{n} (a_ih)
$$
By Lemma \ref{SublExtensionViaLoxo}, there is a constant $c>1$ independent of $W$ so that $\Phi(W)$ labels a $c$-quasi-geodesic.

\begin{lem}\label{FreeSemigroup}
The map $\Phi: \mathbb W(A)\to K$ is  injective, so $T:=Ah$ generates a free semi-group $H=\langle T\rangle^+$ in $K$.  
\end{lem}
\begin{proof}
To derive a contradiction, assume that  there are two distinct words $$W:=(a_i)_1^n\quad \ne\quad W':=(a_j')_1^l$$ with $n,l\ge 1$ so that $\Phi(W)=\Phi(W')$. The corresponding words     $$\prod_{i=1}^{n} \left(a_ih\right),\qquad \prod_{j=1}^{l} \left(a_j'h\right)$$ label two $c$-quasi-geodesics $\alpha$ and $\beta$ with the same endpoints. For reference, fix a $d_X$-geodesic $\gamma$ with the same endpoints as $\alpha$ and $\beta$ do. By Morse Lemma,   there exists a constant $r$ depending on $c$ so that $\alpha$ is contained in $\mathcal N_r(\beta)$ and vice versa. 

\medskip
We may assume   $a_1\ne a_1'$ without loss of generality. Denote  $U_1=a_1E(f)$ and $V_1=a_1'E(f)$. Note that  $\|a_1\|_S, \|a_1'\|_S\in [m-\Delta,m+\Delta]$. We consider the following two cases.
\medskip

\noindent \textbf{Case 1.} Assume first that $U_1 = V_1$. If $u\in [1,a_1]_S$ and $v\in [1,a_1']_S$ are the corresponding entry points in $\mathcal N_C(U_1)=\mathcal N_C(V_1)$, then  $d_S(u,U_1), d_S(v,V_1)\le C$. By Lemma \ref{SubLinearCommonEntry}, after moving the entry points a distance at most $C$ to $U_1$, 
$$
d_S(u,v)\le 2\kmorse(m+\Delta)+6C.
$$ 
By construction, the terminal intersections of $[1,a_1]_S$ and $[1,a_1']_S$ with the neighborhoods of $a_1E(f)$ and $a_1'E(f)$ have diameter at most $D$. Thus $d_S(a_1,u)\le D$ and $d_S(a_1',v)\le D$. Hence
$$
d_S(a_1,a_1')\le 2\kmorse(m+\Delta)+6C+2D.
$$
Enlarging $\kappa$ by a constant multiple of $\kmorse$ and taking $m$ sufficiently large, this contradicts the choice of $A$ being $\kappa(m)$-separated.
\medskip

\noindent \textbf{Case 2.} Assume  now that $U_1\ne V_1$. Then there exists  some $B_0=B_0(r)$ so  that $$\mathrm{diam}_X\big(\mathcal N_{2r}(\pi(U_1))\cap \mathcal N_{2r}(\pi(V_1))\big)\le B_0.$$  Without loss of generality, assume $d_X(o,a_1o)\ge d_X(o,a_1'o).$ The subpath of $\beta$ labelled by $h$ is a translate of $[o,ho]_X$; after increasing $r$, it lies in $\mathcal N_r(\pi(V_1))$ and has $X$-diameter $d_X(o,ho)$. Since $\alpha$ and $\beta$ are uniform quasi-geodesics with the same endpoints, the order-preserving fellow-travelling property of quasi-geodesics implies, after increasing $r$, that $\alpha$ contains a subpath of diameter at least $d_X(o,ho)-2r$ in the $r$-neighborhood of $\pi(V_1)$. The part of this subpath lying in the connector $a_1[o,ho]_X\subseteq \mathcal N_r(\pi(U_1))$ has diameter at most $B_0$. Hence, the first piece $[o,a_1o]_X$ intersects $\mathcal N_{2r}(\pi(V_1))$ in a diameter at least
$$
d_X(o,ho)-B_0-2r>\kappa(m+\Delta),
$$
where the last inequality holds when $m\gg 0$ by the choice of $h$. By the defining property of $\kappa$ in Eq. (\ref{eq:kappa0defn}), this shows that $[1,a_1]_S$ intersects $\mathcal N_C(V_1)$. 

\medskip
Let $u\in [1,a_1]_S$ be the   entry point in $\mathcal N_C(V_1)$ and $v\in [1,a_1']_S$ the   entry point in $\mathcal N_C(V_1)$. Then, after moving the entry points a distance at most $C$ to $V_1$, Lemma \ref{SubLinearCommonEntry} gives
$$
d_S(u,v)\le 2\kmorse(m+\Delta)+6C.
$$
By construction, the terminal intersection of $[1,a_1']_S$ with $\mathcal N_C(a_1'E(f))=\mathcal N_C(V_1)$ has diameter at most $D$, so $d_S(v,a_1')\le D$. Noting $u\in [1,a_1]_S,v\in [1,a_1']_S$ are on geodesics and $\|a_1\|_S\le \|a_1'\|_S+2\Delta$, we see
$$
\begin{aligned}
d_S(u,a_1)&= d_S(1,a_1)-d_S(1,u) \\
&\le d_S(1,a_1')+2\Delta -d_S(1,u)\\
&\le d_S(u,v)+d_S(v,a_1')+2\Delta \\
&\le 2\kmorse(m+\Delta)+6C+D+2\Delta 
\end{aligned}$$ 
Hence,
$$
d_S(a_1,a_1')\le d_S(u,v)+ d_S(u,a_1)+d_S(v,a_1')\le 4\kmorse(m+\Delta)+12C+2D+2\Delta.
$$
Enlarging $\kappa$ as above and taking $m$ sufficiently large, this again contradicts  that $A$ is $\kappa(m)$-separated.  

\medskip
In each case, we arrived to a contradiction, so  the injectivity of $\Phi$ follows.
\end{proof}

By Lemma \ref{FreeSemigroup},  the standard Cayley graph of  the free semi-group $H=\langle Ah\rangle^+<K$   is a rooted tree with outer-ward valence $|A|$ in (\ref{eqn:Asize}). This produces at least
$$
|A|^n \ge (|S_m(\Delta)\cap K|)^n\mathrm{e}^{ -n(\omega_0 \kappa(m)+\log 6) }
$$
elements of $K$ in the ball $B_{n(m+\kappa(m)+\Delta)}$ in $\mathrm{Cay}(G,S)$. Enlarging $\kappa$ by a constant multiple and absorbing the fixed number $\Delta$, this ball is contained in $S^{\le n(m+\kappa(m))}$.    

Setting $s(p)=\log \max\{1,|S_p(\Delta)\cap K|\}$ and $\tilde s(p)=\log |B_{p}\cap K|$, we obtain
$$
\tilde s({nm+n\kappa(m)}) \ge n\cdot s({m}) - n\kappa(m).
$$
Thus Proposition \ref{LimitExistsProp} follows.
\begin{rem}
When $K=G$, we may decompose the word geodesic at position $nm$ so that   $$|B_{nm+3n\kappa(m)}\cap K|\le |B_{nm}\cap K|\cdot | B_{3n\kappa(m)}\cap K|$$ 
Replacing $\kappa$ by a constant multiple if necessary, we obtain $$\tilde s({nm})\ge n\cdot s({m})-  n\kappa(m)$$
This is simpler than the inequality (\ref{eq:sublinearsubadditive}). 
    
\end{rem}

\subsection{Largeness of free semigroups}
To conclude this section, we  record a counting consequence of Lemma~\ref{FreeSemigroup} which will be used in next section.

\begin{lem}\label{LargeFreeSemigroupGrowth}
There exists a sequence of free semigroups \(H_i<G\) such that
\[
        \e{H_i}\to \e{G}
\]
and every non-trivial element of \(H_i\) acts loxodromically on \(X\). In particular, $H_i$ consists of Morse elements.
\end{lem}

\begin{proof}
Apply the construction in the proof of Lemma~\ref{FreeSemigroup} with
\(K=G\). Thus, for every sufficiently large \(m\), we may choose a  
\(\kappa(m)\)-separated subset $A_m\subset S^m(\Delta)$ so that (\ref{eqn:Asize}) holds:
\[
        |A_m|
        \ge
        \frac{1}{6}|S^m(\Delta)|\exp(-\omega_0\kappa(m)).
\]
and an element \(f_m\in F\) and $h_m\in \langle f_m\rangle$
such that
\[
        \|{h_m}\|_{S}\le \kappa(m)
        \quad\text{and}\quad
        d_X(o,h_mo)>L_0\kappa(m).
\]
Thus, by Lemma~\ref{FreeSemigroup}, 
\[
        H_m:=\langle A_m h_m\rangle^+
\]
is a free semigroup for each large $m\gg 1$.

Every generator \(a h_m\), with \(a\in A_m\), has \(S\)-length at most $m+\Delta+\kappa(m)$. Enlarging $\kappa$ by a constant multiple, we write this upper bound as $m+\kappa(m)$. Since \(H_m\) is freely generated by \(A_m h_m\), we have, for every \(n\ge 1\),
\[
        |B_{n(m+\kappa(m))}\cap H_m|
        \ge
        |A_m|^n.
\]
Therefore
\[
        \e{H_m}
        \ge
        \frac{\log |A_m|}{m+\kappa(m)}
        \ge
        \frac{\log |S_m(\Delta)|-\omega_0\kappa(m)-\log 6}
             {m+\kappa(m)}.
\]
Choose a sequence \(m_i\to\infty\) such that
\[
        \frac{\log \max\{1,|S_{m_i}(\Delta)|\}}{m_i}\to \e{G}.
\]
Since \(\kappa(m)=o(m)\), we get
\[
        \e{H_{m_i}}\to \e{G}.
\]

Finally, if \(1\ne g\in H_m\), then all powers \(g^n\) are represented by
words over the alphabet \(A_m h_m\). By Lemma~\ref{SublExtensionViaLoxo},
the corresponding paths in \(X\) are $c$-quasi-geodesics. Hence   \(g\) is loxodromic on \(X\). The ``in particular" statement follows by Sisto \cite{Sisto16}; see Lemma \ref{MorseInAcylin} as well.
\end{proof}

\section{No growth gap for quotients}\label{SecNoGrowthGapQuotients}

In this section, we prove that the growth rate of an acylindrically
hyperbolic group can be approximated by the growth rates of proper
quotients. By a proper quotient we mean that its kernel is infinite. We restate \ref{NoGapThm}  as follows.

\begin{thm}
\label{NoGrowthGapAHQuotients}
Let \(G\) be an acylindrically hyperbolic group with finite generating set
\(S\). Then for any loxodromic WPD element \(h\in G\), there exists an integer
\(k\ge 1\), and a sequence \(q_i\to\infty\) such that
\[
        \bar G_i:=G/\langle\!\langle h^{kq_i}\rangle\!\rangle
\]
are proper quotients and
\[
        \e{\bar G_i,S_i}\to \e{G,S},
\]
where \(\e{\bar G_i,S_i}\) is computed with respect to the image of \(S\) in
\(\bar G_i\).
\end{thm} 

Modulo Lemma \ref{LargeFreeSemigroupGrowth}, the proof follows closely the arguments of \cite{YANG7} and \cite{HLY}. Below we give a streamlined version of the argument in \cite{YANG7}, pointing out the differences and referring to these papers for details. 

Assume that $G$ admits a non-elementary action on a hyperbolic space $X$ so that \(h\in G\) is a loxodromic WPD element on $X$. By Lemma \ref{lem:bddintersectionamples}, the system of quasi-geodesic axes
\[
        \mathcal A:=\{g\ax(h):g\in G\}.
\]
has bounded intersection. After rescaling $X$, we cone off the members of $\mathcal A$ so that sufficiently high powers of $h$ form a very rotating family in the sense of \cite{DGO}. The required inputs are the bounded intersection of $\mathcal A$ and the proper action of $E(h)$ on $\ax(h)$, the latter following from the WPD assumption. The following is proved in \cite[Lemma~8.8]{YANG7} using the Greendlinger lemma from \cite{DGO}; the same proof applies in the present setting.
\begin{lem}\label{KernelLargeIntersectionX}
There exist constants
\(k=k(h)\ge 1\), \(\epsilon>0\), and a function
\[
        L_h(q)\to\infty
\]
such that for every \(q\ge 1\),   if \(1\ne g\in N_q:=\langle\!\langle h^{kq}\rangle\!\rangle\), then there exists
\(Y\in\mathcal A\) such that
\[
        \mathrm{diam}_X\bigl([o,go]_X\cap \mathcal N_{\epsilon}(Y)\bigr)
        \ge L_h(q).
\]
\end{lem}

By the thin-triangle property of $X$, we obtain the following consequence.
\begin{cor}
\label{PairLargeIntersectionX}
There exist \(k=k(h)\ge 1\), \(\epsilon_h>0\), \(C>0\), and
\(L_h(q)\to\infty\) such that the following holds.
 
If \(g_1,g_2\in G\) satisfy $g_1^{-1}g_2\in N_q\setminus\{1\},$
then  one of the two geodesics \([o,g_1o]\), \([o,g_2o]\) satisfies
\[
        \mathrm{diam}_X([o,g_io]_X\cap \mathcal N_{\epsilon_h+C}(Y))
        \ge L_h(q)/3-C
\]
for some \(Y\in\mathcal A\).
\end{cor}

Choose a set \(F\subset G\) of three pairwise independent loxodromic WPD elements such that \(h\)
is independent from every element of \(F\). This is possible because the action of $G$ on $X$ is non-elementary. Apply the construction of Lemma \ref{LargeFreeSemigroupGrowth} using this set $F$. Since the radius $\rho:=\epsilon_h+C$ is now fixed, we may take the constant $L_0$ in Lemma \ref{SublExtensionViaLoxo} sufficiently large for the bounded-intersection argument below. This does not affect the growth estimate, since the word length of each inserted element remains bounded by a constant multiple of $\kappa(m)$.

The resulting free semigroups have uniformly bounded intersection with each axis in $\mathcal A$.
\begin{lem}
\label{FreeSemigroupBoundedHIntersection}
Let $H_m=\langle A_m h_m\rangle^+$
be constructed by Lemma~\ref{LargeFreeSemigroupGrowth}. Then
there exists \(D_m>0\) such that
\[
        \sup_{g\in H_m}
        \sup_{Y\in\mathcal A}
        \mathrm{diam}_X\bigl([o,go]_X\cap \mathcal N_\rho(Y)\bigr)
        \le D_m.
\] 
\end{lem}

\begin{proof}
Fix \(m\). By Lemma~\ref{SublExtensionViaLoxo}, every word in $H_m$ over the alphabet
\(A_m h_m\) labels a uniform quasi-geodesic in \(X\). Denote the corresponding
path from \(o\) to \(go\) by \(\alpha_g\). Its pieces labelled by elements of
\(A_m\) have \(S\)-length at most \(m+\Delta\), and hence their \(X\)-length is
bounded above by a constant $M_m$. Each connector piece labelled by $h_m$ lies
in a uniform neighborhood of some translate of $\ax(f_m)$, where \(f_m\in F\).

Let $R$ be a uniform fellow-travelling constant between $\alpha_g$ and
$[o,go]_X$, and let $Q$ be such that every connector lies in the
$Q$-neighborhood of its corresponding translate of $\ax(f_m)$. Fix a constant
$T$ so that any subsegment of a geodesic with endpoints in
$\mathcal N_\rho(Y)$ lies in $\mathcal N_T(Y)$. Since $F\cup\{h\}$ is a finite
set of independent loxodromic WPD elements, Lemma
\ref{lem:bddintersectionamples} gives a constant $B>0$ such that
\[
\mathrm{diam}_X\bigl(\mathcal N_Q(U)\cap\mathcal N_{T+R}(Y)\bigr)\le B
\]
for every translate $U$ of $\ax(f_m)$ and every $Y\in\mathcal A$. By our
choice of $L_0$, every connector has $X$-diameter greater than $B$.

Suppose that $x,y\in[o,go]_X\cap\mathcal N_\rho(Y)$. By the choice of $T$, the
subsegment $[x,y]_X$ lies in $\mathcal N_T(Y)$. The order-preserving fellow-travelling between
$[o,go]_X$ and $\alpha_g$ therefore gives a subpath of $\alpha_g$, with endpoints
uniformly close to $x,y$, which lies in $\mathcal N_{T+R}(Y)$.
This subpath cannot contain a complete connector, for otherwise that connector
would have diameter at most $B$. Consequently, it consists of at most two
partial connectors and one piece labelled by an element of $A_m$. Its diameter
is therefore bounded in terms of $m$. After also accounting for the
fellow-travelling error, we obtain a constant $D_m$ such that
\[
        \mathrm{diam}_X\bigl([o,go]_X\cap \mathcal N_\rho(Y)\bigr)
        \le D_m
\]
for all \(g\in H_m\) and all \(Y\in\mathcal A\).
\end{proof}

Taking a sufficiently high power relative to $m$, the free semigroups inject into the appropriate quotients.
\begin{lem}
\label{EmbedFreeSemigroupQuotient}
For every \(m\) sufficiently large, there exists \(q(m)\ge 1\) such that,
for every \(q\ge q(m)\), the quotient map
\[
        \pi_q:G\to G/\langle\!\langle h^{kq}\rangle\!\rangle
\]
is injective on \(H_m\).
\end{lem}

\begin{proof}
Fix \(m\). Let \(D_m\) be given by
Lemma~\ref{FreeSemigroupBoundedHIntersection}, where
$\rho=\epsilon_h+C$ and \(C\) is the constant from
Corollary~\ref{PairLargeIntersectionX}.  
If  \(\pi_q\) is not injective on \(H_m\), then there exist distinct
\(u,v\in H_m\) such that $\pi_q(u)=\pi_q(v).$
Hence $v^{-1}u\in N_q\setminus\{1\}.$
Applying Corollary~\ref{PairLargeIntersectionX} with $g_1=v, g_2=u,$
we get, for some $Y\in\mathcal A$,
\[
        \max\{\mathrm{diam}_X([o,uo]_X\cap \mathcal N_{\epsilon_h+C}(Y)),\mathrm{diam}_X([o,vo]_X\cap \mathcal N_{\epsilon_h+C}(Y))\} \ge L_h(q)/3-C.
\]
If we choose
\(q\) so large that
\[
        L_h(q)/3-C>D_m,
\]
this contradicts Lemma \ref{FreeSemigroupBoundedHIntersection}.
Therefore \(\pi_q\) is injective on \(H_m\).
\end{proof}

We now complete the proof of Theorem \ref{NoGrowthGapAHQuotients}.

By Lemma~\ref{LargeFreeSemigroupGrowth},  $\e{H_{m_i},S}\to \e{G,S}.$
For each \(i\), choose \(q_i\ge i\) sufficiently large so that
Lemma~\ref{EmbedFreeSemigroupQuotient} applies to \(H_{m_i}\). Thus the
quotient map $\pi_i:G\to \bar G_i$
is injective on \(H_{m_i}\), and $q_i\to\infty$. Moreover, $N_{q_i}$ is infinite
because it contains the infinite cyclic subgroup $\langle h^{kq_i}\rangle$, so
$\bar G_i$ is a proper quotient. Since the quotient map does not increase word
length, injectivity on $H_{m_i}$ gives
\[
\e{H_{m_i},S}\le \e{\bar G_i,S_i}\le \e{G,S}.
\]
It follows that \(\e{\bar G_i,S_i}\to \e{G,S}\). The proof is complete.

\section{The growth-cogrowth inequality}
Suppose that a non-elementary group $G$ is either acylindrically hyperbolic (type I) or contains a Morse element and
satisfies the Morse local-global assumption (type II).  In both cases, $G$ contains a unique maximal finite normal subgroup denoted by $E(G)$  by Lemma \ref{lem:E(G)}. A subset $P \subset G$ is called \textit{non-degenerate} if it  is disjoint from $E(G)$. 

Let $H$ be a subgroup of $G$. Define 
$$\omega(G/H,\bar S):=\limsup_{n\to\infty} \frac{\log |\{Hg: d_S(H,Hg)\le n\}|}{n}$$
If $H$ is a normal subgroup, then $\omega(G/H,\bar S)$ is exactly the growth rate of the quotient $G/H$ relative to the image generating set $\bar S$.
The following is the main theorem of this section.
    \begin{thm}\label{thm:growthcogrowthformula}
Suppose that a non-elementary group $G$ is either acylindrically hyperbolic or contains a Morse element and
satisfies the Morse local-global assumption. For every confined subgroup
$H<G$ with a non-degenerate confining subset,
\[
        \omega(H,S)+\frac{\omega(G/H,\bar S)}{2}\ge \omega(G,S).
\]
    \end{thm}
\begin{rem}\label{rem:AHMorse}
The classes of groups of type I and type II may overlap, but neither is
contained in the other. We illustrate these two non-inclusions with the
following examples.

First, \cite[Theorem 4.26(2)]{OOS} constructs groups containing Morse
elements that do not satisfy the Morse local-global assumption. More
precisely, every torsion-free infinite hyperbolic group admits a quotient
in which every proper subgroup is generated by a Morse element. Let $G$
be such a quotient. Then the free product $G\star G$ is hyperbolic
relative to its factors and is therefore acylindrically hyperbolic.
However, $G\star G$ does not satisfy the Morse local-global assumption.

Conversely, it was recently shown in \cite{AZ26} that there exists a
finitely generated, non-virtually cyclic Morse local-to-global group
containing an infinite-order Morse element that is not acylindrically
hyperbolic.
\end{rem}

\subsection{Extension Lemma adapted to confined subgroups}
The main tool is a sublinear extension lemma adapted to confined subgroups, formulated as a variant of \cite[Lemma 5.7]{CGYZ} for constructing admissible sequences in the present setting. The main novelty of this subsection is to establish such a result for groups satisfying the Morse local-global assumption. 
    
    \begin{lem}\label{lem:indepsetF}
    Assume that $G$ contains a Morse element and
satisfies the Morse local-global assumption. Let $P$ be a finite non-degenerate set in $G$. Then there exists an infinite set $F$ of pairwise independent Morse elements in $G$ such that for every $p\in P$ and $f\in F$, $pE(f)\ne E(f)$.
    \end{lem}
    \begin{proof}
    For every Morse element $h$, Lemma \ref{lem:E(G)} gives
    $E(G)=\bigcap_{g\in G}gE(h)g^{-1}$. Hence, for every $p\notin E(G)$ and every Morse element $h$, there exists $g\in G$ such that $p\notin gE(h)g^{-1}$. Starting with an infinite set of pairwise independent Morse elements given by Lemma \ref{lem:Morse-non-elementary} and conjugating its elements separately, we obtain an infinite set $F_p$ of pairwise independent Morse elements satisfying $pE(f)\ne E(f)$ for every $f\in F_p$. We shall combine these sets using the following two facts.

    Fix two independent Morse elements $h,k\in G$.

   \begin{claim}\label{clm:MorseElements}
    For all $m\gg 1$, the element $f_m:=h^mk^m$ is Morse. Moreover, if $p\notin E(h)$ and $q\notin E(k)$, then $p,q\notin E(f_m)$ for all sufficiently large $m$.
   \end{claim}
   \begin{proof}[Proof of the claim]
   The path
   \[
   \gamma_m=\bigcup_{n\in \mathbb Z}(f_m)^n[1,h^m]h^m[1,k^m]
   \]
   is a uniform Morse quasi-geodesic by Lemma \ref{lem:Morseadmissiblepath} when $m\gg1$. Suppose, to the contrary, that $p\in E(f_m)$ for infinitely many $m$. Then $p\gamma_m$ and $\gamma_m$ are uniform Morse quasi-geodesics with the same endpoints and hence have uniformly bounded Hausdorff distance. Since the system of translates of $E(h)$ and $E(k)$ has bounded intersection, the $pE(h)$-piece of $p\gamma_m$ must coincide, for all sufficiently large such $m$, with an $E(h)$-piece of $\gamma_m$. Thus there is an integer $n=n(m)$ such that
   \[
   pE(h)=(f_m)^nE(h).
   \]
   The integer $n$ is nonzero, since $p\notin E(h)$. Choose two such integers $m_1>m_2$ with $m_1-m_2\gg1$ and with $n(m_1),n(m_2)$ of the same sign. We obtain
   \[
   (f_{m_1})^{n(m_1)}E(h)=(f_{m_2})^{n(m_2)}E(h).
   \]
   Consequently, for some $h'\in E(h)$, the word
   \[
   W=(f_{m_1})^{-n(m_1)}(f_{m_2})^{n(m_2)}h'
   \]
   represents the identity. After combining adjacent syllables belonging to the same elementary subgroup, $W$ is an alternating word in $E(h)$ and $E(k)$ whose nontrivial interior syllables have lengths tending to infinity with $\min\{m_2,m_1-m_2\}$. The proof of Lemma \ref{lem:Morseadmissiblepath}, which is unchanged for signed powers, shows that $W$ labels a uniform Morse quasi-geodesic when $m_2,m_1-m_2\gg1$. This contradicts that $W$ is a loop. Hence $p\notin E(f_m)$ for all sufficiently large $m$. The argument for $q$ is symmetric. The same argument applies simultaneously to finite sets of elements outside $E(h)$ and $E(k)$.
   \end{proof}

   \begin{claim}\label{clm:indepset}
    Let $K\subseteq G$ be any finite set of pairwise independent Morse elements. Then $f_m=h^mk^m$ is independent from every element of $K$ for all sufficiently large $m$.
   \end{claim}
    \begin{proof}[Proof of the claim]
    Otherwise, after passing to a subsequence, there is a fixed $f\in K$ such that $f_m$ is not independent from $f$. Thus a translate of $\ax(f_m)$ has finite Hausdorff distance from $\ax(f)$. Since the axes $\ax(f_m)$ are uniformly Morse, this Hausdorff distance is bounded independently of $m$. Each $\ax(f_m)$ contains translated subpaths of $\ax(h)$ and $\ax(k)$ whose lengths tend to infinity with $m$. The bounded-intersection characterization of independence therefore implies that $f$ is not independent from either $h$ or $k$. It follows that $h$ and $k$ have conjugate nonzero powers, contradicting their independence.
    \end{proof}

   The conclusion now follows by induction on $|P|$, the case $|P|=1$ being the construction of $F_p$ above. Suppose that an infinite independent set $F_Q$ has been constructed for a nonempty subset $Q\subset P$, and let $p\in P\setminus Q$. Choose recursively $h_j\in F_Q$ and $k_j\in F_p$ so that $h_j$ and $k_j$ are independent and their commensurability classes have not been used previously. By Claim \ref{clm:MorseElements}, for a sufficiently large $m_j$ the element $h_j^{m_j}k_j^{m_j}$ avoids every element of $Q\cup\{p\}$. By Claim \ref{clm:indepset}, we may also choose $m_j$ so that this element is independent from all the elements constructed at earlier stages. This produces the required infinite set for $Q\cup\{p\}$ and completes the induction.
    \end{proof}
    \begin{lem}\label{lem:indepsetF2}
    Assume that $G$ is acylindrically hyperbolic. Let $P$ be a finite non-degenerate set in $G$. Then there exists an infinite set $F$ of pairwise independent Morse elements in $G$ such that for every $p\in P$ and $f\in F$, $pE(f)\ne E(f)$.
    \end{lem}
   \begin{proof}
Fix a non-elementary acylindrical action of $G$ on a hyperbolic space $X$.
The preceding proof applies with Morse elements replaced by loxodromic WPD
elements for this action. Indeed, the standard WPD ping-pong argument for two
independent loxodromic elements, as in \cite[Section 6]{DGO}, shows that
$h^m k^m$ is loxodromic WPD for all sufficiently large $m$; its axis is the
alternating union, up to a uniform error, of long translated subpaths of axes
of $h$ and $k$. The geometrical separation of WPD axes, given by
Lemma \ref{lem:bddintersectionamples}, proves the two claims above without
using the Morse local-global assumption. Finally, loxodromic WPD elements are
Morse in the word metric by \cite{Sisto16}. The same induction therefore gives
the required set $F$.
\end{proof}
    
   The following output is the main result of this subsection.
    \begin{lem}\label{ExtensionLemConfined}
	Suppose that a non-elementary group $G$ is either type I or type II. Let $P$ be a finite non-degenerate subset of $G$. Then there exist a set $F$ of three independent Morse elements and a constant $D>0$ with the following property.

For any $g,h\in G$ and any geodesics $\alpha_1=[1,g]_S$ and $\alpha_3=[1,g^{-1}h]_S$, there exists $f\in F$ such that, for every $p\in P$ and every geodesic $\alpha_2=[1,p]_S$, the three geodesics satisfy the conditions of Definition \ref{defn:admissibletuples} relative to $(E(f),E(f))$. In particular, $(g,p,g^{-1}h)$ is $D$-admissible relative to the system associated with $F$.
	\end{lem}

    \begin{proof}
	Let $F$ be a set of three independent Morse elements provided by Lemma \ref{lem:indepsetF} or Lemma \ref{lem:indepsetF2}, so that $pE(f)\ne E(f)$ for all $p\in P$ and $f\in F$. Let $D_0$ be given by Lemma \ref{BddprojViaMorse}, and choose
    \[
    D\ge \max\bigl\{D_0,\|p\|_S:p\in P\bigr\}.
    \]
    Apply Lemma \ref{BddprojViaMorse} to the translate of $\alpha_1$ from $1$ to $g^{-1}$ and to $\alpha_3$. For each of these two geodesics at least two elements of $F$ satisfy the required bound. Since $|F|=3$, the two choices have a common element $f$, and
    \[
    \mathrm{diam}_S\bigl(\alpha_1\cap \mathcal N_C(gE(f))\bigr),\quad
    \mathrm{diam}_S\bigl(\alpha_3\cap \mathcal N_C(E(f))\bigr)\le D.
    \]
    For every $p\in P$, both intersections of $\alpha_2$ with $\mathcal N_C(E(f))$ and with $\mathcal N_C(pE(f))$ have diameter at most $\ell(\alpha_2)=\|p\|_S\le D$. Finally, $E(f)\ne pE(f)$, so all the conditions of Definition \ref{defn:admissibletuples} hold.
	\end{proof}

\subsection{The cones and balls are comparable}
This subsection provides the main geometric ingredient in the proof. Although it may not be immediately apparent from its technical formulation, the intuition and motivation behind the key result come from Roblin's Shadow Principle for Patterson--Sullivan measures. In this sense, Lemma \ref{lem:shadowBaby} may be viewed as a counting-theoretic reverse engineering of the Shadow Principle. Its proof adapts the strategy of \cite{CGYZ}, using the sublinear Extension Lemma.

The \textit{critical exponent} of a subset $A \subseteq G$, denoted by $\e A$, is intimately related to the (partial) Poincar\'e series 
\begin{equation}\label{PoincareEQ}
s\ge 0, \quad \p_A(s) = \sum_{g\in A} \mathrm{e}^{-sd_S(1,g)}  
\end{equation}
as $\p_{A}(s)$ diverges for $s<\e A$ and converges for $s>\e A$.

Given $x\in G$ and $r\ge 0$, we define the cone by
$$\begin{aligned}
\Omega(x,r) &:= \{y\in G: [1,y]_S\cap B(x,r)\ne\emptyset
\text{ for every geodesic }[1,y]_S\}.
\end{aligned}
$$
where $B(x,r) := \{y\in G: d_S(x,y)\le r\}.$





\begin{lem}\label{lem:shadowBaby}
There exists a sublinear function $\kappa$ with the following property.
Fix $g \in G$ and any integer $n\ge 0$. Then 
\[
\bigl|H \cap S_n(g)\bigr| \le {3}\, \bigl |H \cap S_n(g, R_{g,n}) \cap \Omega(g,r_g)\bigr|
\]
where $R_{g,n}=\kappa(n+\|g\|_S)$ and $r_g=\kappa(\|g\|_S)$.
\end{lem}

\begin{proof} 
Write $m=\|g\|_S$, $N=m+n$, and $P_0=\max\{\|p\|_S:p\in P\}$. It suffices to consider sufficiently large $N$. Indeed, there are only finitely many remaining pairs $(g,n)$, and they are absorbed by increasing the final function $\kappa$ on a bounded interval so that $1\in B(g,\kappa(m))$.

\medskip
\noindent\textbf{Choice of $\hat f_h$.} Increase $\kmorse$ by an additive constant so that, for every sufficiently large $N$ and every $f\in F$, the first positive power $\hat f(f,N)$ satisfying $\|\hat f(f,N)\|_S>10\kmorse(N)$ also satisfies
\[
\|\hat f(f,N)\|_S\le20\kmorse(N).
\]
This is possible because $F$ is finite and $\langle f\rangle$ is undistorted. Since $\kmorse$ is sublinear, after increasing the lower bound on $N$ we have
\[
20\kmorse(N)\le M:=\max\{m,n,P_0\}.
\]
Moreover, $N\ge M$ and hence $\kmorse(N)\ge\kmorse(M)$. Thus $\hat f(f,N)$ and its inverse satisfy the length conditions in Lemma \ref{ExtensionViaMorse2}.

Given $h\in H\cap S_n(g)$, fix geodesics $[1,g]_S$ and $[1,g^{-1}h]_S$. By Lemma \ref{ExtensionLemConfined}, choose $f_h\in F$ such that $(g,p,g^{-1}h)$ satisfies the required $D$--admissibility conditions relative to $(E(f_h),E(f_h))$ for every $p\in P$, and put $\hat f_h=\hat f(f_h,N)$. In particular, the choice of $\hat f_h$ depends only on $f_h$ and $N$.

\medskip
\noindent\textbf{Choice of $p_h$.} Fix an ordering of the finite set $P$ a priori. Choose the first $p_h\in P$ in this order such that
\[
g\hat f_h p_h\hat f_h^{-1}g^{-1}\in H,
\]
which is ensured by the definition of a confined subgroup. Define
\[
\widehat R(N)=\max\{2\|\hat f(f,N)\|_S+\|p\|_S:f\in F,\ p\in P\}.
\]
The function $\widehat R$ is bounded above by $40\kmorse(N)+P_0$ and therefore admits a sublinear upper bound. Choose one sublinear function $\kappa$ which dominates $3\kmorse$ and $\widehat R$, as well as the finitely many exceptional values above, and set
\[
r_g=\kappa(m),\qquad R_{g,n}=\kappa(N).
\]
\begin{rem}\label{rem:fhphChoice}
Once $\hat f_h$ and $g$ are chosen, $p_h$ is determined by the fixed order on $P$.
\end{rem}

Denote $\mathbf{h}=g\hat f_hp_h\hat f_h^{-1}g^{-1}h$. Since $g\hat f_hp_h\hat f_h^{-1}g^{-1}\in H$ and $h\in H$, we have $\mathbf h\in H$. By Lemma \ref{ExtensionViaMorse2}, the path labeled by $(g,\hat f_h,p_h,\hat f_h^{-1},g^{-1}h)$ sublinearly fellow travels every geodesic $\gamma=[1,\mathbf h]_S$. In particular, $\gamma$ contains a point $z_1$ satisfying
\[
d_S(g,z_1)\le3\kmorse(m)\le r_g.
\]
Moreover,
\[
\bigl|d_S(g,\mathbf h)-n\bigr|
\le2\|\hat f_h\|_S+\|p_h\|_S\le R_{g,n}.
\]
Consequently, $\mathbf h\in\Omega(g,r_g)\cap S_n(g,R_{g,n})$, and we have constructed the map
$$
\begin{aligned}
\Psi_g:H \cap S_n(g)&\longrightarrow H \cap S_n(g, R_{g,n}) \cap \Omega(g,r_g) \\
 h &\longmapsto \mathbf{h}
\end{aligned}
$$
If $\Psi_g(h)=\mathbf h$, then
\[
h=g\hat f_hp_h^{-1}\hat f_h^{-1}g^{-1}\mathbf h.
\]
Once $f_h$ is fixed, both $\hat f_h$ and $p_h$ are fixed. Since $|F|=3$, it follows that $|\Psi_g^{-1}(\mathbf h)|\le3$.
\end{proof}

\subsection{Proof of Theorem \ref{thm:growthcogrowthformula}}\label{secinequality}

The overall strategy for proving the growth--cogrowth inequality goes back to Coulon \cite{Coulon22}, who established a similar inequality for groups containing strongly contracting elements. Here we work in the broader setting of groups of types I and II, which need not contain strongly contracting elements. This requires the development of the sublinear Extension Lemma, as well as the more delicate estimates in Lemma \ref{lem:phis_finite}.

Let $H$ be a non-trivial confined subgroup of $G$ with a non-degenerate confining subset $P$.
Consider the Hilbert space $\ell^2(G/H)$ over the coset space $G/H=\{Hg: g\in G\}$.
For each $t>0$, define 
$$
\begin{aligned}
\varphi_t:\quad G/H\quad\longrightarrow&\quad\mathbb R\\
  Hg\quad\longmapsto   &\quad \mathrm{e}^{-td_S(H, Hg)}  
\end{aligned}
$$
Note that $\e{G/H}$ is the convergence radius of the series
$$
\sum_{Hg\in G/H}\mathrm{e}^{-t d_S(H, Hg)}
$$
Thus, if $t>\frac{\e{G/H}}{2}$ then $\varphi_t\in \ell^2(G/H)$; if $t<\frac{\e{G/H}}{2}$ then $\varphi_t\notin \ell^2(G/H)$.

\medskip
\noindent
For each $s>0$, define 
$$
\begin{aligned}
\phi_s:\quad G/H\quad\longrightarrow &\quad \mathbb R\\
  Hg\quad\longmapsto  &\quad \sum_{h\in H} \mathrm{e}^{-sd_S(1, hg)}  
\end{aligned}
$$
The crucial ingredient is  the following, which is now possible thanks to Lemma \ref{lem:shadowBaby}. 
\begin{lem}\label{lem:phis_finite}
For any $s>\e H$, we have $\phi_s\in \ell^2(G/H)$.    
\end{lem}

\begin{proof}
We first reformulate the square of the $\ell^2$-norm  $\|\phi_s\|$ as follows 
\begin{equation}\label{eq:phi_s_norm}
\begin{aligned}
\|\phi_s\|^2\quad &=\quad\sum_{\underset{Hg_1=Hg_2}{(g_1,g_2)\in G\times G}} \mathrm{e}^{-s(\|g_1\|_S+\|g_2\|_S)}\\
&=\quad\sum_{g\in G}  \mathrm{e}^{-s\|g\|_S} \left(\sum_{h\in H} \mathrm{e}^{-sd_S(g, h)}\right)\\
&=\quad\sum_{m= 0}^\infty \mathrm{e}^{-s m} \left(\sum_{g\in S_m}\, \sum_{h \in H} \mathrm{e}^{-s d_S(g, h)} \right)\\
&=\quad\sum_{m= 0}^\infty \mathrm{e}^{-s m}  \sum_{g\in S_m}\sum_{n =0}^{\infty} \mathrm{e}^{-s n} \, \big| H \cap S_n(g) \big|
\end{aligned}  
\end{equation}
where $S_m$ denotes the $m$-sphere. Denote $R_{m,n}=\kappa(n+m)$ and $r_m=\kappa(m)$, where $\kappa$ is given by Lemma \ref{lem:shadowBaby}. For clarity, let us write $A_g= S_n(g, R_{m,n}) \cap \Omega(g,r_m)$, 
\begin{equation}\label{eq:A_g_spheresum}
\begin{aligned}
  \quad\sum_{g\in S_m} \big| H \cap S_n(g) \big| \le &   \sum_{g\in S_m}   3 \, \big|H \cap A_g \big|
\end{aligned}
\end{equation}
Therefore, the next step is to give an upper bound on the RHS of Eq. (\ref{eq:A_g_spheresum}). 
\medskip

For each $\mathbf h\in H$, fix a geodesic
$\gamma_{\mathbf h}=[1,\mathbf h]_S$. If  $\mathbf h$ is an element in $A_g$, then
$d_S(g,\gamma_{\mathbf h})\le r_m.$
Choose $v\in\gamma_{\mathbf h}$ with $d_S(g,v)\le r_m$. Since
$|g|_S=m$, we have $\bigl||v|_S-m\bigr|\le r_m.$
Hence there are at most $2r_m+1$ possible choices for $v$, and for
each such $v$ there are at most $|B(1,r_m)|$ possible choices for $g$.
It follows that
\[
\#\{g\in S_m:\mathbf h\in A_g\}
\le (2r_m+1)|B(1,r_m)|.
\]
Therefore,
\begin{align*}
\sum_{g\in S_m} |H\cap A_g|
&=\sum_{\mathbf h\in H}
  |\{g\in S_m:\mathbf h\in A_g\}|\\
&\le (2r_m+1)|B(1,r_m)|
  \Bigl|\bigcup_{g\in S_m}(H\cap A_g)\Bigr|\\
&\le (2r_m+1)|B(1,r_m)|
  |H\cap S_{m+n}(1,R_{m,n}+2r_m)|.
\end{align*}
Indeed, if $v\in\gamma_{\mathbf h}$ satisfies $d_S(g,v)\le r_m$, then
\[
\bigl|\|\mathbf h\|_S-(m+n)\bigr|
\le \bigl|\|v\|_S-m\bigr|
 +\bigl|d_S(v,\mathbf h)-d_S(g,\mathbf h)\bigr|
 +\bigl|d_S(g,\mathbf h)-n\bigr|
\le R_{m,n}+2r_m.
\]
Set $\rho_{m,n}:=R_{m,n}+2r_m.$
Since
\[
H\cap S_{m+n}(1,\rho_{m,n})
\quad \subseteq \quad
\bigcup_{\lvert l-(m+n)\rvert\le \rho_{m,n}}
(H\cap S_l),
\]
we have, for every $s>0$,
\begin{align*}
 \mathrm{e}^{-s(m+n)}
\left|H\cap S_{m+n}(1,\rho_{m,n})\right|
&\le
\sum_{\lvert l-(m+n)\rvert\le \rho_{m,n}}
\mathrm{e}^{-s(m+n)}|H\cap S_l|\\
&\le
\mathrm{e}^{s\rho_{m,n}}
\sum_{\lvert l-(m+n)\rvert\le \rho_{m,n}}
\mathrm{e}^{-sl}|H\cap S_l|.
\end{align*}
Since $S$ is finite, there exists $\omega>0$ such that
$|B(1,r)|\le e^{\omega r}$ for every $r\ge0$. It follows that
\begin{equation}\label{eq:A_g_spheresum2}
\begin{aligned}
\mathrm{e}^{-s(m+n)}
3\sum_{g\in S_m}|H\cap A_g|
&\le
3(2r_m+1)|B(1,r_m)|
\mathrm{e}^{s(R_{m,n}+2r_m)}
\sum_{\lvert l-(m+n)\rvert\le \rho_{m,n}}
\mathrm{e}^{-sl}|H\cap S_l|\\
&\le C_{m,n}\; \sum_{\lvert l-(m+n)\rvert\le \rho_{m,n}}
\mathrm{e}^{-sl}|H\cap S_l|.
\end{aligned}
\end{equation}
where  
\[
C_{m,n}
:=
3(2r_m+1)\mathrm e^{\omega r_m}
\mathrm e^{s(R_{m,n}+2r_m)} =\exp(o(m+n))
\]
is a subexponential function. 
\medskip

Finally, plugging (\ref{eq:A_g_spheresum}) and (\ref{eq:A_g_spheresum2}) into Eq. (\ref{eq:phi_s_norm}) gives
\[
\begin{aligned}
\|\phi_s\|^2
&\le
\sum_{m,n\ge 0}
\mathrm e^{-s(m+n)}
\sum_{g\in S_m}3|H\cap A_g|\\
&\le
\sum_{m,n\ge0}
C_{m,n}
\sum_{\lvert l-(m+n)\rvert\le \rho_{m,n}}
\mathrm e^{-sl}|H\cap S_l|.
\end{aligned}
\tag{\(\ast\)}
\]
The estimate \((\ast)\) is finite for $s>\omega(H)$, since multiplying the coefficients of a Poincaré series by any subexponential sequence does not change its critical exponent. 
Indeed, writing \(k=m+n\), there are exactly \(k+1\) pairs
\((m,n)\) with \(m+n=k\). Since \(r_m\le \kappa(k)\) and
\(\rho_{m,n}\le 3\kappa(k)\), the estimate \((\ast)\) gives
\[
\|\phi_s\|^2
\le
\sum_{k=0}^{\infty}
(k+1)\mathrm e^{o(k)}
\sum_{|l-k|\le 3\kappa(k)}
\mathrm e^{-sl}|H\cap S_l|.
\]
Interchanging the nonnegative sums, we obtain
\[
\|\phi_s\|^2
\le
\sum_{l=0}^{\infty}
D_l\,\mathrm e^{-sl}|H\cap S_l|,
\]
where
\[
D_l:=
\sum_{\substack{k\ge0\\|l-k|\le3\kappa(k)}}
(k+1)\mathrm e^{o(k)}
\]
is a sub-exponential function.
Hence \(\phi_s\in\ell^2(G/H)\).
\end{proof}

We now conclude the proof of Theorem \ref{thm:growthcogrowthformula}.

Given any $s>\e H$ and $t>\e {G/H}/2$, we shall prove  $s+t\ge \e G$.   Recall that $$\p_G(s+t)=\sum_{g\in G} \mathrm{e}^{-(s+t)\|g\|_S}.$$  Grouping the elements of $G$ in the same $Hg$ and noting $d_S(H, Hg)\le d_S(1,Hg)$, we have  $$\p_G(s+t)\le \sum_{Hg\in G/H} \left(\mathrm{e}^{-td_S(H,Hg)} \sum_{h\in H} \mathrm{e}^{-sd_S(1,hg)}\right)$$ 
The Cauchy--Schwarz inequality gives the finiteness of the scalar product of $\phi_s$ and $\varphi_t$:
$$
\begin{aligned}
(\phi_s, \varphi_t) &=\sum_{Hg\in G/H}  \left(\mathrm{e}^{-td_S(H, Hg)} \sum_{h\in H} \mathrm{e}^{-sd_S(1, hg)}\right)\\
&\le \|\phi_s\|\|\varphi_t\| <\infty  
\end{aligned}
$$
Hence, $\p_G(s+t)<\infty$ for any $s>\e H$ and $t>\e {G/H}/2$. By definition of the critical exponent, $s+t\ge \e G$ and thus $\e H+\e {G/H}/2 \ge \e G$.
Theorem \ref{thm:growthcogrowthformula} is proved.



\bibliographystyle{alpha}
 \bibliography{bibliography}

\end{document}